\documentclass[UTF-8,reqno]{amsart}
\usepackage{enumerate}
\usepackage{amssymb,url,color, booktabs}
\usepackage{mathrsfs}
\usepackage{color}
\usepackage[colorlinks=true]{hyperref}
\hypersetup{
    linkcolor=blue,          % color of internal links
    citecolor=red,        % color of links to bibliography
    filecolor=blue,      % color of file links
    urlcolor=cyan
}

\numberwithin{equation}{section}

\newcommand{\be}{\begin{eqnarray}}
\newcommand{\ee}{\end{eqnarray}}
\newcommand{\ce}{\begin{eqnarray*}}
\newcommand{\de}{\end{eqnarray*}}
\newtheorem{theorem}{Theorem}[section]
\newtheorem{lemma}[theorem]{Lemma}
\newtheorem{remark}[theorem]{Remark}
\newtheorem{definition}[theorem]{Definition}
\newtheorem{proposition}[theorem]{Proposition}
\newtheorem{Examples}[theorem]{Example}
\newtheorem{corollary}[theorem]{Corollary}

\def\nor{|\mspace{-3mu}|\mspace{-3mu}|}

\def\eps{\varepsilon}

\def\p{\partial}

\def\[{{\Big[}}
\def\]{{\Big]}}
\def\<{{\langle}}
\def\>{{\rangle}}
\def\({{\big(}}
\def\){{\big)}}

\def\bx{{\mathbf{x}}}

\def\sgn{\mbox{\rm sgn}}

\def\dif{{\mathord{{\rm d}}}}

\def\bb2{{\boldsymbol{2}}}
\def\no{\nonumber}
\def\={&\!\!=\!\!&}

\def\bC{{\mathbf C}}

\def\cB{{\mathcal B}}

\def\cE{{\mathcal E}}
\def\cF{{\mathcal F}}

\def\cH{{\mathcal H}}

\def\cL{{\mathcal L}}
\def\cM{{\mathcal M}}

\def\cP{{\mathcal P}}

\def\mE{{\mathbb E}}

\def\mI{{\mathbb I}}

\def\mK{{\mathbb K}}
\def\mH{{\mathbb H}}
\def\mL{{\mathbb L}}

\def\mN{{\mathbb N}}

\def\mP{{\mathbb P}}
\def\mQ{{\mathbb Q}}
\def\mR{{\mathbb R}}

\def\mZ{{\mathbb Z}}

\def\bP{{\mathbf P}}

\def\1{{\mathbf{1}}}

\def\sA{{\mathscr A}}
\def\sB{{\mathscr B}}

\def\sF{{\mathscr F}}
\def\sG{{\mathscr G}}

\def\sI{{\mathscr I}}
\def\sJ{{\mathscr J}}

\def\sL{{\mathscr L}}

\def\geq{\geqslant}
\def\leq{\leqslant}
\def\ge{\geqslant}
\def\le{\leqslant}

\def\eps{\varepsilon}

\def\p{\partial}

\def\[{{\Big[}}
\def\]{{\Big]}}
\def\<{{\langle}}
\def\>{{\rangle}}

\def\bx{{\mathbf{x}}}

\def\sgn{\mbox{\rm sgn}}

\def\dif{{\mathord{{\rm d}}}}

\def\no{\nonumber}
\def\={&\!\!=\!\!&}
\def\bt{\begin{theorem}}
\def\et{\end{theorem}}
\def\bl{\begin{lemma}}
\def\el{\end{lemma}}
\def\br{\begin{remark}}
\def\er{\end{remark}}
\def\bx{\begin{Examples}}
\def\ex{\end{Examples}}
\def\bd{\begin{definition}}
\def\ed{\end{definition}}
\def\bp{\begin{proposition}}
\def\ep{\end{proposition}}
\def\bc{\begin{corollary}}
\def\ec{\end{corollary}}

\def\geq{\geqslant}
\def\leq{\leqslant}
\def\ge{\geqslant}
\def\le{\leqslant}

\def\bP{{\mathbf P}}

 \def\R{\mathbb R}
 \def\R{\mathbb R}    
\def\N{\mathbb N}  
   
\def\<{\langle} \def\>{\rangle}

\def\X{{\widetilde X}}

\def\X{{\bf X}}

\allowdisplaybreaks

\begin{document}

\title[Averaging principle of DDSDE with singular drift]
{Strong and weak convergence for averaging principle of DDSDE with singular drift}

\author{Mengyu Cheng, Zimo Hao, Michael R\"ockner}

\date{}%\today}

\address{Mengyu Cheng:
Fakult\"at f\"ur Mathematik, Universit\"at Bielefeld, 33615 Bielefeld, Germany;
School of Mathematical Sciences, Dalian University of Technology, Dalian
116024, P. R. China
\newline
Email: mengyucheng@mail.dlut.edu.cn; mengyu.cheng@hotmail.com
 }

\address{Zimo Hao:
Fakult\"at f\"ur Mathematik, Universit\"at Bielefeld,
33615, Bielefeld, Germany;
School of Mathematics and Statistics, Wuhan University,
Wuhan, Hubei 430072, P. R. China
	 \newline
Email: zimohao@whu.edu.cn; zhao@math.uni-bielefeld.de}

\address{Michael R\"ockner:
Fakult\"at f\"ur Mathematik, Universit\"at Bielefeld,
33615, Bielefeld, Germany; Academy of Mathematics and System Sciences, Chinese Academy of Science, Beijing 100190, P. R. China
\newline
Email: roeckner@math.uni-bielefeld.de
 }


\thanks{
This work is supported by the China Scholarship Council (CSC 202006060083), NNSFC grant of China (No. 12131019, 11731009) and the German Research Foundation (DFG) through the Collaborative Research Centre (CRC) 1283/2 2021 - 317210226
``Taming uncertainty and profiting from randomness and low regularity in analysis, stochastics and their applications''. }

%\address{
%zhao@math.uni-bielefeld.de
% }

%\thanks{
%This work is partially supported by NNSFC grants of China (Nos. 12131019, 11731009), and the German Research Foundation (DFG) through the Collaborative Research Centre(CRC) 1283 ``Taming uncertainty and profiting from randomness and low regularity in analysis, stochastics and their applications".}

\maketitle \rm

%\pagecolor{MyDarkBlue}\color{yellow}

\section*{Abstract}
In this paper, we study the averaging principle for distribution dependent stochastic differential equations
with drift in localized $L^p$ spaces. Using Zvonkin's transformation and estimates for solutions to Kolmogorov equations,
we prove that the solutions of the original system strongly and weakly converge to the solution of the
averaged system as the time scale $\eps$ goes to zero. {Moreover, we obtain rates of the strong and
weak convergence that depend on $p$ respectively.}

\tableofcontents

\section{Introduction}
Consider the following distribution dependent stochastic differential equation (in short, DDSDE):
\begin{align}\label{in:NDDSDE}
\dif X_t=b(t,X_t,\mu_t)\dif t+\Sigma(t,X_t,\mu_t)\dif W_t,
\end{align}
where $b: \mR_+\times\mR^d\times\cP(\mR^d)\to\mR^d$ and $\Sigma:\mR_+\times\mR^d\times\cP(\mR^d)\to\mR^d\otimes\mR^d$ are measurable functions
and $\mu_t:=\cL(X_t)$ is the time marginal law of $X_t$.
Here $\cP(\mR^d)$ denotes the space of all probability measures over $\mR^d$
and $(W_t)_{t\geq 0}$ is a $d$-dimensional standard Brownian motion on some stochastic basis
$(\Omega,\sF,\bP; (\sF_t)_{t\geq 0})$. The {Kolmogorov operator} of \eqref{in:NDDSDE} is given by
\begin{equation*}
\sL_t\varphi(x)=A_{ij}(t,x,\mu_t)\p_i\p_j\varphi(x)+b(t,x,\mu_t)\cdot\nabla\varphi(x),
\end{equation*}
where $A_{ij}(t,x,\mu_t):=\frac{1}{2}\Sigma_{ik}\Sigma_{jk}(t,x,\mu_t)$.
Here and below we use the usual Einstein
convention for summation.

Note that DDSDE \eqref{in:NDDSDE} is also called mean-field SDE or McKean-Vlasov SDE in the literature if
\begin{equation}\label{MVb}
b(t,x,\mu):=\int_{\R^d}\hat{b}(t,x,y)\mu(\dif y),\quad \Sigma(t,x,\mu):=\int_{\R^d}\hat{\Sigma}(t,x,y)\mu(\dif y)
\end{equation}
for all $(t,x,\mu)\in\R_+\times\R^d\times\cP(\R^d)$, where
$\hat{b}:\mR_+\times\mR^{2d}\to\mR^d$ and $\hat{\Sigma}:\mR_+\times\mR^{2d}\to\mR^d\otimes\mR^d$ are measurable functions.
It naturally appears in the studies on the limiting behavior of interacting particle systems and mean-field games. Roughly speaking,
DDSDE (or McKean-Vlasov SDE) describe stochastic systems whose evolution is determined not only by the microcosmic location of the particle, but also by the macrocosmic distribution. So it has attracted wide attention (see \cite{CD13,CD18b,CGPS20,HRZ2021,Mc,Sz} and references therein).

When $b$ and $\Sigma$ satisfy some continuity assumptions, there are numerous results devoted
to studying the existence and uniqueness of this type of DDSDE \eqref{in:NDDSDE}
(see \cite{HSS18, Wang2018} for examples). For the case that $b$ is only measurable,
of at most linear growth and Lipschitz continuous with respect to $\mu$,
when $\Sigma$ does not depend on $\mu$, is uniformly non-degenerate and Lipschitz continuous,
by using the classical Krylov estimates, Mishura and Veretenikov \cite{MV2020} obtained
the strong well-posedness of \eqref{in:NDDSDE}. After that, the strong well-posedness was extended
to local $L^q_tL^p_x$ drift by R\"ockner and Zhang in \cite{RZ21}.
Moreover, by the relative entropy method and Girsanov's theorem, Lacker \cite{Lack2021} also
obtained some well-posedness results for linear growth cases (see also \cite{La18}).
{Then, in the special case \eqref{MVb}, Han obtained the well-posedness for $L^q_tL^p_x$ drift based on
the relative entropy method in \cite{Ha22}.}
In \cite{Zhao2020}, by some heat kernel estimates and Schauder-Tychonoff fixed point theorem,
Zhao established the well-posedness results for DDSDE in a more general case \eqref{in:NDDSDE}
when $b$ satisfies some $L^q_tL^p_x$ condition and the derivatives of $\Sigma$ with respect to
$\mu$ are H\"older continuous. It should be noted that by Zvonkin's transformation of
\cite{Zv} (see also \cite{Xi-Xi-Zh-Zh} and \cite{Zh-Zh1}) and entropy method,
the last two authors together with Zhang in \cite{HRZ22} obtained the strong well-posedness
for \eqref{in:NDDSDE}, where $\Sigma$ does not dependent on $\mu$ and $b$ is in mixed $L^q_tL^p_x$ space. For existence and uniqueness results on weak solutions to Nemytskii-type (=density dependent) DDSDE with merely measurable coefficients we refer to \cite{BR2018,BR2020, BR2022A, BR2022B}. 

In this work, we are interested in investigating the averaging principle of the following
\iffalse
slow-fast system:
\begin{align*}
\begin{cases}
\dif X_t^\eps=b(Y^\eps_t,X^\eps_t,\mu^\eps_t)\dif t+\sigma(X^\eps_t)\dif W_t,\quad X^\eps_0=\xi\in\sF_0,\\
\dif Y_t^\eps=\frac{1}{\eps}\dif t,\quad Y^\eps_0=0,
\end{cases}
\end{align*}
\fi
DDSDE with {highly oscillating time component}
\begin{align}\label{MV1}
\dif X_t^\eps=b\left(\frac{t}{\eps},X_t^\eps,\mu_t^\eps\right)\dif t+\sigma(X_t^\eps)\dif W_t, \ \ X_0^\eps=\xi\in\sF_0,
\end{align}
where $\sigma:\mR^d\to\mR^d\otimes\mR^d$ is a measurable function, $\mu^\eps_t:=\cL(X^\eps_t)$ is
the time marginal law of $X^\eps_t$ and the time scale $0<\eps \ll1$.

Usually, solving the original system \eqref{MV1} is
relatively difficult because of the {highly oscillating time component}. Therefore, it is desirable to find a simplified system which simulates and predicts the evolution of the original system over a long time scale.
As is well known, the {highly oscillating time component} can be ``averaged'' out to produce such a simplified system under some suitable conditions, which is called
averaging principle.

More exactly, let
\begin{align}\label{in:KBM}
\bar{b}(x,\mu):=\lim\limits_{T\to\infty}\frac{1}{T}\int_0^T b(t,x,\mu)\dif t.
\end{align}
If $b$ is independent of $\mu$, $b$ is called
a {\em KBM-vector feld} (KBM stands for Krylov, Bogolyubov and Mitropolsky)
if the convergence \eqref{in:KBM} is uniformly with respect to $x$ on
bounded subsets of $\R^d$ (see e.g. \cite{SV1985}).
The averaging principle states, as the time scale $\eps$ goes to zero, that the solution of the original equation \eqref{MV1} converges
to that of the following averaged equation on any finite time interval
\begin{align}\label{MV2}
\dif X_t=\bar{b}(X_t,\mu_t)\dif t+\sigma(X_t)\dif W_t,\ \ X_0=\xi,
\end{align}
where $\mu_t$ stands for the distribution of $X_t$.

The averaging principle was firstly established for deterministic systems by Krylov, Bogolyubov and
Mitropolsky \cite{BM1961,KB1943}. Then it was extended to stochastic differential equations by
Khasminskii \cite{Khas1968}. After that,
extensive work on the averaging principle for finite and infinite dimensional stochastic
differential equations was done; see e.g.
\cite{BK2004, Cerr2009, CF2009, CL2017, CL2021, DW2014, FW2012,
ML2020, Kif2004, MSV1991, PIX2021, Vere1999,
WR2012} and the references therein.

Recall that the strong convergence rate of the averaging principle for slow-fast McKean-Vlasov SDE
%with global Lipschitz coefficients
was established by the techniques of time discretization and Poisson equation in \cite{RSX2021}.
Furthermore, as discussed in \cite{HLL2022}, the strong convergence rate of the averaging principle for slow-fast McKean-Vlasov SPDE was studied, based on the variational approach and the technique of time discretization.
Note that the coefficients of the slow equation with fast variables were globally Lipschitz continuous
with respect to the slow variable in the above results.
Recently, the strong convergence without a rate for DDSDE with {highly oscillating time component}
driven by fractional Brownian motion and standard Brownian motion was shown in \cite{SXW2022},
which requires that the drift term is continuous in the slow variable.

Obviously, Lipschitz or continuity assumptions on $b$ are too strong for some applications.
There are a lot of interesting models from physics, only having bounded measurable or even singular $L^p$ interaction kernels $b$ and $\hat{b}$. For example, the rank-based interaction studied in \cite{La18} has an indicator kernel, which is discontinuous; the Biot-Savart law appearing in the vortex description of 2-dimensional incompressible Navier-Stokes equations has a singular kernel like $x^\perp/|x|^2$ (see e.g. \cite{Zh2020}). However, to the best of our knowledge, there is no result concerning the averaging principle both of DDSDE and SDE with $L^p$ drift.

Following the above motivations, we consider the strong and weak convergence of the averaging principle
for DDSDE with $L^p$ drift in the present paper. Moreover, we obtain the rate of the strong and weak convergence,
which is important for functional limit theorems in probability and homogenization in PDEs.
Throughout this paper we need the following conditions.

\begin{enumerate}[{\bf (H$_b^1$)}]
\item
Let $p_0\in(d\vee2,\infty)$ and assume that there is a nonnegative constant $\kappa_0$ such that for all $t\ge0$
and $\mu,\nu\in\cP(\R^d)$
$$
\nor b(t,\cdot,\mu)\nor_{p_0}+
\frac{\nor b(t,\cdot,\mu)-b(t,\cdot,\nu)\nor_{p_0}}{\|\mu-\nu\|_{var}}
\leq\kappa_0,$$
where
$\|\mu-\nu\|_{var}:=\sup\limits_{A\in\sB(\mR^d)}|\mu(A)-\nu(A)|$ is total variation.\end{enumerate}

\begin{enumerate}[{\bf (H$_b^2$)}]
\item
There are functions
$\bar{b}:\mR^{d}\times\cP(\mR^d)\to\mR^{d}$,
$\omega:\mR_+\to\mR_+$ and $H:\mR^{d}\times\cP(\mR^d)\to\mR_+$ such that for all $(T,t,x,\mu)\in\mR_+^2\times\mR^{d}\times\cP(\mR^d)$
\begin{align}\label{in:00}
\left|\frac{1}{T}\int_{t}^{T+t}(b(s,x,\mu)-\bar{b}(x,\mu))\dif s\right|\le \omega(T)H(x,\mu),
\end{align}
where $\lim\limits_{t\to\infty}\omega(t)=0$ and $\sup\limits_\mu\nor H(\cdot,\mu)\nor_{p_0}<\kappa_0$.
Here $p_0$ and $\kappa_0$ are as in {\bf(H$_b^1$)}.

\end{enumerate}
\begin{enumerate}[{\bf (H$_\sigma$)}]
\item
There are constants $p>d\vee 2$, $\kappa_1>1$ and $\beta\in(0,1)$ such that for all $x,y,\xi\in\mR^d$,
\begin{align*}
\kappa_1^{-1}|\xi|\le |\sigma(x)\xi|\le \kappa_1|\xi|,\quad \nor \nabla\sigma\nor_{p}\le \kappa_1,
\end{align*}
and
$$
\|\sigma(x)-\sigma(y)\|_{HS}\le \kappa_1|x-y|^{\beta},
$$
where $\|\cdot\|_{HS}$ is the Hilbert-Schmidt norm.
\end{enumerate}
%Note that $b$ is the KBM-vector field provided $b$ satisfies {\bf (H$_b^2$)}.
Please see the definition of the localized $L^p$ norm $\nor\cdot\nor_p$ in Section \ref{locBp}.
\br \rm
We Note that
\begin{align*}
\nor \bar{b}(\cdot,\mu)\nor_{p_0}&\le \frac{1}{T}\int_0^T\nor b(s,\cdot,\mu)\nor_{p_0}\dif s+\nor \frac{1}{T}\int_0^T\(b(s,\cdot,\mu)-\bar{b}(\cdot,\mu)\dif s\)\nor_{p_0}\\
&\le \kappa_0+\omega(T)\nor H(\cdot,\mu)\nor_{p_0},
\end{align*}
provided conditions {\bf (H$^1_b$)} and {\bf (H$^2_b$)} hold. Taking $T\to\infty$, we have
\begin{align*}
\nor \bar{b}(\cdot,\mu)\nor_{p_0}\le \kappa_0.
\end{align*}
Similarly, we have
\begin{align*}
\nor \bar{b}(\cdot,\mu)-\bar{b}(\cdot,\nu)\nor_{p_0}\le \kappa_0\|\mu-\nu\|_{var}.
\end{align*}
Thus, the coefficient $\bar{b}$ satisfies the condition {\bf (H$^1_b$)} with the same constant $\kappa_0$.
\er
Under assumptions {\bf (H$_b^1$)} and {\bf (H$_\sigma$)}, for any initial value $\xi\in\sF_0$,
it is well-known that there is a unique strong solution to DDSDE \eqref{MV1} (respectively, \eqref{MV2} );
see \cite{Zhao2020} and \cite{HRZ22}. The aim of this work is to show the following strong
and weak convergence of the averaging principle for DDSDE and SDE with $L^p$ drift.
\bt\label{in:Main}
Under {\bf (H$_b^1$)}, {\bf (H$_b^2$)} and  {\bf(H$_\sigma$)},
 for any $T>0$ and $\ell\in(0,1)$,
there is a constant $C$, depending only on $\kappa_0,\kappa_1,T,d,\beta,p_0,p,\ell$, such that for any $\eps>0$
\begin{align}\label{in:NMain}
\sup_{t\in[0,T]}\|\mu^\eps_t-\mu_t\|_{var}\le
 C \inf_{h>0}\Big(h^{\frac12-\frac{d}{2p_0}}
+\omega\left(\frac{h}{\eps}\right)\Big)
\end{align}
%\begin{align*}
%\sup_{t\in[0,T]}\|\mu^\eps_t-\mu_t\|_{var}
%\le C\inf_{h>0}\left(\left(\omega\left(\frac{h}{\eps}\right)\right)^2+h^{1-\delta}\right)
%\end{align*}
and
\begin{align}\label{in:StMain}
\mE\left(\sup_{t\in[0,T]}|X^\eps_t-X_t|^{2\ell}\right)
\le
 C\inf_{h>0}\left(\left(\omega(h/\eps)\right)^2+h^{1-\frac{d}{p_0}}\right)^\ell.
 \end{align}

\et
\iffalse
\br\rm
By Sobolev embedding \eqref{Eml00}, under the condition {\bf(H$_\sigma$)}, we have
\begin{align*}
\sigma\in \bC^{1-\frac{d}{p_0}}.
\end{align*}
Hence, we can assume a priori that $\beta\ge 1-\frac{d}{p_0}$.

\er
\fi
When the drift $b$ is independent of the distribution, we have the following results, where the convergence rate is independent of $p_0$.
\bt\label{in:Main2}
Assume that
$$
b(t,x,\mu)\equiv b(t,x).
$$
 Under {\bf (H$_b^1$)}, {\bf (H$_b^2$)} and {\bf(H$_\sigma$)}, for any $T>0$, $\delta>0$ and $\ell\in(0,1)$,
there is a constant $C$, depending only on $\kappa_0,\kappa_1,T,d,p_0,p,\delta,\ell$, such that for any $\eps>0$
$$
\mE\left(\sup_{t\in[0,T]}|X^\eps_t-X_t|^{2\ell}\right)
\le C\inf_{h>0}\left(\left(\omega\left(\frac{h}{\eps}\right)\right)^2+h^{1-\delta}\right)^\ell.$$
\et
\br\rm
\begin{itemize}
\item[(i)] Since we use the Zvonkin transformation using the parabolic equation,
when $\bar{b}=\bar{b}(t)=\bar{b}(\cdot,\mu_t)$ depends on the time variable $t$,
the time regularity for solutions to this parabolic equation affects
the convergence rate (see \eqref{GG051} and Lemma \ref{LemEE01} for more details).
When $\bar{b}$ is independent of time, we can construct the Zvonkin transformation
using the elliptic equation. Hence, the convergence rates in Theorems \ref{in:Main}
and \ref{in:Main2} are different.

\item[(ii)] Noting that $\nor f\nor_{p_0}\lesssim\| f\|_\infty$ for all $p_0\in(1,\infty)$, all results in our paper are valid for $p_0=\infty$, in which case the rate of convergence in \eqref{in:StMain} is
$$
\inf_{h>0}\left(\left(\omega\left(\frac{h}{\eps}\right)\right)^2+h^{1-\delta}\right)^\ell
$$
for any $\delta>0$. In particular, we obtain the convergence rate $\eps^{\frac13-\delta}$ for a large number of examples (see e.g. Example \ref{Ex2} below), which is faster than $\eps^{\frac16}$ in \cite{HLL2022}.

\end{itemize}
\er

\br\rm
These averaging principle results are also applicable to the following system
\begin{equation}\label{inseq}
\dif X_t=\eps b\left(t,X_t,\cL(X_t)\right)\dif t+\sqrt{\eps}\sigma(X_t)\dif W_t.
\end{equation}
Define ${Z^\eps_t}:=X_{t/\eps}$ and $W^\eps_t:=\sqrt{\eps}W_{t/\eps}$ for all $t\in\R_+$. We rewrite \eqref{inseq} as
\begin{equation*}
\dif {Z^\eps_t}=b\left(t/\eps,{Z^\eps_t},\cL({Z^\eps_t})\right)\dif t+\sigma({Z^\eps_t})\dif W^\eps_{t}.
\end{equation*}
Then we can consider the following system
\begin{equation*}
\dif \tilde{X}^\eps_t=b\left(t/\eps,\tilde{X}^\eps_t,\cL(\tilde{X}^\eps_t)\right)\dif t+\sigma(\tilde{X}^\eps_t)\dif W_{t}.
\end{equation*}
\er

Note that since the drift of both the DDSDE and SDE in this paper is locally $L^{p_0}_x$ integrable,
we cannot use Gronwall's lemma or the generalized Gronwall lemma directly to
prove the convergence of $X^\eps$ to $X$ as in \cite{HLL2022, SXW2022}. On the other hand, {our system \eqref{MV1} can be rewritten in the following slow-fast system:
\begin{align*}
\left\{
\begin{aligned}
&\dif X_t^\eps=b\left(Y_t^\eps,X_t^\eps,\mu_t^\eps\right)\dif t+\sigma(X_t^\eps)\dif W_t,\\
&\dif Y_t^\eps=\frac1\eps\dif t.
\end{aligned}
\right.
\end{align*}
Since the Kolmogorov operator of the fast process $Y_t^\eps=\frac{t}{\eps},~t\geq0$, does not have a second order elliptic part,}
we cannot use the technique based on the Poisson equation as in \cite{RSX2021}.
To overcome these difficulties, we use Zvonkin's transformation to remove the drift $b$
and employ the classical technique of time discretization.

More precisely, consider the following backward PDE for $t\in[0,T]$ related to \eqref{MV2}
\begin{align*}
\p_tu+a_{ij}\p_i\p_ju-\lambda u+B\cdot\nabla u+B=0,\quad u(T)=0,
\end{align*}
where $B(t,x):=\bar{b}(x,\mu_t)$, $\lambda\ge0$, is the dissipative term.
Under {\bf (H$_\sigma$)} and {\bf (H$^1_b$)}, Xia et al \cite{Xi-Xi-Zh-Zh}
proved that for a sufficiently large number $\lambda$, there is a solution $u$ such that
\begin{align*}
|\nabla u(t,x)|\le\frac{1}{2},\quad t\in[0,T], x\in\mR^d.
\end{align*}
Hence, if we define $\Phi_t(x):=x+u(t,x)$, then $x\to\Phi_t(x)$ is a $C^1$ diffeomorphism of $\mR^d$.
{By It\^o's formula}, {$Y^\eps_t:=\Phi_t(X^\eps_t)$ and $Y_t:=\Phi_t(X_t)$ solve the following new SDEs:}
\begin{align*}
\dif Y^{\eps}_t=
&
\lambda u(t,\Phi^{-1}_t(Y^\eps_t))\dif t+(\sigma^*\nabla\Phi_t(\Phi^{-1}_t(Y^\eps_t)))\dif W_t\\
&
+\(b(t/\eps,X^\eps_t,\mu^\eps_t)-\bar {b}(X^\eps_t,\mu_t)\)\cdot\nabla \Phi_t\(X^{\eps}_t\)\dif t
\end{align*}
and
\begin{align*}
\dif Y_t=\lambda u(t,\Phi^{-1}_t(Y_t))\dif t+(\sigma^*\nabla\Phi_t)(\Phi^{-1}_t(Y_t))\dif W_t,
\end{align*}
where $\sigma^*$ is the transpose of $\sigma$ and $\Phi^{-1}_t$ is the inverse of $x\to\Phi_t(x)$.
Since these new systems have differentiable diffusion coefficients and the drifts are Lipshitz continuous,
we can use the stochastic Gronwall inequality. Please see the complete formulation in Section \ref{Section5}.

The remaining part of the proof is about how to use the technique of time discretization to estimate the following crucial term
\begin{align}\label{in:NCD00}
\mE\Big[\sup_{t\in[0,T]}\Big|\int_0^t\(b(\frac{s}{\eps},X^\eps_s,\mu^\eps_s)-\bar{b}(X^\eps_s,\mu_s)\)\cdot\nabla \Phi_s\(X^{\eps}_s\)\dif s\Big|^2\Big].
\end{align}
In particular, we need to estimate
\begin{align}\label{in:NCC01}
\|\mu_t-\mu^\eps_t\|_{var}
\end{align}
and the following difference for $B(t)=b(s/\eps)\cdot\nabla \Phi_t$ and $B(t)=\bar{b}\cdot \nabla \Phi_t$:
\begin{align}\label{in:NCC00}
\mE\Big|\int_0^t\(B(s,X^\eps_{s},\mu^\eps_{s})-B(s,X^\eps_{\pi_h(s)},\mu^\eps_{\pi_h(s)})\)\dif s\Big|^2,
\end{align}
where $\pi_h(s)$ is defined by $\pi_h(s)=s$ for $s\in[0,h)$ and
\begin{align*}
\pi_h(s):=kh,\quad s\in[kh,(k+1)h),~~\forall k\in\mN.
\end{align*}
%Moreover, we need to estimate
%\begin{align}\label{in:NCC01}
%\|\mu_t-\mu^\eps_t\|_{var}.
%\end{align}
Since $b$ is not continuous, it is not an easy problem to give an estimate for \eqref{in:NCC00}.
Thanks to the important observation in \cite[Lemma 2.1]{DG20},
we deal with this problem by considering the time regularity for the semi-group
generated by the solutions to \eqref{MV1}. Basically, if $X^\eps_t=W_t$ is a standard Brownian motion and $f$ is a discontinuous function, we have
\begin{align*}
&\quad\mE\Big|\int_0^t\(f(W_{s})-f(W_{\pi_h(s)})\)\dif s\Big|^2\\
&=2\mE\int_0^t\int_s^t\(f(W_{s})-f(W_{\pi_h(s)})\)\(f(W_{r})-f(W_{\pi_h(r)})\)\dif r\dif s\\
&=2\mE\int_0^t\(f(W_{s})-f(W_{\pi_h(s)})\)\Big(\mE^{\sF_s}\int_s^t\(f(W_{r})-f(W_{\pi_h(r)})\)\dif r\Big)\dif s,
\end{align*}
where $\sF_s:=\sigma(W_r,r\le s)$ and $\mE^{\sF_s}$ is the conditional expectation with respect to $\sF_s$.
We note that by the Markov property,
\begin{align*}
&\quad\mE^{\sF_s}\int_s^t\(f(W_{r})-f(W_{\pi_h(r)})\)\dif r\\
&=\mE^{\sF_s}\int_s^{s+h}\(f(W_{r})-f(W_{\pi_h(r)})\)\dif r+\int_{s+h}^t\(P_{r-s}f(W_{s})-P_{\pi_h(r)-s}f(W_{s})\)\dif r,
\end{align*}
where $P_tf(x):=\mE f(x+W_t)$ is the semi-group of Brownian motion.
{Note that as $h\to0$ the first term converges to zero.
And the second term also converges to zero as $h\to0$ since $s\to P_sf$ is continuous.}
Please see Section \ref{Section3} below for more details.
We mention that similar estimates are obtained in \cite{LL21} by the stochastic sewing techniques.

To estimate \eqref{in:NCC01}, we employ a method based on the Kolmogorov equation which is also used in \cite{RSX19}.
Then, again by time discretization, we estimate the difference \eqref{in:NCC01} and obtain \eqref{in:NMain} (see Section \ref{Section4}).
%%%%%%%%%%%%%%%%%%%%%%%%%%%%%%%%%%%%%%%

\iffalse
More specifically,
for Theorem \ref{in:Main2}, the coefficient is independent of distribution,
the critical point in the analysis for convergence of $X^\eps$ to $X$ is Lemma \ref{LemEE01}, which is based on Krylov's estimates and Girsanov's theorem. With the help of Lemma \ref{LemEE01} and Zvonkin's transformation, we obtain averaging principle for SDE with singular drift. When we consider DDSDE (Theorem \ref{in:Main}), another main obstacle is establishing the weak convergence. In particular, we show that $\mu^\eps$ converges, as $\eps$ goes to zero, to $\mu$ based on the Cauchy problem to Kolmogorov equation generated by $X_t,t\geq0$; see Theorem \ref{S4:Main}.
\fi

Now we illustrate our results by some examples. Firstly, the following example shows that the function $\omega(t)$
can be of the form $t^{-\alpha}$ with any power $\alpha\in(0,1]$,
\iffalse
\begin{Examples}\label{Ex0}\rm
Let
\begin{align}\label{Ex0eq}
b(t,x,\mu):=\left(\sum_{k=1}^\infty\left( \1_{[k,k+1]}(t)a_k\sin (2k\pi t)\right)\right)g(x,\mu),
\end{align}
where $g$ satisfies $\bf (H_1)$ and $|a_k|\leq f(k-1)$. Here nonnegative decreasing function $f$ satisfies $\lim\limits_{T\rightarrow\infty}\frac{f(T)}{T}=0$. Then it can be verified that
$$
\bar{b}(x,\mu)=g(x,\mu), \quad
\omega(t)=\frac{f(t)}{t},\quad \text{and}\quad H(x,\mu)=g(x,\mu).
$$
Note that $f$ can be $t^{-\alpha},\alpha>0$ and $e^{-t}$.
\end{Examples}
\fi
which can also be used in some systems with singular interactions.
\begin{Examples}\label{Ex1}\rm
Consider the following DDSDE in $\mR^d$
\begin{align*}%\label{Ex1eq}
\dif X_t^\eps
&
=\left(
\Big[(1+t/\eps)^{-\alpha_1}+1\Big]\int_{\R^d} \frac{X_t^\eps-y} {|X_{t}^{\eps}-y|^{\alpha_2}}\mu_t^\eps(\dif y)\right)\dif t+\dif W_t\\\nonumber
&
=:b(t/\eps,X_{t}^{\eps},\mu_{t}^{\eps})\dif t+\dif W_t,
\end{align*}
where $\alpha_1>0$, $1<\alpha_2<2\wedge(1+\frac{d}{2})$ and $\mu_t^\eps$ is the distribution of $X_t^\eps$. It is clear that
the averaged equation is
\begin{align*}
\dif X_t
&
=\left(\int_{\R^d}\frac{X_t-y} {|X_t-y|^{\alpha_2}}\mu_t(\dif y)\right)\dif t+\dif W_t\\
&
=:\bar{b}(X_t,\mu_t)\dif t+\dif W_t,
\end{align*}
where $\mu_t$ is the distribution of $X_t$,
and
$$
\left|\frac{1}{T}\int_{t}^{t+T}\left(b(s,x,\mu)
-\bar{b}(x,\mu)\right)\dif s\right|
\leq\omega\left(T\right)(1-\alpha_1)^{-1}
\int_{\R^d} \frac{|x-y|} {|x-y|^{\alpha_2}} \mu(\dif y)
$$
for all $(T,t,x,\mu)\in\R^2\times\R\times\cP(\R^d)$, where
\begin{eqnarray*}
\omega(t)=
\begin{cases}
t^{-(\alpha_1\wedge1)} & \text{for}~ \alpha_1\in(0,1)\cup(1,\infty)\\
t^{-1}\ln t           & \text{for}~ \alpha_1=1.
\end{cases}
\end{eqnarray*}
Then we have for any $\delta>0$
$$
\sup_{t\in[0,T]}\|\mu_t^\eps-\mu_t\|_{var}\leq
C\eps^{\frac{\alpha(2-\alpha_2)}{2+2\alpha-\alpha_2}-\delta}
$$
and
$$
\left[\mE\left(\sup_{t\in[0,T]}|X^\eps_t-X_t|^{2\ell}\right)
\right]^{\frac{1}{\ell}}
\le C\eps^{\frac{4\alpha-2\alpha\alpha_2}
{2+2\alpha-\alpha_2}-\delta}
$$
for any $0<\ell<1$, where $\alpha=\alpha_1\wedge1$ for $\alpha_1\in(0,\infty)$.
\end{Examples}

Next we give a more general example, where the function $\omega(t)\asymp t^{-1}$, i.e.
there exists a constant $C$ such that $C^{-1}t^{-1}\leq\omega(t)\leq Ct^{-1}$.\begin{Examples}\label{Ex2}\rm
Let $p_0\in(d\vee2,\infty)$.
Consider the following DDSDE
\begin{align}\label{Ex2eq}
\dif X_t^\eps
=\left[\int_{\R} F\left(\sin(\xi t/\eps),\int_{\R^d}\phi(X_t^\eps,y)\mu_t^\eps(\dif y)\right)\nu(\dif \xi)\right]
\dif t+\dif W_t,
\end{align}
where $\mu_t^\eps$ is the time marginal law of $X_t$, $F:[-1,1]\times\mR^m\to\mR^d$ is measurable and satisfies for some constant $L_F>0$
\begin{align}\label{in:EX2}
|F(t,0)|\leq L_F, ~~|F(t,x)-F(t,y)|\leq L_F|x-y|\quad \text{for all $(t,x,y)\in [-1,1]\times\mR^{2m}$},
\end{align}
${\nu}$ is some finite measure on $\R$ satisfying
$$
\int_{\mR\backslash\{0\}}\frac{1}{\xi}\nu(\dif\xi)<\infty
$$
and $\phi :\mR^d\to\mR^m$ is measurable and satisfies
$$
\sup\limits_y\nor\phi(\cdot,y)\nor_{p_0}<\infty.
$$
{Set
\begin{align*}
b(t,x,\mu):=\int_{\R} F\left(\sin(\xi t),\int_{\R^d}\phi(x,y)\mu(\dif y)\right)\nu(\dif \xi)
\end{align*}
and
\begin{align*}
\bar{b}(x,\mu):=&\frac{1}{2\pi}\int_0^{2\pi} F\left(\sin\tau,\int_{\R^d}\phi(x,y)\mu(\dif y)\right)\nu(\R\setminus\{0\})\dif\tau\\
&+F\left(0,\int_{\R^d}\phi(x,y)\mu(\dif y)\right)\nu(\{0\}).
\end{align*}
We claim that
\begin{align}\label{NNA00}
\begin{split}
\text{$(b,\bar b)$ satisfies conditions {\bf (H$_b^1$)} and {\bf (H$_b^2$)} with $\omega(T):=\frac{4\pi L_F}{T}\int_{\mR\setminus\{0\}}\frac{\nu(\dif \xi)}{|\xi|}$,}
\end{split}
\end{align}
which is proved in the Appendix.

}
Hence, based on Theorem \ref{in:Main} and Lemma \ref{lem61}, we have
$$
\sup_{t\in[0,T]}\|\mu_t^\eps-\mu_t\|_{var}\leq
C\eps^{\frac{1}{3}-\frac{2d}{9p_0-3d}}
$$
and
\begin{equation}\label{Ex2con}
\left[\mE\left(\sup_{t\in[0,T]}|X^\eps_t-X_t|^{2\ell}\right)
\right]^{\frac{1}{\ell}}
\le C\eps^{\frac{2}{3}\left(1-\frac{2d}{3p_0-d}\right)}
\end{equation}
for any $0<\ell<1$ and $T>0$, where $X_t$ is the solution of
\begin{align*}
\dif X_t=\bar{b}(X_t,\mu_t)\dif t+\dif W_t.
\end{align*}
Here $\mu_t$ is the distribution of $X_t$ for all $t\geq0$.
%For example, we can consider
%$$\int_{\R_\xi} F\left(\sin(\xi t),\int_{\R^d}\phi(x,y)
%\mu(\dif y)\right)\nu(\dif\xi)=\int_{\R_\xi}|\sin\xi t+\int_{\R^d}\phi(x,y)\mu(\dif y)|\nu(\dif \xi).$$
%Note that
%$\int_\R\sin\left(\xi t\right)\hat{\nu}(\dif \xi)
%=\sin t+\sin (\sqrt{2}t)$ is a quasi-periodic function
%if $\hat{\nu}=\delta_1+\delta_{\sqrt{2}}$.
\end{Examples}

%Note that the rate of strong convergence in \eqref{Ex2con} is better than \cite{RSX2021} provided
%$p_0$ is large enough.

\br\rm
In this paper, we only consider the case where $\sigma$ does not depend on time $t$ and the measure $\mu$.
We hope that in future work
we can use our framework here to study the time-inhomogeneous case $\sigma=\sigma(t,x)$ or the even more general case $\sigma=\sigma(t,x,\mu)$ under the following standard condition on $\sigma$
\begin{enumerate}[{\bf (H$_\sigma^{'}$)}]
\item
There are functions $\omega_\sigma$ and $H_\sigma$ such that
\begin{align*}%\label{in:001}
\frac{1}{T}\int_{t}^{T+t}\|\sigma(t,x,\mu)-\bar{\sigma}(x,\mu)\|_{HS}^2\dif t\le \omega_\sigma(T)H_\sigma(x,\mu)
\end{align*}
for all $(T,t,x)\in(\R_+^2\times\R^d)$,
where $\lim\limits_{t\to\infty}\omega_\sigma(t)=0$ and $\sup\limits_\mu\nor H_\sigma(\cdot,\mu)\nor_{\frac{p_0}{2}}<\infty$.
\end{enumerate}

\er

{\bf Structure of the paper.}

In Section \ref{Section2}, we first introduce the localized Bessel potential space,
{the embedding lemma and the local Hardy-Littlewood maximal function.}
Then we show some well-posedness results
and priori estimates about parabolic and elliptic PDEs for later establishing
the time regularity of solutions in Section \ref{Section3} and
performing Zvonkin's transformation in Section \ref{Section5}.

In Section \ref{Section3}, we study the time regularity of solutions to
both SDE and PDE in the $L^p$ framework.
As usual, we use the technique of time discretization
to obtain the convergence for the averaging principle.
That is we need to show a time discretization's type estimate
of \eqref{in:NCC00},
and then estimate \eqref{in:NCD00} by condition {\bf (H$^2_b$)} (see Section \ref{Sec3.1}).
For \eqref{in:NCC00}, we first asume that $Z_t^\eps$ is a solution to
the SDE \eqref{MV1} without drift, that is to say $b\equiv0$.
By the semi-group property from condition {\bf (H$_\sigma$)}
and heat kernel estimates we estimate \eqref{in:NCC00}
for any localized $L^p$ integrable $B$ with $p>(d\vee2)$. Then it follows
from the Girsanov transform that \eqref{in:NCC00} holds for $X_t^\eps$,
where $X_t^\eps$ is the solution to \eqref{MV1} (see Section \ref{sec:3.1}).
% then use the Girsanov transform.
%We mention that \eqref{BB03} has been investigated
%in the area of Euler-Maruyama schema (see \cite{DG20} and \cite{LL21}).
%Here one can use the semi-group property from condition {\bf (H$_\sigma$)}
%and heat kernel estimates to estimate \eqref{in:NCC00} when $X_t^\eps$ is a solution to the SDE without drift (see Section \ref{sec:3.1}). %\eqref{BB03}.
%Moreover, the difference $\|\mu_s-\mu_{\pi_h(s)}\|_{\var}$ comes from Duhamel's formula
%and time regularity of solutions to SDEs in Section \ref{Sec3.1}.
We also prove the time regularity of the gradient of solutions for parabolic PDEs in Section \ref{SubS32},
which is used in Section \ref{Section4} and Section \ref{Section5} respectively.

{In Section \ref{Section4}, {after realizing that \eqref{NN03} holds},
we prove the weak convergence rate of \eqref{in:NCC01},
where the key point is the estimates for time regularity obtained in Section \ref{SubS32}.}

In Section \ref{Section5},
we give the proofs to our main results Theorems \ref{in:Main}
and \ref{in:Main2} by Zvonkin's
transformation.
\vspace{5mm}

{\bf Notations.}

Throughout this paper, $|\cdot|$ denotes the Euclidean norm on $\mR^d,d\in\N$.
For $j\in\mN\cup\{0\}$, we use the notation $\nabla^j$ to denote
the $j$th order derivative. Moreover, let $C^\infty_0$ denote
the function space of all smooth functions with compact support.
We write $C^k_b$ (respectively, $C^\infty_b$) to mean the space of all smooth functions
with bounded $j$th derivatives for all integer $j\in[0,k]$ (respectively, $j\in\mN\cup\{0\}$).
Let $[\alpha]$ be the largest integer which is smaller than $\alpha$
for any constant $\alpha>0$.
We use $C$ with or without subscripts to denote an unimportant constant,
whose value may change from line to line.
Writing $``:="$ we mean equal by definition.
By $A\lesssim B$
%and $A\asymp B$,
we mean that for some unimportant constant $C\ge1$,
\begin{align*}
A\le CB.
%\quad \text{and}\quad C^{-1}B\le A\le CB.
\end{align*}

%$\tilde{\mL}^p(T):=\mL^\infty\left([0,T];\widetilde{L}^p(\mR^d)\right).$

\section{Preliminaries}\label{Section2}

In this section we show some auxiliary results for later use.
\subsection{Localized Bessel potential space}\label{locBp}
For any $(\alpha,p)\in\R\times(1,\infty)$, we write
$$
H^{\alpha,p}:=(\mathbb I-\Delta)^{-\alpha/2}
\left(L^p(\R^d)\right)
$$
for the usual Bessel potential space with the norm given by
$
\|f\|_{\alpha,p}:=\|(\mathbb I-\Delta)^{\alpha/2}f\|_p,
$
where $\|\cdot\|_p$ is the usual $L^p$-norm.
Here $(\mathbb I-\Delta)^{\alpha/2}f$ is defined through
Fourier's transform
$$
(\mathbb I-\Delta)^{\alpha/2}f:=\cF^{-1}\left((1+|\cdot|^2)^{\alpha/2}
\cF f\right).
$$
We note that if $\alpha=n\in\N$ and $p\in(1,\infty)$, an equivalent norm in $H^{n,p}$
is given by 
$$
\|f\|_{n,p}=\|f\|_p+\|\nabla^nf\|_p.
$$

Let $\chi\in C_0^\infty(\R^d)$
such that $\chi(x)=1$ for $|x|\leq1$,  $\chi(x)=0$
for $|x|>2$ and $|\chi(x)|\leq1, \forall x\in\R^d$. Define
$$
\chi_r(x):=\chi(x/r), \quad \chi_r^z(x):=\chi_r(x-z)
$$
for all $r>0$ and $z\in\R^d$. Given $r>0$, we introduce the
following localized $H^{\alpha,p}$-space:
$$
\widetilde{H}^{\alpha,p}:=\left\{f\in H^{\alpha,p}_{\rm{loc}}(\R^d):\nor f\nor_{\alpha,p}:=\sup_z\|\chi_r^zf\|_{\alpha,p}<\infty\right\}.
$$
Clearly, this space does not depend on $r$ and the corresponding norms are equivalent.
When $\alpha=0$, we simply write
$$
\widetilde{L}^p:=\widetilde{H}^{0,p}\quad {\text{and}}\quad
\nor f\nor_p:=\nor f\nor_{0,p}.
$$
{It follows from H\"older's inequality that for any $1\le p_2\le p_1\le\infty$
\begin{align*}%\label{Pre1:emb}
L^{p_1}\subset\widetilde{L}^{p_1}\subset\widetilde{L}^{p_2}.
%\ \ \widetilde{\mL}^p(T)\subset\widetilde{\mL}^p_q(T).
\end{align*}
}

For $T>0$, $p,q\in(1,\infty)$ and $\alpha\in\R$, we set
\begin{align*}
  \mL_q^p(T):=L^q\([0,T];L^p\),\quad
  \mH_q^{\alpha,p}(T):=L^q\([0,T];H^{\alpha,p}\).
\end{align*}
Now we introduce the localized space
\begin{align*}
\widetilde{\mH}_q^{\alpha,p}(T):=\left\{f\in\mH_q^{\alpha,p}(T):
\nor f\nor_{\widetilde{\mH}_q^{\alpha,p}(T)}:=
\sup_{z\in\R^d}\|\chi_r^zf\|_{\mH_q^{\alpha,p}(T)}<\infty\right\}.
\end{align*}
By a finite covering technique, it can be verified  that also the definition of $\widetilde{\mH}_q^{\alpha,p}$
does not depend on the choice of $r$ (see \cite[Section 2]{Xi-Xi-Zh-Zh}).
We note that all these spaces are Banach spaces and that
\begin{align*}
  L^q\([0,T];\widetilde{H}^{\alpha,p}\)\subset
  \widetilde{\mH}_q^{\alpha,p}(T).
\end{align*}
{For $\alpha=0$, set
\begin{align*}
\widetilde{\mL}^p_q(T):=\widetilde{\mH}^{0,p}_q(T).
\end{align*}
}
If $q=\infty$, for simplicity, we define
\begin{align*}
\widetilde{\mH}^{\alpha,p}(T):=
L^\infty([0,T];\widetilde{H}^{\alpha,p}),
\quad
\widetilde{\mL}^p(T):=\widetilde{\mL}^p_\infty(T),
\quad\text{and}\quad \mL^\infty_T:=L^\infty([0,T]\times\mR^d).
\end{align*}

Let $C^\alpha$ denote the H\"older space of order $\alpha$, which consists of all functions $g$ for which the norm
\begin{align*}
\|g\|_{C^\alpha}:=\sum_{|\beta|\leq[\alpha]}
\|\nabla^\beta g\|_{L^\infty}
+\sum_{|\beta|=[\alpha]}
\left[\nabla^\beta g\right]_{C^{\alpha-[\alpha]}}
\end{align*}
is finite, where
\begin{align*}
  [g]_{C^{\alpha-[\alpha]}}:=\sup_{x,y\in\R^d \atop x\neq y}
  \frac{|g(x)-g(y)|}{|x-y|^{\alpha-[\alpha]}}.
\end{align*}

\begin{lemma}[Embedding lemma]\label{Eml}
Let $1<p<\infty$.
Then we have
\begin{align*}%\label{Eml00}
\widetilde{H}^{\alpha,p}\subset C^{\alpha-d/p}
\end{align*}
and
\begin{align*}
  \widetilde{\mH}^{\alpha,p}(T)\subset L^\infty\left([0,T];C^{\alpha-d/p}\right)
\end{align*}
provided $\alpha>d/p$.
\end{lemma}
\begin{proof}
It follows from Sobolev's embedding theorem that $H^{\alpha,p}\subset C^{\alpha-d/p}$ if $\alpha>d/p$.
Note that
\begin{align*}
\|g\|_{C^{\alpha-d/p}}\leq\sup_{z}\|\chi_r^zg\|_{C^{\alpha-d/p}}
\end{align*}
for all $g\in C^{\alpha-d/p}$ and $r>0$. Therefore, we have
\begin{align*}
\|g\|_{C^{\alpha-d/p}}
\leq\sup_{z}\|\chi_r^zg\|_{C^{\alpha-d/p}}
\lesssim \sup_{z}\|\chi_r^zg\|_{H^{\alpha,p}}
=\nor g\nor_{\widetilde{H}^{\alpha,p}}.
\end{align*}
Moreover, for all $f\in\widetilde{\mH}^{\alpha,p}$ one sees that
\begin{align*}
\sup_{t\in[0,T]}\|f(t)\|_{C^{\alpha-d/p}}
\leq\sup_{t\in[0,T]}\sup_{z}\|\chi_r^zf(t)\|_{C^{\alpha-d/p}}
\lesssim \sup_{t\in[0,T]}\sup_{z}\|\chi_r^zf(t)\|_{H^{\alpha,p}}
=\nor f\nor_{\widetilde{\mH}^{\alpha,p}}.
\end{align*}
\end{proof}

Now we introduce the local Hardy-Littlewood maximal function, which is defined by
\begin{align*}
\cM f(x):=\sup_{r\in(0,1)}\frac{1}{|B_r|}\int_{B_r}f(x+y)\dif y.
\end{align*}

The following result is taken from Lemma 2.1 in \cite{Xi-Xi-Zh-Zh} (see also \cite[Appendix A]{CD08}).

\bl
\begin{enumerate}[(i)]
\item There is a constant $C=C(d)>0$, such that for any $f\in L^\infty(\mR^d)$ with $\nabla f\in L^1_{\rm{loc}}(\mR^d)$
and for Lebesgue-almost all $x,y\in\mR^d$,
\begin{align}\label{Pre1:Mnf}
|f(x)-f(y)|\le C|x-y|\(\cM|\nabla f|(x)+\cM|\nabla f|(y)+\|f\|_\infty\).
\end{align}
\item For any $p\in(1,\infty)$, there is a constant $C=C(d,p)>0$ such that for all $f\in \widetilde{L}^p$,
\begin{align}\label{Pre1:Mix}
\nor \cM f\nor_{p}\le C\nor f\nor_{p}.
\end{align}
\end{enumerate}
\el

\subsection{Parabolic equation}
In order to study DDSDE, we consider the following second order parabolic PDE in $\mR_+\times\mR^d$:
\begin{align}\label{Pre4:PDE}
\p_tu=a_{ij}\p_i\p_ju-\lambda u+b\cdot\nabla u+f,\ \ u(0)=\varphi,
\end{align}
where $\lambda\ge0$, $a=(a_{ij}):\R^d\rightarrow\R^d\otimes\R^d$
is a symmetric matrix-valued Borel measurable function satisfying {\bf{\bf(H$_a$)}}, i.e.,
\begin{enumerate}
  \item[{\bf(H$_a$)}] there exist constants $c_0>0$ and $\theta\in(0,1)$ such that
  \begin{equation*}
    c_0^{-1}|\xi|\leq|a(x)\xi|\leq c_0|\xi|,\quad
    \|a(x)-a(y)\|_{HS}\leq c_0|x-y|^\theta
  \end{equation*}
  for all $\xi\in\R^d$ and $x,y\in\R^d$,
\end{enumerate}
and $b:\R\times\R^d\rightarrow\R^d$ is a vector-valued Borel measurable function.
Firstly, we introduce the definition of a solution to PDE \eqref{Pre4:PDE}.
\bd\label{Pre4:Def}\rm
Let $T>0$, $p,q\in(1,\infty)$, $\lambda\ge0$,
$b,f\in \widetilde{\mL}^p_q(T)$ and $\varphi\in C_b^\infty$.
We call a function $u$ with
$\p_tu\in\widetilde{\mL}_q^p(T)$ and
$u\in \widetilde{\mH}^{2,p}_q(T)$ a {\em solution} of PDE \eqref{Pre4:PDE} if for Lebesgue almost all $(t,x)\in\mR_+\times\mR^d$,
\begin{align*}
u(t,x)=\int_0^t\(a_{ij}\p_i\p_ju(s,x)-\lambda u(s,x)+b\cdot\nabla u(s,x)+f(s,x)\)\dif s+\varphi(x).
\end{align*}
\ed

\begin{remark}\rm
For any $\chi\in C_0^\infty(\R^d)$ and $f\in\widetilde{\mathbb H}_q^{\alpha,p}(T)$,
by the definition of the localized spaces $\widetilde{\mathbb H}_q^{\alpha,p}(T)$,
we have $\chi f\in{\mathbb H}_q^{\alpha,p}(T)$. Hence for any solution $u$ of PDE
\eqref{Pre4:PDE} in the sense of Definition \ref{Pre4:Def}, $\chi u$ is H\"older continuous
on $[0,T]\times\R^d$ if $d/p+2/q<2$ according to \cite[Lemma 10.2]{KR05}. Moreover,
$\nabla (\chi u)$ is H\"older continuous on $[0,T]\times\R^d$
if $d/p+2/q<1$. In view of the arbitrariness
of the cut-off function $\chi$, $u$ (respectively, $\nabla u$) are locally H\"older continuous
on $[0,T]\times\R^d$ if $d/p+2/q<2$ (respectively, $d/p+2/q<1$).
\end{remark}

The following property of the solution $u$ comes from \cite[Theorem 3.2]{Xi-Xi-Zh-Zh}.
\bl\label{lem210}
Let $T>0$, $p,q\in(1,\infty)$ with $2/q+d/p<1$,
$\lambda\ge0$, $b\in \widetilde{\mL}^{p}_{q}(T)$. Set
$$
\Theta:=(d,T,p,q,\|b\|_{\widetilde{\mL}^{p}_{q}(T)},c_0,\theta).
$$ Then there is a constant $\lambda_0=\lambda_0(\Theta)$ such that for all $\lambda\ge\lambda_0$ and $f\in \widetilde{\mL}^{p}_{q}(T)$ and $\varphi\in \widetilde{H}^{2,p}$,
there is a unique solution $u$ to PDE \eqref{Pre4:PDE} on $[0,T]$ in the sense of Definition \ref{Pre4:Def}
such that for any $\alpha\in[0,2)$, $p'\in[p,\infty]$, $q'\in[q,\infty]$ with
$$
\beta:=2-\alpha+\frac{2}{q'}+\frac{d}{p'}-\(\frac{2}{q}+\frac{d}{p}\)>0,
$$
there is a constant  $C=C(\Theta,\alpha,p',q')>0$ such that for all $\lambda\ge\lambda_0$
\begin{align}\label{CCA02}
\lambda^{\frac{\beta}{2}}
\nor u\nor_{\widetilde{\mH}^{\alpha,p'}_{q'}(T)}
+\nor\p_tu\nor_{\widetilde{\mL}^p_q(T)}
+\nor\nabla^2 u\nor_{\widetilde{\mL}^p_q(T)}
\le C\(\nor f\nor_{\widetilde{\mL}^p_q(T)}+\nor\varphi\nor_{2,p}\).
\end{align}
\el
\br\label{Re001}\rm
By Lemma \ref{Eml}, we have $u\in L^\infty([0,T]; \bC^\gamma)$ for any $\gamma\in(1,2-2/q-d/p)$.
\er
\begin{proof}
We note that $u$ is a solution to PDE \eqref{Pre4:PDE} with
$u(0)=\varphi$ in the sense of Definition \ref{Pre4:Def}
if and only if $\bar{u}:=u-\varphi$ is a solution to PDE
\eqref{Pre4:PDE} with $\bar{u}(0)=0$ and
$f=f+a_{ij}\p_i\p_j\varphi-\lambda\varphi-b\cdot\nabla \varphi$.
Based on Lemma \ref{Eml}, we have
\begin{align*}
\nor f+a_{ij}\p_i\p_j\varphi-\lambda\varphi-b\cdot\nabla \varphi\nor_{\widetilde{\mL}^p_q(T)}&\lesssim \nor f\nor_{\widetilde{\mL}^p_q(T)}+ \nor \varphi\nor_{2,p}+ \nor b\nor_{\widetilde{\mL}^p_q(T)}\|\nabla\varphi\|_\infty\\
&\lesssim \nor f\nor_{\widetilde{\mL}^p_q(T)}+(1+ \nor b\nor_{\widetilde{\mL}^p_q(T)})\nor \varphi\nor_{2,p},
\end{align*}
and complete the proof by \cite[Theorem 3.2]{Xi-Xi-Zh-Zh}.
\end{proof}

With the help of a priori estimate \eqref{CCA02}, we obtain
the well-posedness of PDE \eqref{Pre4:PDE} for any $\lambda\geq0$.
\begin{proposition}\label{well-PDE}
Let $T>0$, $p,q\in(1,\infty)$ with $2/q+d/p<1$, $\lambda\geq0$, $b\in\widetilde{\mL}_q^p(T)$.
Then for all $f\in \widetilde{\mL}^{p}_{q}(T)$ and $\varphi\in \widetilde{H}^{2,p}$
there is a unique solution $u$ to PDE \eqref{Pre4:PDE} on $[0,T]$
in the sense of Definition \ref{Pre4:Def} such that
\begin{equation}\label{CCA022}
\nor\nabla u\nor_{\widetilde{\mL}_T^\infty}
+\nor\p_tu\nor_{\widetilde{\mL}^p_q(T)}
+\nor u\nor_{\widetilde{\mH}_q^{2,p}(T)}\leq
C\left(\nor f\nor_{\widetilde{\mL}_q^p(T)}
+\nor\varphi\nor_{\widetilde{H}^{2,p}}\right),
\end{equation}
where $C=C(\Theta,\lambda)$.
\end{proposition}
\begin{proof}
By the standard continuity method, it suffices to show a priori estimate \eqref{CCA022} for \eqref{Pre4:PDE}.
To this end, we rewrite \eqref{Pre4:PDE} as
\begin{equation*}%\label{rePDE}
\p_tu=a_{ij}\p_i\p_ju-(\lambda+\lambda_0)u+b\cdot\nabla u
+f+\lambda_0u,
\end{equation*}
where $\lambda_0$ is as in Lemma \ref{lem210}. In view of Lemma \ref{lem210}, we have
\begin{align}\label{RePDE1}
(\lambda+\lambda_0)^{\frac{\beta}{2}}\nor u\nor_{\widetilde{\mH}^{\alpha,p}_{q'}}
+\nor\p_tu\nor_{\widetilde{\mL}^p_q(T)}
+\nor\nabla^2u\nor_{\widetilde{\mL}_q^p(T)}\leq
C\left(\nor f\nor_{\widetilde{\mL}_q^p(T)}
+\lambda_0\nor u\nor_{\widetilde{\mL}_q^p(T)}
+\nor\varphi\nor_{2,p}\right),
\end{align}
where $C=C(\Theta,\alpha,q')>0$, $\beta=2-\alpha+\frac{2}{q'}-\frac{2}{q}>0$.
Taking $q'=\infty$ in \eqref{RePDE1}, then we have
\begin{align*}%\label{RePDE2}
(\lambda+\lambda_0)^{\frac{\beta}{2}}\sup_{t\in[0,T]}
\nor u(t)\nor_p\leq
C\left(\nor f\nor_{\widetilde{\mL}_q^p(T)}
+\lambda_0\left(\int_0^T\nor u(t)\nor_p^q\dif t\right)^{\frac{1}{q}}
+\nor\varphi\nor_{2,p}\right).
\end{align*}
Now it follows from Gronwall's Lemma that
\begin{align}\label{RePDE3}
\sup_{t\in[0,T]}\nor u(t)\nor_p\leq C\left(
\nor f\nor_{\widetilde{\mL}_q^p(T)}
+\nor\varphi\nor_{2,p}\right),
\end{align}
where $C$ depends on $\Theta,\alpha,\lambda$. Combining
\eqref{RePDE1} and \eqref{RePDE3}, we obtain for $1<\alpha<2-2/q$
\begin{align*}
\nor u\nor_{\widetilde{\mH}^{\alpha,p}}
+\nor\p_tu\nor_{\widetilde{\mL}^p_q(T)}
+\nor u\nor_{\widetilde{\mH}_q^{2,p}(T)}\leq
C\left(\nor f\nor_{\widetilde{\mL}_q^p(T)}
+\nor\varphi\nor_{2,p}\right).
\end{align*}
and complete the proof by Lemma \ref{Eml}.
\end{proof}

\subsection{Elliptic equation}
Now we consider the following second order elliptic PDE in $\mR^d$:
\begin{align}\label{Pre3:PDE}
a_{ij}\p_i\p_ju-\lambda u+b\cdot\nabla u=f,
\end{align}
where $\lambda\ge0$, $a=(a_{ij}):\R^d\rightarrow\R^d\otimes\R^d$
is a symmetric matrix-valued Borel measurable function satisfying {\bf(H$_a$)} and
$b:\R^d\rightarrow\R^d$ is a vector-valued Borel measurable function.
Firstly, we introduce the definition of a solution to PDE \eqref{Pre3:PDE}.
\bd\label{Pre3:Def}\rm
Let $p\in(1,\infty)$, $\lambda,T\ge0$ and $b,f\in \widetilde{L}^p$. We call
$u\in\widetilde{H}^{2,p}$ a {\em solution} of PDE \eqref{Pre3:PDE} if for Lebesgue almost all $x\in\mR^d$,
\begin{align*}
a_{ij}\p_i\p_ju(x)-\lambda u(x)+b(x)\cdot\nabla u(x)=f(x).
\end{align*}
\ed
As a corollary of Lemma \ref{lem210}, we have the following results.

\bl
Assume $b\in \widetilde{L}^{p}$ for some $p>d$. Then
there are constants $\lambda_0=\lambda_0(d,\nor b\nor_p,p,c_0,\theta)$
and $C=C(d,\nor b\nor_p,p,p',c_0,\theta)$ such that for any $\lambda\geq\lambda_0$
and $f\in\widetilde{L}^{p}$, there exists a unique solution
$u$ to PDE \eqref{Pre3:PDE} in the sense of Definition \ref{Pre3:Def} such that
\begin{align}\label{CC02}
\lambda^{\frac{\beta}{2}}
\nor u\nor_{\widetilde{H}^{\alpha,p'}}
+\nor\nabla^2 u\nor_{p}\le C\nor f\nor_{p},
\end{align}
where $\alpha\in[0,2),p\in[p,\infty], p'\in[p,\infty]$ with
$\beta:=2-\alpha+\frac{d}{p'}-\frac{d}{p}>0$.
\el
\begin{proof}
As usual, it suffices to show the a priori estimate \eqref{CC02}.
Let $T>0$, $u$ be a solution to \eqref{Pre3:PDE}
and $\phi$ be a nonnegative and nonzero smooth function on $[0,\infty)$ with $\phi(0)=0$.
Define $\tilde{u}(t,x):=\phi(t)u(x)$. Then, one sees that
$\tilde{u}$ is a solution to the following parabolic
equation in the sense of Definition \ref{Pre4:Def}:
\begin{align*}
\p_t\tilde{u}=a_{ij}\p_i\p_j\tilde{u}-\lambda\tilde{u}+b\cdot\nabla\tilde{u}-\phi f+\phi'u,\quad \tilde{u}(0)=0.
\end{align*}
By \eqref{CCA02}, we have for any $\alpha\in[0,2),p'\in[p,\infty]$ with
$\beta:=2-\alpha+\frac{d}{p'}-\frac{d}{p}>0$,
\begin{align*}
\lambda^{\frac{\beta}{2}}
\nor \tilde{u}\nor_{\widetilde{\mH}^{\alpha,p'}(T)}
+\nor\nabla^2 \tilde{u}\nor_{\widetilde{\mL}^p(T)}\le C\nor\phi' u-\phi f\nor_{\widetilde{\mL}^p(T)},
\end{align*}
which implies that
\begin{align*}
\lambda^{\frac{\beta}{2}}
\nor u\nor_{\widetilde{H}^{\alpha,p'}}
+\nor\nabla^2 u\nor_{p}\le C\|\phi\|^{-1}_\infty
\left(\|\phi'\|_\infty \nor u\nor_p+\|\phi\|_\infty \nor f\nor_p\right),
\end{align*}
where $\|\phi\|_\infty:=\sup_{t\in[0,T]}|\phi(t)|$. Noting that $\nor u\nor_p\le \nor u\nor_{p'}$,
%by taking $\lambda$ large enough,
we obtain \eqref{CC02} and complete the proof.

\end{proof}

\section{Analysis of time regularity}\label{Section3}
{ In this section,
 letting $T>0$, we assume that
 \begin{align}\label{NBB03}
 B\in \widetilde{\mL}^{p_0}(T)\text{ and $\sigma$ $:\mR^d\to\mR^d\otimes\mR^d$ satisfy {\bf (H$_\sigma$)} for some $p_0>d$}.
\end{align}
and consider the following SDE on a probability space
$(\Omega,\sF,(\sF_t)_{t\ge0}, \mP)$:
\begin{equation}\label{SZZ001}
X_{s,t}^x=x+\int_s^tB(r,X_{s,r}^x)\dif r
+\int_s^t\sigma(X_{s,r}^x)\dif W_r,
\end{equation}
where $W_t$ is a standard $d$-dimensional Brownian motion. Furthermore, consider the PDE on $[0,T]\times\mR^d$
\begin{align}\label{NBB00}
\p_tu=a_{ij}\p_i\p_ju-\lambda u+B\cdot\nabla u+f,\quad u(0)=\varphi,
\end{align}
where $\lambda\ge0$, $f\in \widetilde{\mL}^{p_0}(T)$, $\varphi\in C^\infty_b$ and $a_{ij}:=\frac{1}{2}\sum\limits_{k=1}^{d}\sigma_{ik}\sigma_{jk}$.
Under condition \eqref{NBB03}, by \cite{Xi-Xi-Zh-Zh}
there is a unique strong solution $X_{s,\cdot}^x$ to \eqref{SZZ001} for any $(s,x)\in\mR_+\times\mR^d$.
The purpose of this section is to obtain some moment estimates for the following functionals of $X_{0,t}^x$
\begin{align*}
\int_0^T f(s,X^x_{0,\pi_h(s)})\dif s\quad \text{and}\quad\int_0^T\left[f(s,X^x_{0,s},\mu^x_{s})-f(s,X^x_{0,\pi_h(s)},\mu^x_{\pi_h(s)}) \right]\dif s,
\end{align*}
where $\mu^x_t$ is the distribution of $X^x_{0,t}$ and $f\in \mL^p_q(T)$ for some $2/q+d/p<2$.
{The first integral is estimated by Krylov's type estimates.
Compared to the case of smooth coefficients in \cite{SXW2022}, $f$ in the second integral has no regularity.}
To overcome this obstacle, we use the observation mentioned in the introduction to replace
$f(X^x_{0,t})-f(X^x_{0,\pi_h(t)})$ by $P^X_{0,t}f-P^X_{0,\pi_h(s)}f$,
where $P^X$ is the transition semi-group of $X$. Hence, we only need to
obtain some time regularity results for the semigroup.

In Subsection \ref{sec:3.1}, we consider the time-homogeneous case with $B\equiv0$.
By Girsanov's Theorem, we extend the results in Subsection \ref{sec:3.1} to $X^x_{s,t}$ in Subsection \ref{Sec3.1}.
Moreover, we obtain additional time regularity estimates for $P^X$ by Duhamel's formula
which can not be gotten from Girsanov's theorem. In the light of Duhamel's formula again,
we also have two time regularity estimates for $\nabla u$ in Subsection \ref{SubS32},
where $u$ is the solution to \eqref{NBB00}.

For simplicity, throughout this section we set
\begin{align*}
\Xi:=(d,T,p_0,\nor B\nor_{\widetilde{{\mL}}^{p_0}(T)},\kappa_1,\beta).
\end{align*}

}

\subsection{Time regularity for solutions to SDE with no drift}\label{sec:3.1}
First of all, we recall the following generalized It\^o formula from
\cite[Lemma 4.1-iii)]{Xi-Xi-Zh-Zh} (see also \cite[Theorem 3.7]{KR05} for the original one).
\bl[Generalized It\^o formula]
Let $p,q\in[2,\infty)$ with $2/q+d/p<1$.
For any $T>0$ and any $u\in\widetilde{\mH}_q^{2,p}(T)$ with $\p_tu\in \widetilde{\mL}^p_q(T)$,
we have for any $t\in[s,T]$ and $x\in\mR^d$,
\begin{align}\label{GIF001}
\begin{split}
u(t,X^x_{s,t})=u(s,x)&+\int_s^t(\p_ru+a_{ij}\p_i\p_j u
+B\cdot\nabla u)(r,X_{s,r}(x))\dif r\\
&+\int_s^t \nabla u(r,X_{s,r}(x))\dif W_r.
\end{split}
\end{align}
\el

In this subsection, we consider the following case where $B\equiv0$:
\begin{equation}\label{SZZ00}
Z_t^x=x+\int_0^t\sigma(Z_s^x)\dif W_s.
\end{equation}
Define $P_t^\sigma f(x):=\mE f(Z_t^x)$. By Proposition \ref{well-PDE}, there is a unique solution to the following second order parabolic PDE on $[0,T]\times\mR^d$:
\begin{align}\label{PreZ:PDE}
\p_tu=a_{ij}\p_i\p_ju,\quad u_0=\varphi.
\end{align}

\begin{lemma}\label{PR-PDE:Z}
Assume {\bf (H$_\sigma$)} holds. Let $0\leq s\leq t$, $\varphi\in C_b^\infty(\R^d)$, $u$ and
$Z_{t}^x$ be the solution to \eqref{PreZ:PDE} and \eqref{SZZ00} respectively. Then we have $\mP$-a.s.
\begin{equation}\label{PR-PDE:Zeq}
\mE\left[\varphi(Z_t^x)|\sF_s\right]=u(t-s,Z_s^x).
\end{equation}
In particular,
\begin{align}\label{ChH00}
\mE\left[\varphi(Z_t^x)|\sF_s\right]=P^\sigma_{t-s}\varphi(Z_s^x)\quad
\mP-{\rm a.s.}
\end{align}
Moreover, for any $t\ge0$ and $f\in C^2_b$,
\begin{align}\label{NAA01}
P^\sigma_{t}f-f=\int_0^t P^\sigma_{r}(a_{ij}\p_i\p_jf)\dif r.
\end{align}
\end{lemma}
\begin{proof}
For all $t>0$, applying the generalized It\^o formula \eqref{GIF001}
to $s\mapsto u(t-s,Z_s^x)$, we have
\begin{align*}
u(0,Z_t^x)=u(t-s,Z_s^x)+\int_s^t\left(-\p_ru(t-r,Z_r^x)
+a_{ij}\p_i\p_ju(t-r,Z_r^x)\right)\dif r
+\int_s^t\nabla u(t-r,Z_r^x)\dif W_r.
\end{align*}
Noting that $u(t-s,Z_s^x)$ is $\sF_s$-measurable, and
taking conditional expectation with respect to $\sF_s$ on both sides, we have
\begin{align*}
\mE\left[\varphi(Z_t^x)|\sF_s\right]=u(t-s,Z_s^x)\quad
\mP-{\rm a.s.},
\end{align*}
which for $s=0$ implies that
\begin{align*}
P^\sigma_t\varphi(x)=\mE\varphi(Z_t^x)=u(t,x).
\end{align*}
Then \eqref{ChH00} is straightforward from \eqref{PR-PDE:Zeq}.
For \eqref{NAA01}, since $f\in C^2_b$, we use the classical It\^o formula and have
\begin{align*}
f(Z^x_{t})=f(x)+\int_0^t a_{ij}\p_i\p_j f(Z^x_{r})\dif r+\int_0^t \nabla f(Z^x_{r})\dif W_r.
\end{align*}
Then, we have \eqref{NAA01} by taking expectation and complete the proof.
\end{proof}

Based on Lemma \ref{PR-PDE:Z} and the uniqueness of \eqref{PreZ:PDE},
we have the following Chapman-Kolmogorov equations
\begin{align}\label{CKE00}
P^\sigma_{s}P^\sigma_{t}=P^\sigma_{s+t}.
\end{align}
Set
$$
\rho_t(x):=(2\pi t)^{-d/2}e^{-|x|^2/(2t)}.
$$
Then the following lemma is from \cite[Theorem 2.3]{CHXZ}.
\begin{lemma}\label{TDeqlem}
Assume {\bf (H$_\sigma$)} holds. Then there is a unique function
$p_\cdot^\sigma(\cdot,\cdot):\R_+\times\R^{2d}\rightarrow\R$ such that for any $j=0,1,2$
\begin{equation}\label{TDeq}
|\nabla_x^jp_t^\sigma(x,y)|\leq c_1t^{-\frac{j}{2}}\rho_{c_2t}(x-y)
\end{equation}
and
\begin{equation}\label{PR-RE}
  P_t^\sigma f(x)=\int_{\R^d}f(y)p_t^\sigma(x,y)\dif y
\end{equation}
for any $f\in C(\mR^d)$, where $c_1$ and $c_2$ are positive constants depending on $\Xi$.
\end{lemma}

For any $h\in(0,1)$, recall that $\pi_h(t):=t$ for $t\in[0,h)$ and
\begin{align*}
\pi_h(t):={kh},\ \ t\in[{kh},{(k+1)h}), ~k\ge1.
\end{align*}
\br
The reason why we define $\pi_h(t)=t$ for $t\in[0,h)$ is that the function space here is $L^p$.
If the initial data don't have an $L^q$ density, $\mE f(Z_{\pi_h(t)})=\mE f(Z_0)$, $f\in L^p$, will blow up for all $t<{h}$.
\er
Now we give the following Krylov estimate and Khasminskii estimate.
\bl
Assume {\bf (H$_\sigma$)} holds. For any $T>0$, $k=0,1,2$, $p\in[1,\infty]$ and $q\in[p,\infty]$, there is a constant $C=C(\Xi,p,q)$ such that for all $0\le s<t\le T$, $x\in\mR^d$ and nonnegative functions $f\in\widetilde{L}^p$
\begin{align}\label{BB00}
\nor \nabla^kP^\sigma_tf\nor_q\le Ct^{-k/2-d/(2p)+d/(2q)}\nor f\nor_{p}
\end{align}
and for $2/q+d/p<2$, $h>0$ and nonnegative functions $f\in\widetilde{\mL}^p_q(T)$
\begin{align}\label{Kry00}
\mE\int_s^t f(r,Z_r^x)\dif r+\mE\int_s^t f(r,Z^x_{\pi_h(r)})\dif r\le C(t-s)^{1-\frac{1}{q}-\frac{d}{2p}}
\nor f\nor_{\widetilde{\mL}^p_q(T)}.
\end{align}
Moreover, for any $f\in\widetilde{\mL}^p_q(T)$ with $d/p+2/q<2$,
\begin{align}\label{Kh00}
\sup_{x\in\mR^d}\mE\exp\left(\int_0^T f(t,Z^x_t)\dif t\right)<\infty.
\end{align}

\el
\begin{proof}
Without loss of generality we assume that $c_2=1$ in \eqref{TDeq}.
Combining Lemma \ref{TDeqlem} and Young's convolution inequality, one sees that
\begin{align*}
\nor\nabla^kP^\sigma_{t}f\nor_{q}
&
\lesssim \nor\int_{\R^d}f(y)\nabla_x^kp_t^\sigma(\cdot,y)\dif y\nor_q\\
&
\lesssim t^{-k/2}\nor\rho_{t}*f\nor_q
\lesssim t^{-k/2}\sup_w\|\1_{|\cdot-w|\le 1}\int_{\mR^d}f(\cdot-y)\rho_t(y)\dif y\|_q\\
&
\lesssim t^{-k/2}\sup_w\|\1_{|\cdot-w|\le 1}
\frac{1}{|B_1|}\int_{\mR^d}\int_{\mR^d}\1_{|y-z|\le 1}
f(\cdot-y)\1_{|y-z|\le1}\rho_t(y)\dif y\dif z\|_q\\
&
\lesssim t^{-k/2} \sup_w\|\int_{\mR^d}\int_{\mR^d}\1_{|\cdot-y-w+z|\le 2}f(\cdot-y)\1_{|y-z|\le1}\rho_t(y)\dif y\dif z\|_q\\
&
\lesssim t^{-k/2}\sup_w \int_{\mR^d}\|\int_{\mR^d}\1_{|\cdot-y-w+z|\le 2}f(\cdot-y)\1_{|y-z|\le1}\rho_t(y)\dif y\|_q\dif z\\
&
\lesssim t^{-k/2}\int_{\mR^d}\sup_w\|\1_{|\cdot-w+z|\le 2}f(\cdot)\|_p\|\1_{|\cdot-z|\le1}\rho_t(\cdot)\|_r\dif z\\
&
\lesssim  t^{-k/2} \int_{\mR^d}\left(\int_{|y-z|\le1}|\rho_t(y)|^{r}\dif y\right)^{1/r}\dif z\nor f\nor_p,
\end{align*}
where $1+1/q=1/r+1/p$.
Next, one sees that
\begin{align*}
\int_{\mR^d}\left(\int_{|y-z|\le1}(\rho_t(y))^{r}\dif y\right)^{1/r}\dif z
&
\lesssim \|\rho_t\|_{r}+\int_{|z|>2}\left(\int_{|y-z|\le1}(\rho_t(y))^{r}\dif y\right)^{1/r}\dif z.\\
\end{align*}
We note that
\begin{align*}
|z|>2,|y-z|\le1\Rightarrow |y|\ge|z|-|y-z|\ge\frac{|z|}{2},
\end{align*}
which implies that $\rho_t(y)\lesssim \rho_{t}(z/2)$, and
$\|\rho_t\|_{r}\lesssim t^{-d/2+d/(2r)}=t^{-d/(2p)+d/(2q)}$, we have
\begin{align*}
\int_{\mR^d}\left(\int_{|y-z|\le1}(\rho_t(y))^{r}\dif y\right)^{1/r}\dif z
&
\lesssim \|\rho_t\|_{r}+ \int_{|z|>2}\left(\int_{|y-z|\le1}(\rho_t(z/2))^{r}\dif y\right)^{1/r}\dif z\\
&
\lesssim t^{-d/(2p)+d/(2q)}+ \int_{|z|>2}\rho_t(z/2)\dif z\lesssim t^{-d/(2p)+d/(2q)}+1
\end{align*}
and obtain \eqref{BB00}.

Now we show \eqref{Kry00}. Set $p':=\frac{p}{p-1}$ and $q':=\frac{q}{q-1}$. Without loss of generality, we take $s=0$.
By \eqref{BB00}, for any $h>0$,
\begin{align*}
\mE\int_0^t f(s,Z^x_{\pi_h(s)})\dif s\lesssim \int_{\mR^d}\Big(\int_0^t\(\int_{|y-z|\le1}|\rho_{\pi_h(s)}(y)|^{p'}\dif y\)^{q'/p'}\dif s\Big)^{1/q'}\dif z\nor f\nor_{\widetilde{\mL}^p_q(T)}.
\end{align*}
Then, we have
\begin{align*}
\sI&:=\int_{\mR^d}\Big(\int_0^t\(\int_{|y-z|\le1}|\rho_{\pi_h(s)}(y)|^{p'}\dif y\)^{q'/p'}\dif s\Big)^{1/q'}\dif z\\
&\lesssim \(\int_0^t\|\rho_{\pi_h(s)}\|_{p'}^{q'}\dif s\)^{1/q'}+\int_{|z|>2}\Big(\int_0^t|\rho_{\pi_h(s)}(z/2)|^{q'}\dif s\Big)^{1/q'}\dif z\\
&\lesssim  \(\int_0^t(\pi_h(s))^{-dq'/(2p)}\dif s\)^{1/q'}+t\int_{|z|>2}|z|^{-d}\exp(-|z|^2/(16T))\dif z\\
&\lesssim t^{1-1/q-d/(2p)}+t,
\end{align*}
since $dq'/(2p)<1$ and $\rho_s(z)\le C|z|^{-d}\exp(-\frac{|z|^2}{4T})$ for all $s<T$.
Similarly, we obtain
$$
\mE\int_0^tf(s,Z_s^x)\dif s\lesssim t^{1-\frac1q-\frac{d}{2p}}+t.
$$

Finally, noting that by \eqref{ChH00}
\begin{align*}
\mE\Big[\int_s^t f(s,Z_r^x)\dif r\Big|\sF_s\Big]=\mE\int_s^t f(s,Z_{r-s}^y)\dif r\Big|_{y=Z_s^x},
\end{align*}
 \eqref{Kh00} is direct from \eqref{Kry00} (see \cite[Corollary 3.5]{Zh-Zh1} for example) and we complete the proof.

\end{proof}

\br
We note that for any fixed $h>0$ and $x\in\mR^d$
\begin{align*}
\mE\exp\(\int_0^T f(t,Z^x_{\pi_h(t)})\dif t\)<\infty
\end{align*}
is not true. For example, letting $Z^x_t=W_t$, when $d\ge2$ and $f(t,x)=|x|^{-1/2}\in \widetilde{L}^{d+1}(\mR^d)$, we have
\begin{align*}
&\quad\mE\exp\left(\int_{h}^{2h} f(t,W_{\pi_h(t)})\dif t\right)=\mE\exp(h|W_{h}|^{-1/2})\\
&=\int_{\mR^d}e^{\frac{h}{\sqrt{|x|}}}\rho_h(x)\dif x\ge\frac{h^{2d}}{(2d)!}\int_{\mR^d}\frac{1}{|x|^d}\rho_h(x)\dif x=\infty.
\end{align*}

\er

\begin{lemma}\label{lemmm03}
Let $T>0$, $k=0,1$, and $1\le p\le q\le\infty$. Then there is a constant $C=C(\Xi,p,q)$ such that for all $0< s\le t \le T$,
\begin{equation}\label{MM03}
\nor \nabla^k(P^\sigma_t\varphi-P^\sigma_s\varphi)\nor_q\le C\Big([(t-s)^{\frac{2-k}{2}}s^{\frac{k-2}{2}}]\wedge1\Big)s^{-\frac{k}{2}-\frac{d}{2p}
+\frac{d}{2q}}\nor \varphi\nor_p.
\end{equation}

\end{lemma}

\begin{proof}
Based on \eqref{CKE00} and \eqref{NAA01}, one sees that
\begin{align*}
P^\sigma_t\varphi-P^\sigma_s\varphi=P^\sigma_{t-s}(P^\sigma_s\varphi)-P^\sigma_s\varphi=\int_0^{t-s}P^\sigma_r(a_{ij}\p_i\p_jP^\sigma_s\varphi)\dif r.
\end{align*}
By \eqref{BB00}, we have
\begin{align*}
\nor \nabla^k(P^\sigma_t\varphi-P^\sigma_s\varphi)\nor_q
&
\lesssim\int_0^{t-s}r^{-\frac{k}{2}}\nor\nabla^2P_s^\sigma\varphi\nor_q\dif r\\
&
\lesssim\int_0^{t-s}r^{-\frac{k}{2}}s^{-1+\frac{d}{2q}-\frac{d}{2p}}\nor\varphi\nor_p\dif r\\
&
\lesssim \Big[(t-s)^{\frac{2-k}{2}}s^{\frac{k-2}{2}}\Big]s^{-\frac{k}{2}-\frac{d}{2p}
+\frac{d}{2q}}\nor \varphi\nor_p.
\end{align*}

Moreover, noting that by \eqref{BB00}
\begin{align*}
\nor \nabla^k(P^\sigma_t\varphi-P^\sigma_s\varphi)\nor_q&\le \nor \nabla^kP^\sigma_t\varphi\nor_q+\nor\nabla^k P^\sigma_s\varphi\nor_q\\
&\lesssim s^{-\frac{k}{2}-\frac{d}{2p}
+\frac{d}{2q}}\nor \varphi\nor_p
\end{align*}
for $s\le t$, we complete the proof.
\end{proof}

When $p=\infty$ and $\sigma\equiv\mI$, the following lemma has been proved in \cite[Lemma 2.1]{DG20} for Brownian motion.
For classical $L^p$ spaces, L\^e and Ling obtained these results by the stochastic sewing lemma in \cite{LL21}.
For the localized $L^p$ space, we provide a different proof here,
which is based on \eqref{MM03} and \eqref{BB00}.
\bl\label{CruL}
Assume {\bf (H$_\sigma$)} holds. Then for any $T>0$ and $p\in(d\vee 2,\infty)$, there is a constant $C=C(\Xi,p)$ such that for any stopping time $\tau\le T$,  $h\in(0, 1)$, $x\in\mR^d$ and $f\in\widetilde{\mL}^p(T)$,
\begin{align}\label{BB03}
\mE\Big|\int_{0}^\tau(f(t,Z_t^x)-f(t,Z^x_{\pi_h(t)}))\dif t\Big|^2
\le C h\ln h^{-1}\nor f\nor_{\widetilde{\mL}^p(T)}^2.
\end{align}

\el
\begin{proof}
We divide the proof into four steps. In Step 1, we prove that
\begin{align}\label{MM00}
\mE\Big|\int_{s}^t(f(r,Z^x_r)-f(r,Z^x_{\pi_h(r)}))\dif r\Big|^2
\le C \(h s^{-\frac{d}{2p}}(t-s)^{1-\frac{d}{2p}}+ h\ln h^{-1}(t-s)^{1-\frac{d}{p}}\)\nor f\nor_{\widetilde{\mL}^p(T)}^2;
\end{align}
 In Step 2, we show \eqref{BB03} for $\tau=T$; In Step 3, we show
\begin{align}\label{DD000}
\sup_{a,b\in[0,T]}\mE\left(\int_{a}^b(f(t,Z^x_t)-f(t,Z^x_{\pi_h(t)}))\dif t\right)^2
\le C h\ln h^{-1}\nor f\nor_{\widetilde{\mL}^p(T)}^2;
\end{align}
In Step 4, we show \eqref{BB03} for any stopping time $\tau$ with $\tau\le T$.\\
\vspace{2mm}

{\bf (Step 1)}
{{For simplicity of notation}, we drop the time $t$ in $f(t,x)$.}
First, we note that by H\"older's inequality and \eqref{Kry00},
\begin{align*}
&\quad\mE\left(\int_0^{{2h}}f(Z^x_t)-f(Z^x_{\pi_h(t)})\dif t\right)^2\le2h\mE\int_0^{{2h}}\Big|f(Z^x_t)-f(Z^x_{\pi_h(t)})\Big|^2\dif t\lesssim h\nor f\nor_p.
\end{align*}
Hence, without loss of generality, we may assume $s>2h$.
The symmetry implies
\begin{align*}
&
\quad\mE\Big|\int_{s}^t(f(Z^x_r)-f(Z^x_{\pi_h(r)}))\dif r\Big|^2\\
=&2\int_{s}^t\int_{{r_1}}^t \mE\[(f(Z^x_{r_1})-f(Z^x_{\pi_h(r_1)}))(f(Z^x_{r_2})-f(Z^x_{\pi_h(r_2)}))\]\dif r_2\dif r_1\\
=&2\int_{s}^t\int_{{r_1}}^{r_1+h} \mE\[(f(Z^x_{r_1})-f(Z^x_{\pi_h(r_1)}))(f(Z^x_{r_2})-f(Z^x_{\pi_h(r_2)}))\]\dif r_2\dif r_1\\
&+2\int_{s}^t\int_{{r_1+h}}^t \mE\[(f(Z^x_{r_1})-f(Z^x_{\pi_h(r_1)}))(f(Z^x_{r_2})-f(Z^x_{\pi_h(r_2)}))\]\dif r_2\dif r_1\\
:=&2\sI_1+2\sI_2.
\end{align*}
By H\"older's inequality and \eqref{BB00}, one sees that
\begin{align*}
&
\quad\mE\[(f(Z^x_{r_1})-f(Z^x_{\pi_h(r_1)}))(f(Z^x_{r_2})-f(Z^x_{\pi_h(r_2)}))\]\\
&
\le \Big(\mE\Big|f(Z^x_{r_1})-f(Z^x_{\pi_h(r_1)})\Big|^2\Big)^{1/2}\Big(\mE\Big|f(Z^x_{r_2})-f(Z^x_{\pi_h(r_2)})\Big|^2\Big)^{1/2}\\
&
\lesssim (r_1^{-\frac{d}{2p}}+{\pi_h(r_1)}^{-\frac{d}{2p}})(r_2^{-\frac{d}{2p}}+{\pi_h(r_2)}^{-\frac{d}{2p}})\nor f\nor_p^2,
\end{align*}
which implies that
\begin{align*}
\sI_1&\lesssim\nor f\nor_p^2 \int_{s}^t\int_{{r_1}}^{r_1+h}(r_1^{-\frac{d}{2p}}+{\pi_h(r_1)}^{-\frac{d}{2p}})(r_2^{-\frac{d}{2p}}+{\pi_h(r_2)}^{-\frac{d}{2p}})\dif r_2\dif r_1\\
&\lesssim \nor f\nor_p^2(s-h)^{-\frac{d}{2p}} h \int_{s}^t(r_1^{-\frac{d}{2p}}+{\pi_h(r_1)}^{-\frac{d}{2p}})\dif r_1\lesssim hs^{-\frac{d}{2p}} (t-s)^{1-\frac{d}{2p}}\nor f\nor_p^2.
\end{align*}

For $\sI_2$, we use conditional expectation and the Markov property \eqref{ChH00}. For simplicity, let
\begin{align*}
\mE^{\sG}[\cdot]:=\mE[\cdot|\sG].
\end{align*}
Then, noting that $r_1\le r_2-h<r_2$, by the Markov property \eqref{ChH00}, we have
\begin{align*}
\sI_2&=\int_{s}^t\int_{{r_1+h}}^t \mE\Big((f(Z^x_{r_1})-f(Z^x_{\pi_h(r_1)}))\mE^{\sF_{r_1}}[f(Z^x_{r_2})-f(Z^x_{\pi_h(r_2)})]\Big)\dif r_2\dif r_1\\
&=\int_{s}^t\int_{{r_1+h}}^t \mE\Big(\[f(Z^x_{r_1})-f(Z^x_{\pi_h(r_1)})\]\[P^\sigma_{r_2-r_1}f(Z^x_{r_1})-P^\sigma_{\pi_h(r_2)-r_1}f(Z^x_{r_1})\]\Big)\dif r_2\dif r_1.
\end{align*}
Setting $G:=P^\sigma_{r_2-r_1}f-P^\sigma_{\pi_h(r_2)-r_1}f$, in view of H\"older's inequality and \eqref{BB00}, one sees that
\begin{align*}
\sI_2&\le \int_{s}^t\int_{{r_1+h}}^t \Big(\mE\Big|f(Z^x_{r_1})-f(Z^x_{\pi_h(r_1)})\Big|^2\Big)^{1/2}\Big(\mE |G(Z^x_{r_1})|^2\Big)^{1/2}\dif r_2\dif r_1\\
&\lesssim \nor |f|^2\nor_{p/2}^{1/2}\int_{s}^t\int_{{r_1+h}}^t ({r_1}^{-\frac{d}{2p}}+{\pi_h(r_1)}^{-\frac{d}{2p}}){r_1}^{-\frac{d}{2p}}\nor |G|^2\nor_{p/2}^{1/2}\dif r_2\dif r_1\\
&\lesssim \nor f\nor_{p}\int_{s}^t\int_{{r_1+h}}^t {(r_1-h)}^{-\frac{d}{p}}\nor G\nor_{p}\dif r_2\dif r_1.
\end{align*}
We note that by \eqref{MM03},
\begin{align*}
\nor G\nor_{p}&=\nor P^\sigma_{r_2-r_1}f-P^\sigma_{\pi_h(r_2)-r_1}f\nor_p\\
&\lesssim \Big([(r_2-\pi_h(r_2))(\pi_h(r_2)-r_1)^{-1}]\wedge 1\Big)\nor f\nor_p\\
&\lesssim \Big[(h(r_2-r_1-h)^{-1})\wedge 1\Big]\nor f\nor_p,
\end{align*}
By a change of variables we have
\begin{align*}
\sI_2
&
\lesssim \nor f\nor_{p}^2\int_{s}^t\int_{{r_1+h}}^t {(r_1-h)}^{-\frac{d}{p}}\Big[(h(r_2-r_1-h)^{-1})\wedge 1\Big]\dif r_2\dif r_1\\
&
\lesssim\nor f\nor_{p}^2\int_{s-h}^{t-h}\int_{0}^{t} {(r_1)}^{-\frac{d}{p}}\Big[(h (r_2)^{-1})\wedge 1\Big]\dif r_2\dif r_1\\
&
\lesssim h(t-s)^{1-\frac{d}{p}}\nor f\nor_{p}^2\int_{0}^{t/h}\Big[(r_2)^{-1}\wedge 1\Big]\dif r_2\lesssim h\ln h^{-1}(t-s)^{1-\frac{d}{p}} \nor f\nor_{p}^2,
\end{align*}
and we obtain \eqref{MM00}.

{\bf (Step 2)} In this step, we use the method in \cite[Lemma 3.4]{LL21} to show \eqref{BB03} with $\tau=T$. Let $\alpha:=\frac{d}{4p}$, $\|\cdot\|:=\(\mE|\cdot|^2\)^{\frac12}$,
\begin{align*}
F_t:=\int_{0}^t(f(r,Z^x_r)-f(r,Z^x_{\pi_h(r)}))\dif r\quad \text{and} \quad \cH:=\(C h\ln h^{-1}\)^{1/2}\nor f\nor_{\widetilde{\mL}^p(T)},
\end{align*}
where $C$ is the constant in \eqref{MM00}.
By \eqref{MM00}, we have
\begin{align}\label{MM01}
\|F_t-F_s\|\le \cH (s^{-\alpha}(t-s)^{1/2-\alpha}+(t-s)^{1/2-2\alpha}).
\end{align}
For any $n\in\mN$, set
\begin{align*}
t_n:=2^{-n}T.
\end{align*}
Then, by \eqref{MM01}
\begin{align*}
\|F_T\|\le\sum_{n=0}^\infty\|F_{t_n}-F_{t_{n+1}}\|\le\cH\sum_{n=0}^\infty (t_{n+1}^{-\alpha}(t_{n}-t_{n+1})^{1/2-\alpha}+(t_n-t_{n+1})^{1/2-2\alpha}).
\end{align*}
Noting that $t_{n+1}=T2^{-n-1}$ and $t_n-t_{n+1}=T2^{-n-1}$, we have
\begin{align*}
\|F_T\|
&
\le\cH\sum_{n=0}^\infty (T^{-\alpha}2^{\alpha(n+1)}(t_{n}-t_{n+1})^{1/2-\alpha}
+(t_n-t_{n+1})^{1/2-2\alpha})\\
&
\le 2\cH T^{1/2-2\alpha}\sum_{n=0}^\infty 2^{-(1/2-2\alpha)(n+1)}\le 2\cH T^{1/2-2\alpha}
\end{align*}
and obtain \eqref{BB03} for $\tau=T$.

{\bf (Step 3)}
For any $a,b\in[0,T]$, define
\begin{align*}
f_{a,b}(t,x):=\1_{t\in[a,b]}f(t,x).
\end{align*}
 Then, by Step 1, one sees that
 \begin{align*}
\mE\left(\int_{a}^b(f(t,Z^x_t)-f(t,Z^x_{\pi_h(t)}))\dif t\right)^2
&
=\mE\(\int_{0}^T(f_{a,b}(t,Z^x_t)-f_{a,b}(t,Z^x_{\pi_h(t)}))\dif t\)^2\\
&
\le C h\ln h^{-1}\nor f_{a,b}\nor_{\widetilde{\mL}^p(T)}^2.
\end{align*}
Noting that $\nor f_{a,b}\nor_{\widetilde{\mL}^p(T)}\le \nor f\nor_{\widetilde{\mL}^p(T)}$, we get \eqref{DD000}.

{\bf (Step 4)}
Without loss of generality, we assume that $h<T/2$ and that $\tau$ only takes finite values $a_1,a_2,...,a_n\in[0,T]$.
Otherwise, for any stopping time, we choose $\tau_n, n\in\N$, which only take finite values to approximate $\tau$
and \eqref{BB03} follows from \eqref{Kry00} and the dominated convergence theorem. First, we have
\begin{align*}
\mE\left|\int_{\tau}^T(f(t,Z^x_t)-f(t,Z^x_{\pi_h(t)}))\dif t\right|^2
&
\le 2\mE\left|\int_{\tau}^{(\tau+2h)\wedge T}(f(t,Z^x_t)-f(t,Z^x_{\pi_h(t)}))\dif t\right|^2\\
&\quad
+2\mE\left|\int_{(\tau+2h)\wedge T}^T(f(t,Z^x_t)-f(t,Z^x_{\pi_h(t)}))\dif t\right|^2.
\end{align*}
We note that by H\"older's inequality and \eqref{Kry00} for $p/2>d/2$,
\begin{align*}
&
\quad\mE\Big|\int_{\tau}^{(\tau+2h)\wedge T}(f(t,Z^x_t)-f(t,Z^x_{\pi_h(t)}))\dif t\Big|^2\\
&
\le \mE\Big|\int_{0}^{T}\1_{t\in[\tau,\tau+2h]}\dif t\int_{0}^{T}|f(t,Z^x_t)-f(t,Z^x_{\pi_h(t)})|^2\dif t\Big|\\
&
\lesssim h \nor |f|^2\nor_{\widetilde{\mL}^{p/2}(T)}\lesssim h \nor f\nor_{\widetilde{\mL}^{p}(T)}^2.
\end{align*}
Now we estimate the second term. In fact,
\begin{align*}
&\quad
\mE\left(\int_{(\tau+2h)\wedge T}^T(f(t,Z^x_t)-f(t,Z^x_{\pi_h(t)}))\dif t\right)^2\\
&
=\sum_{i=1}^n\mE\left[\1_{\tau=a_i}\left(\int_{(a_i+2h)\wedge T}^T(f(t,Z^x_t)-f(t,Z^x_{\pi_h(t)}))\dif t\right)^2\right].
\end{align*}
We note that $\1_{\tau=a_i}\in\sF_{a_i}\subset\sF_{([a_i/h]+1)h}$, without loss of generality, assuming $a_i+2h<T$, one sees that
\begin{align*}
&\quad
\mE\left[\1_{\tau=a_i}\left(\int_{a_i+2h}^T(f(t,Z^x_t)-f(t,Z^x_{\pi_h(t)}))\dif t\right)^2\right]\\
&
=\mE\Big(\1_{\tau=a_i}\mE\Big[\(\int_{a_i+2h}^T(f(t,Z^x_t)-f(t,Z^x_{\pi_h(t)}))\dif t\)^2\Big|\sF_{a_i}\Big]\Big)\\
&
:=\mE(\1_{\tau=a_i}\sA_i),
\end{align*}
where
\begin{align*}
\sA_i=\mE\left[\mE\left[\left(\int_{a_i+2h}^T(f(t,Z^x_t)-f(t,Z^x_{\pi_h(t)}))\dif t\right)^2\Big|\sF_{([a_i/h]+1)h}\right]\Big|\sF_{a_i}\right].
\end{align*}
Moreover, by the Markov property \eqref{ChH00}, we have
\begin{align*}
&\quad
\mE\Big[\(\int_{a_i+2h}^T(f(t,Z^x_t)-f(t,Z^x_{\pi_h(t)}))\dif t\)^2\Big|\sF_{([a_i/h]+1)h}\Big]\\
&
=\mE\Big(\int_{a_i+2h}^T(f(t,Z^y_{t-([a_i/h]+1)h})-f(t,Z^y_{\pi_h(t)-([a_i/h]+1)h}))\dif t\Big)^2\Big|_{y=Z^x_{_{([a_i/h]+1)h}}}\\
&
=\mE\Big(
\int_{a_i+2h-([a_i/h]+1)h}^{T-([a_i/h]+1)h}
(f(t+([a_i/h]+1)h,Z^y_{t})-f(t+([a_i/h]+1)h,Z^y_{\pi_h(t)}))\dif t\Big)^2\Big|_{y=Z^x_{_{([a_i/h]+1)h}}}\\
&
\lesssim h\ln h^{-1}\nor f\nor_{\widetilde{\mL}^p(T)}^2
\end{align*}
by \eqref{DD000}. Therefore, we have
\begin{align*}
\mE\Big|\int_{\tau}^T(f(t,Z^x_t)-f(t,Z^x_{\pi_h(t)}))\dif t\Big|^2
&
\lesssim h \nor f\nor_{\widetilde{\mL}^p(T)}^2+\sum_{i=1}^n\mE(\1_{\tau=a_i}\sA_i)\\
&
\lesssim  h \nor f\nor_{\widetilde{\mL}^p(T)}^2+h\ln h^{-1}
\nor f\nor_{\widetilde{\mL}^p(T)}^2\sum_{i=1}^n\mE\1_{\tau=a_i}\\
&
\lesssim h\ln h^{-1}\nor f\nor_{\widetilde{\mL}^p(T)}^2,
\end{align*}
which implies that
\begin{align*}
&\quad
\mE\left|\int_0^{\tau}(f(t,Z^x_t)-f(t,Z^x_{\pi_h(t)}))\dif t\right|^2\\
&
\le 2\mE\left|\int_{0}^T(f(t,Z^x_t)-f(t,Z^x_{\pi_h(t)}))\dif t\right|^2
+2\mE\left|\int_{\tau}^T(f(t,Z^x_t)-f(t,Z^x_{\pi_h(t)}))\dif t\right|^2\\
&
\lesssim h \ln h^{-1}\nor f\nor_{\widetilde{\mL}^p(T)}^2
\end{align*}
and completes the proof.

\end{proof}

\bc
Assume {\bf (H$_\sigma$)} holds. For any $T>0$, $p\in(d\vee 2,\infty)$ and $\delta>0$, there is a constant $C=C(\Xi,p)$ such that for any $x\in\mR^d$ and $f\in\widetilde{\mL}^p(T)$,
\begin{align}\label{DD03}
\mE\left(\sup_{t\in[0,T]}\Big|\int_{0}^t(f(s,Z^x_s)-f(s,Z^x_{\pi_h(s)}))\dif s\Big|^2\right)
\le C h^{1-\delta}\nor f\nor_{\widetilde{\mL}^p(T)}^2.
\end{align}

\ec
\begin{proof}
Let
\begin{align*}
\eta_t:=\Big|\int_{0}^t(f(s,Z^x_s)-f(s,Z^x_{\pi_h(s)}))\dif s\Big|^2
\end{align*}
and $\eta^*_t:=\sup\limits_{s\in[0,t]}\eta_s$.
First of all, it follows from H\"older's inequality and \eqref{Kry00} that for any $\gamma\in(1,p/(d\vee 2))$,
\begin{align}
\mE(\eta^*_T)^{\gamma}
&
\lesssim \mE\int_{0}^T\left(|f|^{2\gamma}(s,Z^x_s)+|f|^{2\gamma}(s,Z^x_{\pi_h(s)})\right)\dif s\no\\
&
\lesssim \nor |f|^{2\gamma}\nor_{\widetilde{\mL}^{p/(2\gamma)}(T)}
\lesssim \nor f\nor_{\widetilde{\mL}^p(T)}^{2\gamma}.\label{DD00}
\end{align}
For any $\lambda>0$, let
\begin{align*}
\tau_\lambda:=\inf\{t\ge0,~\eta_t>\lambda\}.
\end{align*}
We note that $\eta_{\tau_\lambda}=\lambda$, since $\eta$ is a continuous process. Then,
\begin{align*}
\lambda\mP(\eta^*_T>\lambda)\le\lambda\mP(\tau_\lambda\le T)
\le \mE\(\eta_{\tau_\lambda}\1_{\{\tau_\lambda\le T\}}\)\le \mE\eta_{\tau_\lambda\wedge T}.
\end{align*}
In view of \eqref{BB03}, we have
\begin{align*}
\lambda\mP(\eta^*_T>\lambda)\lesssim h\ln h^{-1}\nor f\nor_{\widetilde{\mL}^p(T)}^2.
\end{align*}
Set $\Xi_h:=h\ln h^{-1}\nor f\nor_{\widetilde{\mL}^p(T)}^2$.
Then, for any $\delta\in(0,1)$, by a change of variables,
\begin{align}
\mE\(\eta^*_T\)^{1-\delta}
&
=(1-\delta)\int_0^\infty\lambda^{-\delta}
\mP(\eta^*_T>\lambda)\dif\lambda\no\\
&
\lesssim \int_0^\infty\lambda^{-\delta}
\(1\wedge(\Xi_h\lambda^{-1})\)\dif\lambda\no\\
&
\lesssim \Xi_h^{1-\delta}\int_0^\infty \lambda^{-\delta}(1\wedge \lambda^{-1})\dif\lambda
\lesssim (h\ln h^{-1})^{(1-\delta)}
\|f\|_{\widetilde{\mL}^p(T)}^{2(1-\delta)}.\label{DD00000}
\end{align}
Combining \eqref{DD00} and \eqref{DD00000}, in view of H\"older's inequality, for any $\delta>0$ small enough, we have
\begin{align*}
\mE\eta^*_T
&
=\mE\left[\(\eta^*_T\)^{1-\delta-\sqrt{\delta}}\(\eta^*_T\)^{\delta+\sqrt{\delta}}\right]\\
&
\lesssim \left[\mE\(\eta^*_T\)^{1-\delta}\right]^{\frac{1-\delta-\sqrt{\delta}}{1-\delta}}
\Big[\mE\(\eta^*_T\)^{(\sqrt{\delta}+1)(1-\delta)}\Big]^{\frac{\sqrt{\delta}}{1-\delta}}\\
&
\lesssim h^{1-\delta-\sqrt{\delta}}(\ln h^{-1})^{1-\delta-\sqrt{\delta}}\nor f\nor_{\widetilde{\mL}^p(T)}^{2}
\end{align*}
and complete the proof.

\end{proof}

\subsection{Time discretization for SDEs with $\widetilde{\mL}^p$ drift}\label{Sec3.1}
Now, let us extend estimate \eqref{DD03} from $Z^x_{t-s}$ to the solution $X^x_{s,t}$ of
SDE \eqref{SZZ001} both in the sense of paths and distributions (see \eqref{BB05} and \eqref{EE00} below).
Recall that
\begin{align}\label{II03}
X^x_{s,t}=x+\int_s^t B(r,X^x_{s,r})\dif r+\int_s^t\sigma(X^x_{s,r})\dif W_r.
\end{align}
Let $\mu_{s,t}^x$ denote the distribution of $X^x_{s,t}$. For simplicity, we also set
\begin{align*}
X^x_t:=X^x_{0,t}\quad \text{and}\quad\mu_t^x:=\mu_{0,t}^x.
\end{align*}
The following estimates follow from Girsanov's transform and estimates for $Z^x_{t}$ (see \cite[Lemma 4.1]{Xi-Xi-Zh-Zh} for (ii) and (iii)).
\bl\label{Lem34}
Assume \eqref{NBB03} with $p_0>d\vee 2$.\\
(i) For any $T>0$ and $p\in(1,\infty)$, there is a constant $C=C(\Xi,p)$
such that for any $x\in\mR^d$, $0\le s<t\le T$ and nonnegative $f\in \widetilde{L}^p$,
\begin{align}\label{BB06}
\mE f(X^x_{t})\le C t^{-d/(2p)}\nor f\nor_{p}.
\end{align}
(ii)For any $T>0$ and $2/q+d/p<2$, there is a constant $C=C(\Xi,p,q)$ such that for any $x\in\mR^d$, $0\le s<t\le T$ and nonnegative $f\in \widetilde{\mL}^p_q(T)$,
\begin{align}\label{Kry02}
\mE\int_s^t f(r,X^x_{r})\dif r+\mE\int_{s}^t f(r,X^x_{\pi_h(r)})\dif r\le C(t-s)^{1-\frac{1}{q}-\frac{d}{2p}}\nor f\nor_{\widetilde{\mL}^p_q(T)}.
\end{align}
(iii)For any $T>0$, $d/p+2/q<2$ and $f\in \widetilde{\mL}^p_q(T)$,
\begin{align}\label{Kh02}
\sup_{x\in\mR^d}\mE\exp\left(\int_0^T f(t,X^x_t)\dif t\right)<\infty.
\end{align}
(iv)For any $T>0$, $\delta>0$ and $p>d\vee2$, there is a constant $C=C(\Xi,p)$ such that for any $x\in\mR^d$, $h\in(0,1)$ and $f\in \widetilde{\mL}^p(T)$,
\begin{align}\label{BB05}
\mE\left(\sup_{t\in[0,T]}\Big|\int_{0}^t(f(s,X^x_{s})-f(s,X^x_{\pi_h(s)}))\dif s\Big|^2\right)
\le C h^{1-\delta}\nor f\nor_{\widetilde{\mL}^p(T)}^2.
\end{align}
\el
\begin{proof}
Let $\widetilde{Z}^x$ be a solution on a probability space $(\widetilde{\Omega},\widetilde{\sF},(\widetilde{\sF}_t)_{t\ge0},\widetilde{\mP})$ to the following SDE
\begin{align*}
\widetilde{Z}^x_t=x+\int_0^t\sigma(\widetilde{Z}^x_r)\dif \widetilde{W}_r,
\end{align*}
where $\widetilde{W}_t$ is a standard $d$-dimensional Brownian motion.
Since $p_0>2\vee d$, $B^2\in \widetilde{\mL}^{p_0/2}(T)$. By \eqref{Kh00}, one sees that for any $\gamma>0$
\begin{align*}
%\sup_{\eps\ge0}
\sup_x\widetilde{\mE}_{\widetilde{\mP}}\exp\left(\gamma\int_0^T |\sigma^{-1}B(t,\widetilde{Z}^x_t)|^2\dif t\right)<\infty.
\end{align*}
Hence, by Novikov's criterion,
\begin{align*}
\mZ_T:= \exp\left(-\int_0^T \sigma^{-1}B(t,\widetilde{Z}^x_t)\dif \widetilde{W}_t-\frac{1}{2}\int_0^T|\sigma^{-1}B(t,\widetilde{Z}^x_t)|^2\dif t\right)
\end{align*}
is integrable and for any $q>0$
\begin{align}\label{BB061}
%\sup_{\eps\ge 0}
\widetilde{\mE}_{\widetilde{\mP}}|\mZ_T|^q\le C(\Xi,q).
\end{align}
Define $\dif\mQ:=\mZ_T\dif \widetilde{\mP}$. Then, by Girsanov's theorem,
\begin{align*}
\text{$\bar{W}_t:=\widetilde{W}_t-\int_0^t \sigma^{-1}B(s,\widetilde{Z}^x_s)\dif s$ is a $\mQ$-martingale}.
\end{align*}
In other words,
\begin{align*}
\widetilde{Z}^x_t=x+\int_0^t B(s,\widetilde{Z}_s)\dif s+\int_0^t\sigma(\widetilde{Z}_s)\dif \bar{W}_s,\ \ \mQ-a.e.
\end{align*}
Therefore, by the uniqueness of \eqref{II03}, we have
\begin{align}\label{Tra001}
\mQ\circ(\widetilde{Z}^x_\cdot)^{-1}=\mP\circ(X^x_\cdot)^{-1}.
\end{align}
Now, we show (i)-(iv) one by one.

(i): In view of \eqref{Tra001}, one sees that
\begin{align*}
\mE f(X_t^x)=\widetilde{\mE}_{\mQ}f(\widetilde{Z}^x_t)=\widetilde{\mE}_{\widetilde{\mP}}\[\mZ_Tf(\widetilde{Z}_t)\].
\end{align*}
By H\"older's inequality, \eqref{BB061} and \eqref{BB00}, we have for any $r\in(1,p)$ and $1/r'+1/r=1$,
\begin{align*}
\mE f(X^x_t)\le \left(\widetilde{\mE}_{\widetilde{\mP}}|\mZ_T|^{r'}\right)^{1/r'}\left(\widetilde{\mE}_{\widetilde{\mP}} |f(\widetilde{Z}_t)|^r\right)^{1/r}
\lesssim \left(t^{-dr/(2p)}\nor |f|^r\nor_{p/r}\right)^{1/r}
\lesssim t^{-d/(2p)}\nor f\nor_p,
\end{align*}
which is \eqref{BB06}.

(ii): Similarly to (i), by H\"older's inequality and \eqref{Kry00}, we have
\begin{align*}
\mE \int_s^tf(u,X^x_u)\dif u
&
\le \left(\widetilde{\mE}_{\widetilde{\mP}}|\mZ_T|^{r'}\right)^{1/r'}\left(\widetilde{\mE}_{\widetilde{\mP}} \left[\int_s^tf(u,\widetilde{Z}^x_u)\dif u\right]^r\right)^{1/r}\\
&
\lesssim\left(\widetilde{\mE}_{\widetilde{\mP}}\left( (t-s)^{r-1}\int_s^t|f(u,\widetilde{Z}^x_u)|^r\dif u\right)\right)^{1/r}\\
&
\lesssim (t-s)^{1-1/q-d/(2p)}\left(\nor |f|^r\nor_{\widetilde{\mL}^{p/r}_{q/r}}\right)^{1/r}\lesssim (t-s)^{1-1/q-d/(2p)}\nor f\nor_{\widetilde{\mL}^{p}_{q}}.
\end{align*}
The term $\mE \int_s^tf(u,X^x_{\pi_h(u)})\dif u$ can be estimated the same way.

(iii):
For \eqref{Kh02}, it again follows from H\"older's inequality that
\begin{align*}
\mE\exp\left(\int_0^T f(t,X^x_t)\dif t\right)
&
=\widetilde{\mE}_{\widetilde{\mP}}\left[\mZ_T\exp\(\int_0^T f(t,\widetilde{Z}^x_t)\dif t\)\right]\\
&
\le \(\widetilde{\mE}_{\widetilde{\mP}}|\mZ_T|^{r'}\)^{1/r'}\left(\widetilde{\mE}_{\widetilde{\mP}}
\exp\(r\int_0^T f(t,\widetilde{Z}^x_t)\dif t\)\right)^{1/r}<\infty
\end{align*}
by \eqref{BB061} and \eqref{Kh00}.

(iv): Let
\begin{align*}
A^f(h,X):=\sup_{t\in[0,T]}\Big|\int_{0}^t
(f(s,X^x_s)-f(s,X^x_{\pi_h(s)}))\dif u\Big|^2
\end{align*}
and
\begin{align*}
A^f(h,\widetilde{Z}):=\sup_{t\in[0,T]}\Big|\int_{{0}}^t
(f(s,\widetilde{Z}^x_s)-f(s,\widetilde{Z}^x_{\pi_h(s)}))\dif u\Big|^2.
\end{align*}
For any $\delta\in(0,1)$, we note that
\begin{align*}
\mE A^f(h,X)
&
=\widetilde{\mE}_{\widetilde{\mP}}\(\mZ_TA^f(h,\widetilde{Z})\)\\
&
=\widetilde{\mE}_{\widetilde{\mP}}\(\mZ_T|A^f(h,\widetilde{Z})|^\delta|A^f(h,\widetilde{Z})|^{1-\delta}\).
\end{align*}
Based on H\"older's inequality, \eqref{BB061}, \eqref{Kry02} and \eqref{DD03}, for $1/r'+1/r=1$ with some $r\in(1,p/2)$, we have
\begin{align*}
\mE A^f(h,X)
&
\le \[\widetilde{\mE}_{\widetilde{\mP}}\(|\mZ_T|^{1/\delta}A^f(h,\widetilde{Z})\)\]^{\delta}
\[\widetilde{\mE}_{\widetilde{\mP}}A^f(h,\widetilde{Z})\]^{1-\delta}\\
&
\le\(\widetilde{\mE}_{\widetilde{\mP}}|\mZ_T|^{r'/\delta}\)^{\delta/r'}
\(\widetilde{\mE}_{\widetilde{\mP}}|A^f(h,\widetilde{Z})|^{r}\)^{\delta/r}
\(\widetilde{\mE}_{\widetilde{\mP}}A^f(h,\widetilde{Z})\)^{1-\delta}\\
&
\lesssim \nor |f|^{2r}\nor_{\widetilde{\mL}^{p/(2r)}(T)}^{\delta/r} h^{(1-\delta_0)}
\nor f\nor_{\widetilde{\mL}^{p}(T)}^{2(1-\delta)}\\
&
\lesssim h^{1-\delta_0}\nor f\nor_{\widetilde{\mL}^{p}(T)}^2,
\end{align*}
where $\delta_0=\delta_0(\delta)\to0$ as $\delta\to0$,
and complete the proof.

\end{proof}
Next, we want to prove an estimate for $\sup_x\|\mu_s^x-\mu_t^x\|_{var}$.
To this end, we will use the relation between the PDE and the SDE.
For any $T>0$, consider the following backward PDE:
\begin{align}\label{BPDE0}
\p_t u^T+B\cdot\nabla u^T+a_{ij}\p_i\p_j u^T=0,\quad u^T(T)=\varphi,
\end{align}
where $\varphi\in C^\infty_b$. By Proposition \ref{well-PDE},
there exists a unique solution $u^T$ to \eqref{BPDE0} in the sense of Definition \ref{Pre4:Def}. Set
\begin{align*}
P^X_{s,t}f(x):=\mE f(X^x_{s,t}),\quad P^X_t:=P^X_{0,t}.
\end{align*}
By \eqref{BB06}, the domain of $P^X_{s,t}$ includes $\widetilde{L}^p$ for any $p\in(1,\infty]$.
Then we have the following probabilistic representation.
\bp\label{Pro:PR}
Let $T>0$, $\varphi\in C^\infty_b(\mR^d)$, $u^T$ and $X^x_{s,t}$
be the solutions to \eqref{BPDE0} and \eqref{II03} respectively. Then,
\begin{align}\label{PR001}
u^T(s,x)=\mE\varphi(X^x_{s,T})=P^X_{s,T}\varphi(x).
\end{align}

\ep
\begin{proof}
It is straightforward to obtain \eqref{PR001} by applying the generalized It\^o formula \eqref{GIF001}
to the function $t\mapsto u^T(t,X^x_{s,t})$ and taking expectation.
\end{proof}
Apart from the probabilistic representation, by the generalized It\^o formula, we have the following Duhamel formula.
\bl[Duhamel formula]
For any $\varphi\in C^\infty_b(\mR^d)$,
\begin{align}\label{DF002}
P^X_{s,t}\varphi(x)=P^\sigma_{t-s}\varphi(x)+\int_s^t P^X_{s,r}\(B(r)\cdot\nabla P^\sigma_{t-r}\varphi\)(x)\dif r.
\end{align}

\el
\begin{proof}
For any $t\in[0,T]$, let $v^t=v^t(t,x)$ be the solution to the following backward PDE:
\begin{align*}
\p_r v^t+a_{ij}\p_i\p_j v^t=0,\quad v^t(t)=\varphi.
\end{align*}
Based on \eqref{PR001}, one sees that $v^t(r)=P^\sigma_{t-r}\varphi$.
By the generalized It\^o formula \eqref{GIF001}, we have
\begin{align*}
\mE v^t(t,X^x_{s,t})=v^t(s,x)+\mE\int_s^t(\p_rv^t+a_{ij}\p_i\p_j v^t+B\cdot\nabla v^t)(r,X^x_{s,r})\dif r,
\end{align*}
which implies that
\begin{align*}
P^X_{s,t} \varphi(x)&=P^\sigma_{t-s}\varphi+\mE\int_s^t(B(r)\cdot\nabla P^\sigma_{t-r}\varphi)(X^x_{s,r})\dif r\\
&=P^\sigma_{t-s}\varphi+\int_s^tP^X_{s,r}(B(r)\cdot\nabla P^\sigma_{t-r}\varphi)(x)\dif r,
\end{align*}
and we complete the proof.
\end{proof}
Now we can prove the estimate for $\|\mu_t^x-\mu_s^x\|_{var}$.
We note that similar results have been proved in \cite[Lemma 3.8 (1)(ii)]{Zhao2020}
based on heat kernel estimates. It should be mentioned that
the order of time regularity in \cite{Zhao2020}
depends on the H\"older index $\beta$ of $\sigma$ and thus this is not applicable to our case.
\bl\label{NN0000}
Assume \eqref{NBB03}. For any $T>0$ and $q\in[p_0,\infty)$, there is a constant $C=C(\Xi,q)$ such that for any $0<s\le t\le T$ and $\varphi\in C^\infty_b$,
\begin{align}\label{NN000}
\|P^X_t\varphi-P^X_s\varphi\|_\infty\le C\[[(t-s)^{\frac{1}{2}-\frac{d}{2q}} s^{-\frac{1}{2}+\frac{d}{2q}}]\wedge 1\]s^{-\frac{d}{2q}}\nor \varphi\nor_q.
\end{align}
In particular, when $q=\infty$, for any $\delta>0$, there is a constant $C=C(\Xi,\delta)$ such that for all $x\in\mR^d$,
\begin{align}\label{LL02}
\|\mu^x_t-\mu^x_s\|_{var}\le C\[[(t-s)^{1/2-\delta} s^{-1/2}]\wedge 1\].
\end{align}

\el

\begin{proof}
For simplicity, let
\begin{align*}
\alpha:=\frac{1}{2}-\frac{d}{2q}.
\end{align*}
From \eqref{DF002}, one sees that
\begin{align*}
P^X_t\varphi-P^X_s\varphi=&\(P^\sigma_{t}\varphi-P^\sigma_{s}\varphi\)
+\int_s^t P^X_{r}\(B(r)\cdot\nabla P^\sigma_{t-r}\varphi\)\dif r\\
&+\int_0^s P^X_{r}\Big[B(r)\cdot\nabla \(P^\sigma_{t-r}-P^\sigma_{s-r}\)\varphi\Big]\dif r\\
:=&\sI_1+\sI_2+\sI_3.
\end{align*}
Based on \eqref{MM03}, we have
\begin{align*}
\|\sI_1\|_\infty &\lesssim \[[(t-s)s^{-1}]\wedge 1\]s^{-\frac{d}{2q}}\nor \varphi\nor_q\\
&\lesssim \[(t-s)^\alpha s^{-\alpha}\]s^{-\frac{d}{2q}}\nor \varphi\nor_q.
\end{align*}
By \eqref{BB06} and \eqref{BB00}, we have
\begin{align*}
\|\sI_2\|_\infty&\lesssim \int_s^t r^{-d/(2p_0)}\nor B(r)\cdot\nabla P^\sigma_{t-r}\varphi\nor_{p_0}\dif r\\
&\lesssim \int_s^t r^{-d/(2p_0)}\|\nabla P^\sigma_{t-r}\varphi\|_\infty\dif r\\
&\lesssim s^{-d/(2q)}\int_s^t r^{-d/(2p_0)+d/(2q)}(t-r)^{-1/2-d/(2q)}\dif r\nor \varphi\nor_q\\
&:=s^{-d/(2q)}K(t,s)\nor \varphi\nor_q,
\end{align*}
where
\begin{align*}
K(t,s)&\le \left(s^{-\frac{d}{2p_0}+\frac{d}{2q}}\int_s^t (t-r)^{\alpha-1}\dif r\right)\wedge\left((t-s)^{\frac{1}{2}-\frac{d}{2p_0}}\int_0^1r^{-\frac{d}{2p_0}+\frac{d}{2q}}(1-r)^{-\frac{1}{2}-\frac{d}{2q}}\dif r\right)\\
&\lesssim (s^{-\alpha}(t-s)^\alpha)\wedge 1
\end{align*}
since $q\ge p_0>d$.

It remains to estimate $\sI_3$. By \eqref{BB06} and \eqref{MM03}, we have
\begin{align*}
\|\sI_3\|_\infty&\lesssim \int_0^s r^{-d/(2p_0)}\nor B(r)\cdot\nabla \(P^\sigma_{t-r}-P^\sigma_{s-r}\)\varphi\nor_{p_0}\dif r\\
&\lesssim \int_0^s r^{-d/(2p_0)}\|\nabla \(P^\sigma_{t-r}-P^\sigma_{s-r}\)\varphi\|_\infty\dif r\\
&\lesssim \int_0^s r^{-d/(2p_0)}\[[(t-s)^{\frac12}(s-r)^{-\frac12}]\wedge1\](s-r)^{-1/2-d/(2q)}\dif r\nor \varphi\nor_q.
\end{align*}
We note that
\begin{align}
\sJ&:=\int_0^s r^{-d/(2p_0)}\[[(t-s)^{\frac12}(s-r)^{-\frac12}]\wedge1\](s-r)^{-1/2-d/(2q)}\dif r\no\\
&\le\int_0^{s} r^{-d/(2p_0)}(s-r)^{-1/2-d/(2q)}\dif r\lesssim s^{1/2-d/(2p_0)-d/(2q)}\lesssim s^{-d/(2q)},\label{Star0011}
\end{align}
since $p_0>d$. In addition, when $r\in(0,\frac{s}{2}]$, one sees that
\begin{align*}
\[(t-s)^{\frac12}(s-r)^{-\frac12}\]\wedge1\leq(t-s)^{\frac{1}{2}-\frac{d}{2q}}
(s-r)^{-\frac{1}{2}+\frac{d}{2q}}
\leq(t-s)^{\frac{1}{2}-\frac{d}{2q}}\(\frac{s}{2}\)^{-\frac{1}{2}+\frac{d}{2q}}.
\end{align*}
 Hence,
\begin{align*}
\sJ\lesssim
&
s^{-1/2-d/(2q)}\int_0^{s/2} r^{-d/(2p_0)}\[[(t-s)^{\frac12}(s-r)^{-\frac12}]\wedge1\]\dif r\\
&
+ s^{-d/(2p_0)} \int_{s/2}^s \[[(t-s)^{\frac12}(s-r)^{-{\frac12}}]\wedge1\](s-r)^{-1/2-d/(2q)}\dif r\\
\lesssim
&
s^{-1}(t-s)^{1/2-d/(2q)}\int_0^{s/2} r^{-d/(2p_0)}\dif r\\
&
+ s^{-d/(2p_0)} \int_{s/2}^s \[[(t-s)^{\frac12}(s-r)^{-\frac12}]\wedge1\](s-r)^{-1/2-d/(2q)}\dif r.
\end{align*}
By a change of variables, we have
 \begin{align*}
\sJ\lesssim
&
s^{-d/(2p_0)}(t-s)^{1/2-d/(2q)}\\
&
+ s^{-d/(2p_0)}(t-s)^{1-1/2-d/(2q)} \int_{0}^{s/(t-s)} \[(s-r)^{-\frac12}\wedge1\](s-r)^{-1/2-d/(2q)}\dif r\\
\lesssim
&
s^{-d/(2p_0)}(t-s)^{1-1/2-d/(2q)}
\lesssim s^{-\frac{1}{2}}(t-s)^{\frac{1}{2}-\frac{d}{2q}}
\end{align*}
since $q<\infty$ and $p_0>d$, which combined with \eqref{Star0011} implies that
\begin{align*}
\sJ\lesssim \(s^{-\frac12}(t-s)^\alpha\)\wedge s^{-\frac{d}{2q}}=\Big([(t-s)^\alpha s^{-\alpha}]\wedge 1\Big)s^{-\frac{d}{2q}}.
\end{align*}
Thus, we obtain \eqref{NN000}.
In particular, noting that
\begin{align*}
\nor \varphi\nor_{q}\lesssim \|\varphi\|_{\infty},\quad \forall q<\infty,
\end{align*}
by Lusin's theorem and \eqref{NN000}, we have for all $\delta>0$ and $q\in[p_0\vee(d/(2\delta)),\infty)$
\begin{align*}
\|\mu^x_t-\mu^x_s\|_{var}&=\sup_{\varphi\in C^\infty_b}\frac{|P^X_t\varphi(x)-P^X_s\varphi(x)|}{\|\varphi\|_\infty}\lesssim [(t-s)^{\frac{1}{2}-\frac{d}{2q}} s^{-\frac{1}{2}+\frac{d}{2q}}]s^{-\frac{d}{2q}}\\
&\lesssim (t-s)^{\frac{1}{2}-\delta} s^{-\frac{1}{2}}.
\end{align*}
Moreover, it is easy to see that
\begin{align*}
\|\mu^x_t-\mu^x_s\|_{var}\le 2,
\end{align*}
which completes the proof.
\end{proof}
%%%%%%%%%%%%%%%%%%%%%%%%%%
\iffalse
 By \cite[Lemma 2.6]{RZ21} and \eqref{Kry02}, we have the following Krylov's estimate in distribution dependent version.
\bl
Let $T>0$, $p_1,p_2,q\in(1,\infty)$ with $d/p_1<2$. For any $T>0$ and $Y\in \mK^{p_2,q}_{T,\kappa}$, there is a constant $C=C(d,T,\nor B\nor_{\widetilde{\mL}^{p_0}(T)},p_1)$ such that for any positive $f\in \widetilde{\mL}^{p_1,p_2}_q(T)$,
\begin{align}\label{Kry03}
\mE\int_0^T f(r,Y_r,\mu_r)\dif r\le C\kappa\nor f\nor_{\widetilde{\mL}^{p_1,p_2}_{q}(T)}.
\end{align}
\el
\fi
%%%%%%%%%%%%%%%%%%%%%%%%%%%%%%%%
The following lemma is the distribution dependent version of \eqref{BB05}.
\bl
For any $T>0$, $p\in(d\vee 2,\infty)$, assume that $f:\mR_+\times\mR^{d}\times\cP(\mR^d)\to\mR$ such that
\begin{align*}
\kappa_f:=\sup_{t\in[0,T]}\sup_{\mu,\nu}\Big(\nor f(t,\cdot,\mu)\nor_p+\frac{\nor f(t,\cdot,\mu)-f(t,\cdot,\nu)\nor_p}{\|\mu-\nu\|_{var}}\Big)<\infty.
\end{align*}
Then, for any $\delta>0$, there is a constant $C=C(\Xi,p,\delta)$ such that for any $x\in\mR^d$ and $h\in(0,1)$
\begin{align}\label{EE00}
\begin{split}
\mE\Big(\sup_{t\in[0,T]}\Big|\int_{0}^t(f(s,X^x_s,\mu^x_s)&-f(s,X^x_{\pi_h(s)},\mu^x_{\pi_h(s)}))\dif s\Big|^2\Big)\le C(\kappa_f)^2 h^{1-\delta}.
\end{split}
\end{align}

\el
\begin{proof}
For simplicity, we drop the superscript $x$ from $X^x$ and $\mu^x$.
First of all, we note that
\begin{align*}
&~~~\mE\Big(\sup_{t\in[0,T]}\Big|\int_{0}^t(f(s,X_s,\mu_s)-f(s,X_{\pi_h(s)},\mu_{\pi_h(s)}))\dif s\Big|^2\Big)\\
\lesssim&~ \mE\Big(\sup_{t\in[0,T]}\Big|\int_{0}^t(f(s,X_s,\mu_s)-f(s,X_{\pi_h(s)},\mu_s))\dif s\Big|^2\Big)\\
&+\mE\Big(\sup_{t\in[0,T]}\Big|\int_{0}^t(f(s,X_{\pi_h(s)},\mu_s)-f(s,X_{\pi_h(s)},\mu_{\pi_h(s)}))\dif s\Big|^2\Big)\\
:=&\sI^h_1+\sI^h_2.
\end{align*}
By \eqref{BB05}, for any $\delta>0$, we have
\begin{align*}
\sI^h_1\lesssim h^{1-\delta}\sup_{t\in[0,T]}\nor f(t,\cdot,\mu_t)\nor_{p}^2\lesssim  h^{1-\delta}(\kappa_f)^2.
\end{align*}
For $\sI^h_2$, we use the same method as in Step 2 of the proof of Lemma \ref{CruL}. For any $0\le s<t\le T$,
set
\begin{align*}
\sJ_{s,t}^h:=\int_{s+h}^{t+h}|f(r,X_{\pi_h(r)},\mu_r)-f(r,X_{\pi_h(r)},\mu_{\pi_h(r)})|\dif r
\end{align*}
and $\|\cdot\|:=\(\mE|\cdot|^2\)^{1/2}$. Then
\begin{align*}
\sI^h_2&\le \mE\Big|\int_{h}^{T+h}|f(s,X_{\pi_h(s)},\mu_s)-f(s,X_{\pi_h(s)},\mu_{\pi_h(s)})|\dif s\Big|^2=\|\sJ_{0,T}^h\|^2,
\end{align*}
since $\pi_h(r)=r$, if $r<h$.
Based on H\"older's inequality and \eqref{Kry02}, one sees that if $2/q+d/p<1$, then
\begin{align*}
\mE|\sJ_{s,t}^h|^2&\lesssim (t-s)\mE\int_{s+h}^{t+h}
\Big|f(r,X_{\pi_h(r)},\mu_r)-f(r,X_{\pi_h(r)},\mu_{\pi_h(r)})\Big|^2\dif r\\
&
\lesssim (t-s)\Big(\int_{s+h}^{t+h}\nor f(r,\cdot,\mu_r)-f(r,\cdot,\mu_{\pi_h(r)})\nor_{p}^{q}\dif r\Big)^{2/q}\\
&
\lesssim (\kappa_f)^2(t-s)\Big(\int_{s+h}^{t+h}\|\mu_r-\mu_{\pi_h(r)}\|_{var}^{q}\dif r\Big)^{2/q}.
\end{align*}
Then, by \eqref{LL02} and the fact that $q>2$, for any $\delta>0$ we have
\begin{align*}
\|\sJ_{s,t}^h\|&\lesssim \kappa_f(t-s)^{1/2}\Big(\int_{s+h}^{t+h} h^{(\frac12-\delta)q}(\pi_h(r))^{-q/2}\dif r\Big)^{1/q}\\
&\lesssim \kappa_f(t-s)^{1/2}h^{1/2-\delta}\Big(\int_{s}^{t} r^{- q/2}\dif r\Big)^{1/q}\lesssim \kappa_fh^{1/2-\delta}(t-s)^{1/2+1/q}s^{-1/2}.
\end{align*}
Taking $t_n:=2^{-n}T$, we have
\begin{align*}
\|\sJ_{0,T}^h\|
&
\le\sum_{n=0}^\infty\|\sJ_{t_{n+1},t_n}^h\|\lesssim \kappa_fh^{1/2-\delta}
\sum_{n=0}^\infty(t_n-t_{n+1})^{1/2+1/q}{(t_{n+1})}^{-1/2}\\
&
\lesssim \kappa_fh^{1/2-\delta} T^{1/q}\sum_{n=0}^\infty 2^{-\frac{n+1}{q}}\lesssim \kappa_fh^{1/2-\delta},
\end{align*}
which implies $\sI^h_2\lesssim (\kappa_f)^2h^{1-2\delta}$
and we complete the proof.

\end{proof}

Now, we have the following strong fluctuation result, which is crucial in this paper.
\bl\label{LemEE01}
Let $T>0$ and $g:\mR_+\times\mR^d\to\mR$ be a bounded function satisfying
\begin{align}\label{Cru00}
c_g:=\sup_{t\ne s\in[0,T]\atop x\in\mR^d}\Big(|g(t,x)|+\frac{|g(t,x)-g(s,x)|}{|t-s|^\alpha}\Big)<\infty
\end{align}
for some $\alpha>0$.
Assume \eqref{NBB03},  {\bf (H$_b^1$)} and {\bf (H$_b^2$)} hold. Then for any $\delta>0$, there is a constant $C=C(\Xi,\kappa_0,\delta,c_g)$ such that for any $x\in\mR^d$ and $\eps>0$,
\begin{align}\label{EE01}
\begin{split}
&\mE\left(\sup_{t\in[0,T]}\Big|\int_0^tg(s,X^x_s)(b(\frac{s}{\eps},X^x_s,\mu^x_s)-\bar{b}(X^x_s,\mu^x_s))\dif s\Big|^2\right)\\
&\le C\inf_{h>0}\left(h^{1-\delta}
+h^{2\alpha}+\left(\omega\left(\frac{h}{\eps}\right)\right)^2\right).
\end{split}
\end{align}

\el

\begin{proof}
For simplicity, we drop the superscript $x$ from $X^x$ and $\mu^x$.
Set $b_\eps(t):=b(t/\eps)$ and
\begin{align*}
\X_\cdot:=(X_\cdot,\mu_\cdot).
\end{align*}
For any $h\in(0,1)$,
\begin{align*}
&\mE\left(\sup_{t\in[0,T]}\Big|\int_0^tg(s,X^x_s)(b_\eps(s,X^x_s,\mu^x_s)-\bar{b}(X^x_s,\mu^x_s))\dif s\Big|^2\right)\\
&
\lesssim \mE\left(\sup_{t\in[0,T]}\Big|\int_0^t((gb_\eps)(s,\X_s)-(gb_\eps)(s,\X_{\pi_h(s)}))\dif s\Big|^2\right)\\
&\quad
+\mE\left(\sup_{t\in[0,T]}\Big|\int_0^tg(s,X_{\pi_h(s)})(b_\eps(s,\X_{\pi_h(s)})-\bar{b}(\X_{\pi_h(s)}))\dif s\Big|^2\right)\\
&\quad
+\mE\left(\sup_{t\in[0,T]}\Big|\int_0^t((g\bar{b})(s,\X_s)-(g\bar{b})(s,\X_{\pi_h(s)}))\dif s\Big|^2\right)\\
&
:=\sI^{\eps,h}_1+\sI^{\eps,h}_2+\sI^{\eps,h}_3.
\end{align*}
By {\bf (H$^1_b$)} and \eqref{EE00}, for any $\delta>0$, we have
\begin{align*}
\sI^{\eps,h}_1+\sI^{\eps,h}_3
\lesssim \|g\|_\infty^2\kappa_0^2h^{1-\delta}.
\end{align*}

For $\sI^{\eps,h}_2$, we note that by \eqref{Cru00} and \eqref{Kry02} with $p_0>(d/2)\vee1$,
\begin{align*}
\sI^{\eps,h}_2&\lesssim \mE\left(\sup_{t\in[0,T]}\Big|\int_0^tg(\pi_h(s),X_{\pi_h(s)})(b_\eps(s,\X_{\pi_h(s)}\)-\bar{b}(\X_{\pi_h(s)}))\dif s\Big|^2\right)\\
&+\mE\left(\sup_{t\in[0,T]}\int_0^t|s-\pi_h(s)|^{2\alpha}\Big(|b_\eps(s,\X_{\pi_h(s)})|^2+|\bar{b}(\X_{\pi_h(s)}))|^2\Big)\dif s\right)\\
&\lesssim \sI^{\eps,h}_{21}+h^{2\alpha},
\end{align*}
where
\begin{align*}
\sI^{\eps,h}_{21}:=\mE\left(\sup_{t\in[0,T]}\Big|\int_0^tg(\pi_h(s),X_{\pi_h(s)})(b_\eps(s,\X_{\pi_h(s)})-\bar{b}(\X_{\pi_h(s)}))\dif s\Big|^2\right).
\end{align*}
It suffices to show
\begin{align*}
\sI^{\eps,h}_{21}\lesssim h+\left(\omega\left(\frac{h}{\eps}\right)\right)^2.
\end{align*}
Letting $M(t)=[t/h]$ and noting that $\pi_h(s)=s$ for $s\in[0,h)$, we have
\begin{align*}
\sI^{\eps,h}_{21}
\lesssim
&
\mE\left(\sup_{t\in[0,h)}\Big|\int_0^tg(s,X_s)(b_\eps(s,\X_{s})-\bar{b}(\X_{s}))\dif s\Big|^2\right)\\
&
+\mE\left(\sup_{t\in[h,T]}\Big|\int_0^hg(s,X_s)(b_\eps(s,\X_{s})-\bar{b}(\X_{s}))\dif s\Big|^2\right)\\
&
+\mE\left(\sup_{t\in[h,T]}\Big|\int_{M(t)h}^tg(\pi_h(s),X_{\pi_h(s)})(b_\eps(s,\X_{\pi_h(s)})
-\bar{b}(\X_{\pi_h(s)})\dif s\Big|^2\right)\\
&
+\mE\left(\sup_{t\in[h,T]}\Big|\int_h^{M(t)h}g(\pi_h(s),X_{\pi_h(s)})(b_\eps(s,\X_{\pi_h(s)})
-\bar{b}(\X_{\pi_h(s)}))\dif s\Big|^2\right)\\
:=&
\sI^{\eps,h}_{211}+\sI^{\eps,h}_{212}+\sI^{\eps,h}_{213}+\sI^{\eps,h}_{214}.
\end{align*}
It follows from H\"older's inequality
%and \eqref{Kry02},
 that
\begin{align*}
\sum_{i=1}^3\sI^{\eps,h}_{21i}
&
\lesssim h \|g\|_\infty^2\left(\mE\left[\int_0^T|b_\eps(s,\X_{s})|^2+|\bar{b}(\X_{s})|^2\dif s\right]
+\mE\left[\int_0^T|b_\eps(s,\X_{\pi_h(s)})|^2+|\bar{b}(\X_{\pi_h(s)})|^2\dif s\right]\right)\\
&
\lesssim h\|g\|_\infty^2\kappa_0^2,
\end{align*}
because of \eqref{Kry02}.
Thus, we only need to prove
\begin{align*}
\sI^{\eps,h}_{124}\lesssim \(\omega\left(\frac{h}{\eps}\right)\)^2.
\end{align*}
By the definition of $\pi_h$, it is easy to see that
\begin{align*}
\sI^{\eps,h}_{214}\le\mE\Big(\sup_{2\le m\le M(T)}\Big|\sum_{k=1}^{m-1}g(kh,X_{kh})\int_{kh}^{(k+1)h}(b_\eps(s,\X_{kh})-\bar{b}(\X_{kh}))\dif s\Big|^2\Big).
\end{align*}

Based on the fact that $|\sum_{k=1}^{m-1}a_k|^2\le (m-1)\sum_{k=1}^{m-1}|a_k|^2$, one sees that
\begin{align*}
\sI^{\eps,h}_{214}
&
\le M(T)c_g^2\sum_{k=1}^{M(T)-1}\mE\Big|\int_{kh}^{(k+1)h}
(b_\eps(s,\X_{kh})-\bar{b}(\X_{kh}))\dif s\Big|^2.
\end{align*}
By a change of variables and \eqref{in:00}, we have
\begin{align*}
\sI^{\eps,h}_{214}
&
\lesssim
M(T)\sum_{k=1}^{M(T)-1}\mE\left|\eps\int_{kh/\eps}^{(k+1)h/\eps}\(b(s,\X_{kh})-\bar{b}(\X_{kh})\)\dif s\right|^2\\
&
\lesssim [\frac{T}{h}]h^2\sum_{k=1}^{M(T)-1}\mE\Big|\frac{\eps}{h}\int_{kh/\eps}^{(k+1)h/\eps}
(b(s,\X_{kh})-\bar{b}(\X_{kh}))\dif s\Big|^2\\
&
\lesssim h \left(\omega\left(\frac{h}{\eps}\right)\right)^2
\sum_{k=1}^{M(T)-1}\mE |H(\X_{kh})|^2.
\end{align*}
We note that
\begin{align*}
h\sum_{k=1}^{M(T)-1}\mE |H(\X_{kh})|^2=\mE\int_{h}^{M(T)h}|H(\X_{\pi_h(s)})|^2\dif s
\le\mE\int_{h}^{T}|H(\X_{\pi_h(s)})|^2\dif s.
\end{align*}
Again by \eqref{Kry02}, we have
\begin{align*}
\sI^{\eps,h}_{214}\lesssim \left(\omega\left(\frac{h}{\eps}\right)\right)^2
\sup_{\mu}\nor H(\cdot,\mu)\nor_{p_0}^2
\end{align*}
and complete the proof.

\end{proof}

\subsection{Time regularity for solutions to the parabolic equation}\label{SubS32}
In this section, we establish the time regularity for the solution of PDE \eqref{NBB00}.
First, we give the following probabilistic representation of the solution to PDE if $B\equiv0$.
\bl
Let $B\equiv0$ and $u$ be a solution to PDE \eqref{NBB00}. Then,
\begin{align}\label{CCA01}
u(t)=\int_0^t e^{-\lambda (t-s)}P^\sigma_{t-s}f(s)\dif s
+e^{-\lambda t}P^\sigma_t\varphi.
\end{align}
\el
\begin{proof}
Applying the generalized It\^o formula \eqref{GIF001} to
$$
s\to e^{-\lambda s}u(t-s,Z_s^x),
$$
we get
\begin{align*}
e^{-\lambda t}u(0,Z_t^x)-u(t,x)=&\int_0^t e^{-\lambda s}(-\p_su+a_{ij}\p_i\p_j u-\lambda u)(t-s,Z_s^x)\dif s\\
&+\int_0^t e^{-\lambda s}\nabla u(t-s,Z_s^x)\dif W_s.
\end{align*}
Taking expectation of both sides, we obtain that
\begin{align*}
e^{-\lambda t}P^\sigma_t\varphi-u(t)=-\int_0^t e^{-\lambda s}P^\sigma_sf(t-s)\dif s,
\end{align*}
which is \eqref{CCA01} by a change of variable and this completes the proof.
\end{proof}

Using the above Lemma, we have the following time H\"older regularity of $\nabla u$.
\bl
Assume $\varphi\equiv0$. Under condition \eqref{NBB03} with some $p_0\in(d,\infty)$, for any $\lambda\ge0$,
there is a constant $C=C(\Xi,p,\lambda)$ such that for all $t,s\in[0,T]$ and $f\in\widetilde{\mL}^{p_0}(T)$
the solution $u$ to \eqref{NBB00} in the sense of Definition \ref{Pre4:Def} satisfies
\begin{align}\label{HH02}
\|\nabla u(t)-\nabla u(s)\|_\infty\le C|t-s|^{\frac{1}{2}-\frac{d}{2p_0}}\nor f\nor_{\widetilde{\mL}^{p_0}(T)}.
\end{align}
\el
\br\label{Re002}\rm
By Remark \ref{Re001} and \eqref{HH02}, we further have that there is a version of the solution such that $u\in C([0,T];\bC^1)$.
\er
\begin{proof}
First, since $B,f\in\widetilde{\mL}^{p_0}(T)\subset\widetilde{\mL}^{p_0}_q(T)$, $\forall q$, we indeed have a unique solution $u$.
Set
\begin{align*}
g(s):=B\cdot\nabla u(s)+f(s).
\end{align*}
In view of \eqref{CCA022},
\begin{align*}
\nor g\nor_{\widetilde{\mL}^{p_0}(T)}
\le\nor b\nor_{\widetilde{\mL}^{p_0}(T)}\|\nabla u\|_{\mL^\infty_T}
+\nor f\nor_{\widetilde{\mL}^{p_0}(T)}\lesssim \nor f\nor_{\widetilde{\mL}^{p_0}(T)}.
\end{align*}
Then, for any $0\le s<t\le T$,
by \eqref{CCA01}, one sees that
\begin{align*}
\|\nabla u(t)-\nabla u(s)\|_\infty
&
=\left\|\int_s^t e^{-\lambda(t-r)}\nabla P^\sigma_{t-r}g(r)\dif r\right\|_\infty\\
&
+\left\|\int_0^s \(e^{-\lambda(t-r)}-e^{-\lambda(s-r)}\)\nabla P^\sigma_{t-r}g(r)\dif r\right\|_\infty\\
&
+\left\|\int_0^s e^{-\lambda(s-r)}\(\nabla P^\sigma_{t-r}-\nabla P^\sigma_{s-r}\)g(r)\dif r\right\|_\infty\\
&
=:\sI_1+\sI_2+\sI_3.
\end{align*}
By \eqref{BB00}, one sees that
\begin{align*}
\sI_1
&
\lesssim \int_s^t (t-r)^{-1/2-d/(2{p_0})}\nor g(r)\nor_{p_0}\dif r\no\\
&
\lesssim \int_s^t (t-r)^{-1/2-d/(2{p_0})}\dif r\nor g\nor_{\widetilde{\mL}^{p_0}(T)}\no\\
&
\lesssim|t-s|^{\frac{1}{2}-\frac{d}{2{p_0}}}\nor f\nor_{\widetilde{\mL}^{p_0}(T)}.%\label{KK00}
\end{align*}
For $\sI_2$, noting that $|e^{-x}-e^{-y}|\le |x-y|$ for any $x,y>0$, it follows from \eqref{BB00} that
\begin{align*}
\sI_2
&
\lesssim |t-s|\int_0^s (t-r)^{-1/2-d/(2{p_0})}\nor g(r)\nor_{p_0}\dif r\\
&
\lesssim|t-s|\nor g\nor_{\widetilde{\mL}^{p_0}(T)}\lesssim|t-s|\nor f\nor_{\widetilde{\mL}^{p_0}(T)}.
\end{align*}
For $\sI_3$, by \eqref{MM03} with $q=\infty$, we have
\begin{align*}
\sI_3
&
\lesssim \int_0^s([(t-s)^{\frac12}(s-r)^{-\frac12}]\wedge1)(s-r)^{-1/2-d/(2p_0)}\dif r\nor g\nor_{\widetilde{\mL}^{p_0}(T)}.
\end{align*}
By a change of variable, we have
\begin{align*}
\sI_3&\lesssim |t-s|^{\frac12-\frac{d}{2p_0}} \int_0^\infty ([(s-r)^{-\frac12}]\wedge1)(s-r)^{-1/2-d/(2p_0)}\dif r\nor f\nor_{\widetilde{\mL}^{p_0}(T)}\\
&\lesssim |t-s|^{1/2-d/(2p_0)}\nor f\nor_{\widetilde{\mL}^{p_0}(T)}
\end{align*}
and complete the proof.

\end{proof}
Moreover, we also have an estimate of time regularity for the solution to the following Cauchy problem.
\bl\label{Lemma3.12}
Assume \eqref{NBB03}. Let $\varphi\in C^\infty_b$ and let $u$ be the unique solution to the following Cauchy problem on $[0,T]$ in the sense of Definition \ref{Pre4:Def}
\begin{align}\label{CE00}
\p_t u=a_{ij}\p_i\p_j u+B\cdot\nabla u,\quad u_0=\varphi.
\end{align}
Then there is a constant $C=C(\Xi)$ such that for all $0\le s<t\le T$,
\begin{align}\label{NN01}
\|\nabla u(t)-\nabla u(s)\|_\infty\le C (t-s)^{\frac12-\frac{d}{2p_0}}s^{-1+\frac{d}{2p_0}}\|\varphi\|_\infty.
\end{align}

\el
%\br\rm
%Under the a priori estimates \eqref{NN00} and \eqref{NN01},  one can easy obtain that there is a unique mild solution to \eqref{CE00} for $\widetilde{L}^q$-initial data in some time weighted space with norm $\sup_{t\in[0,T]}t^{\delta}\|u(t)\|_{C^\beta}$.
%\er
\begin{proof}
First, by \eqref{CCA01}, we have
\begin{align}\label{MM07}
u(t)=\int_0^t P^\sigma_{t-s}(B\cdot \nabla u)(s)\dif s+P^\sigma_t \varphi,
\end{align}
which implies that
\begin{align*}
\|\nabla u(t)\|_\infty
&
\lesssim \int_0^t (t-s)^{-\frac{1}{2}-\frac{d}{2p_0}}\nor B\cdot \nabla u(s)\nor_{p_0}\dif s+t^{-\frac{1}{2}}\nor\varphi\nor_\infty\\
&
\lesssim \int_0^t (t-s)^{-\frac{1}{2}-\frac{d}{2p_0}}\|\nabla u(s)\|_\infty\dif s+t^{-\frac{1}{2}}\nor\varphi\nor_\infty
\end{align*}
because of \eqref{BB00}.
Hence, by Gronwall's inequality, we have
\begin{align}\label{MM06}
\|\nabla u(t)\|_\infty\lesssim t^{-\frac{1}{2}}\nor\varphi\nor_\infty.
\end{align}
 \iffalse
From \eqref{MM07}, for any $0\le s<t\le T$,
\begin{align*}
\|u(t)-u(s)\|_\infty\le& \int_s^t\|P_{t-r} (b\cdot\nabla u)(r)\|_\infty\dif r+\int_0^s\|(P_{t-r}-P_{s-r})b\cdot\nabla u(r)\|_\infty\dif r\\
&+\|(P_t-P_s)\varphi\|_\infty.
\end{align*}
It follows from \eqref{BB00}, \eqref{MM03} and \eqref{MM06} that
\begin{align*}
\|u(t)-u(s)\|_\infty\lesssim& \int_s^t(t-r)^{-\frac{d}{2p}}\|\nabla u(r)\|_\infty\dif r\\
&
+\int_0^s\Big\{\[(t-s)(s-r)^{-1}\]\wedge1\Big\}(s-r)^{-\frac{d}{2p}}\|\nabla u(r)\|_\infty\dif r\\
&
+\Big\{\[(t-s)s^{-1}\]\wedge1\Big\}s^{-\frac{d}{2q}}\nor\varphi\nor_q\\
\lesssim
&
\Big(\int_s^t(t-r)^{-\frac{d}{2p}}r^{-\frac{1}{2}-\frac{d}{2q}}\dif r+(t-s)^{\frac{1}{2}}\int_0^s(s-r)^{-\frac{1}{2}-\frac{d}{2p}}r^{-\frac{1}{2}-\frac{d}{2q}}\dif r\\
&
+(t-s)^{\frac{1}{2}}s^{-\frac{1}{2}-\frac{d}{2q}}\Big)\nor\varphi\nor_q\\
\lesssim
& (t-s)^{\frac{1}{2}}s^{-\frac{1}{2}-\frac{d}{2q}}\nor\varphi\nor_q,
\end{align*}
which is \eqref{NN00} with the following observation
\begin{align*}
\|u(t)-u(s)\|_\infty\le\|u(t)\|_\infty+\|u(s)\|_\infty\lesssim(t^{-\frac{d}{2q}}+s^{-\frac{d}{2q}})\nor\varphi\nor_q\lesssim s^{-\frac{d}{2q}}\nor\varphi\nor_q.
\end{align*}
\fi
By \eqref{MM07}, \eqref{BB00} and \eqref{MM03}, one sees that
\begin{align*}
\|\nabla u(t)-\nabla u(s)\|_\infty\le& \int_s^t\|\nabla P^\sigma_{t-r} (B\cdot\nabla u)(r)\|_\infty\dif r+\int_0^s\|\nabla(P^\sigma_{t-r}-P^\sigma_{s-r})B\cdot\nabla u(r)\|_\infty\dif r\\
&+\|\nabla(P^\sigma_t-P^\sigma_s)\varphi\|_\infty.\\
\lesssim& \int_s^t(t-r)^{-\frac12-\frac{d}{2p_0}}\|\nabla u(r)\|_\infty\dif r\\
&+\int_0^s\Big\{\[(t-s)^{\frac12}(s-r)^{-\frac12}\]\wedge1\Big\}(s-r)^{-\frac12-\frac{d}{2p_0}}\|\nabla u(r)\|_\infty\dif r\\
&+\Big\{\[(t-s)^{\frac12}s^{-\frac12}\]\wedge1\Big\}s^{-\frac{1}{2}}\|\varphi\|_\infty\\
:=&\sI_1+\sI_2+\sI_3.
\end{align*}
Then, noting that $x\wedge1\le x^\theta$ for all $x\ge0$ and $\theta\in[0,1]$, we have
\begin{align*}
\sI_3\lesssim \left\{\[(t-s)^{\frac12}s^{-\frac12}\]\wedge1\right\}s^{-\frac{1}{2}}\|\varphi\|_\infty
\lesssim (t-s)^{\frac12-\frac{d}{2p_0}}s^{-1+\frac{d}{2p_0}}\|\varphi\|_\infty.
\end{align*}

Based on \eqref{MM06} and a change of variable,
\begin{align*}
\sI_1&\lesssim \|\varphi\|_\infty\int_s^t(t-r)^{-\frac12-\frac{d}{2p_0}}r^{-\frac12}\dif r\lesssim (t-s)^{\frac12-\frac{d}{2p_0}}s^{-\frac12}\|\varphi\|_\infty\\
&\lesssim (t-s)^{\frac12-\frac{d}{2p_0}}s^{-1+\frac{d}{2p_0}}\|\varphi\|_\infty
\end{align*}
since $p_0>d$.
For $\sI_2$, again by \eqref{MM06}, we divide $[0,s]$ into $[0,s/2)$ and $[s/2,s]$ and have
\begin{align*}
\frac{\sI_2}{\|\varphi\|_\infty}&\lesssim\(\int_0^{s/2}+\int_{s/2}^s\)\Big\{\[(t-s)^{\frac12}(s-r)^{-\frac12}\]\wedge1\Big\}(s-r)^{-\frac12-\frac{d}{2p_0}}r^{-\frac12}\dif r\\
&\lesssim (t-s)^{\frac12-\frac{d}{2p_0}}s^{-1}\int_0^{s/2} r^{-\frac12}\dif r+s^{-\frac12}\int_{s/2}^s \Big\{\[(t-s)^{\frac12}(s-r)^{-\frac12}\]\wedge1\Big\}(s-r)^{-\frac12-\frac{d}{2p_0}}\dif r\\
&\lesssim (t-s)^{\frac12-\frac{d}{2p_0}}s^{-\frac12}+s^{-\frac12}\int_{0}^s \Big\{\[(t-s)^{\frac12}r^{-\frac12}\]\wedge1\Big\}r^{-\frac12-\frac{d}{2p_0}}\dif r,
\end{align*}
{where we used the fact that
\begin{align*}
&\left[\Big((t-s)^{\frac12}(s-r)^{-\frac12}\Big)\wedge1\right](s-r)^{-\frac12-\frac{d}{2p_0}}\\
&\le (t-s)^{\frac12-\frac{d}{2p_0}}(s-r)^{-1}\le 2 (t-s)^{\frac12-\frac{d}{2p_0}}s^{-1}\quad \forall r\in[0,s/2).
\end{align*}
}
 From a change of variable, we have
\begin{align*}
\sI_2
&\lesssim \Big[(t-s)^{\frac12-\frac{d}{2p_0}}s^{-\frac12}
+s^{-\frac12}(t-s)^{1-\frac12-\frac{d}{2p_0}}\int_{0}^\infty (r^{-\frac12}\wedge1)r^{-\frac12-\frac{d}{2p_0}}\dif r\Big]\|\varphi\|_\infty\\
&\lesssim (t-s)^{\frac12-\frac{d}{2p_0}}s^{-\frac12}\|\varphi\|_\infty\lesssim(t-s)^{\frac12-\frac{d}{2p_0}}s^{-1+\frac{d}{2p_0}}\|\varphi\|_\infty
\end{align*}
since $p_0\in(d,\infty)$ and complete the proof.

\end{proof}

\section{Convergence rate of the total variation}\label{Section4}
In this section, under {\bf (H$_\sigma$)}, {\bf (H$_b^1$)} and {\bf (H$_b^2$)}, we will derive the convergence rate of $\|\mu^\eps-\mu\|_{var}$.
Recall that on the probability space $(\Omega,\sF,\mP,(\sF_s)_{s\ge0})$ we have a unique strong solution $(X^\eps_\cdot,X_\cdot)$ to the following systems
\begin{align}\label{ori:DDSDE}
\dif X^\eps_t=b(t/\eps,X^\eps_t,\mu^\eps_t)\dif t+\sigma(X^\eps_t)\dif W_t, \ \ X^\eps_0=\xi,
\end{align}
and
\begin{align}\label{ave:DDSDE}
\dif X_t=\bar{b}(X_t,\mu_t)\dif t+\sigma(X_t)\dif W_t, \ \ X_0=\xi,
\end{align}
where $\mu^\eps_t$ and $\mu_t$ are the distributions of $X^\eps_t$ and $X_t$ respectively.
For simplicity, in the sequel, let
\begin{align*}
 b_\eps(t):=b(t/\eps)~~(\eps>0), \ \ \text{and}\ \ b_0:=\bar{b}.
\end{align*}
For any $x\in\mR^d$ and $t\geq s\geq0$,
let $(Y^\eps_{s,t}(x),Y_{s,t}(x))$ be the unique
strong solution to the following SDEs
\begin{align*}
\dif Y^\eps_{s,t}(x)=b_\eps(t,Y^\eps_{s,t}(x),\mu^\eps_t)\dif t+\sigma(Y^\eps_{s,t}(x))\dif W_t,\quad Y^\eps_{s,s}(x)=x
\end{align*}
and
\begin{align*}
\dif Y_{s,t}(x)=\bar{b}(Y_{s,t}(x),\mu_t)\dif t+\sigma(Y_{s,t}(x))\dif W_t,\quad Y_{s,s}(x)=x.
\end{align*}
Set $Y^\eps_{t}(x):=Y^\eps_{0,t}(x)$ and $Y_{t}(x):=Y_{0,t}(x)$ for all $t\geq0$ and $x\in\R^d$.
Let $P^{x,\eps}$ and $P^x$ denote the distributions of $Y^{\eps}_\cdot(x)$ and $Y_\cdot(x)$ in $C([0,T];\mR^d)$ respectively.
Based on the strong uniqueness of the above SDEs, we have
\begin{align}\label{DStar00}
\int_{\mR^d} P^{x,\eps}~~\mP\circ\xi^{-1}(\dif x)=\mP\circ (X^\eps_\cdot)^{-1}\quad \text{and}\quad\int_{\mR^d} P^{x}~~\mP\circ\xi^{-1}(\dif x)=\mP\circ (X_\cdot)^{-1}.
\end{align}
Therefore, the estimates in Section \ref{Sec3.1} hold for $X^\eps$ and $X$,
where the constants are independent of $\eps$, since
\begin{align*}
\sup_{\eps\ge0}\sup_{\mu\in\cP(\mR^d)}\nor b_\eps(\cdot,\mu)\nor_{\widetilde{\mL}^{p_0}_T}<\infty.
\end{align*}
Moreover, for any $t\in\mR_+$ and $\varphi\in C^\infty_b$, consider the following Kolmogorov backward equation
%with initial $\varphi\in C_b^\infty(\mR^d)$
\begin{align}\label{BKeq}
\p_s u^t+a_{ij}\p_i\p_j u^t+b_0(\cdot,\mu_s)\cdot\nabla u^t=0,
\end{align}
with final condition
$$u^t(t)=\varphi.$$
By Proposition \ref{well-PDE} and \ref{Pro:PR}, there exists a unique solution $u^t$ to \eqref{BKeq}, which is given by
$$
u^t(s,x)=\mE\varphi(Y_{s,t}(x)).
$$
Define $\tilde{u}(s,x)=u^t(t-s,x)$. Then $\tilde{u}$ is the solution to
\begin{equation*}
  \p_s\tilde{u}=a_{ij}\p_i\p_j \tilde{u}
  +b_0(\cdot,\mu_{t-s})\cdot\nabla \tilde{u}, \quad \tilde{u}_0=\varphi.
\end{equation*}
By Lemma \ref{Lemma3.12}, we have
\begin{align}\label{NN04}
\|\nabla u^t(s)\|_\infty \lesssim (t-s)^{-\frac12}\|\varphi\|_\infty
\end{align}
and
\begin{align}\label{FStar00}
\|\nabla u^t(s_1)-\nabla u^t(s_2)\|_\infty\lesssim |s_1-s_2|^{\frac12-\frac{d}{2p_0}}(s_1\wedge s_2)^{-1+\frac{d}{2p_0}}\|\varphi\|_\infty.
\end{align}
Then, for any $t\in[0,T]$, by applying the generalized It\^o formula \eqref{GIF001} to $u^t(s,Y^\eps_s(x))$, one sees that
\begin{align}\label{NCZ00}
\mE \varphi(Y^\eps_t(x))-u^t(0,x)
=\mE\int_0^t\Big(b_\eps(s,Y^\eps_s(x),\mu^\eps_s)
-b_0(Y^\eps_s(x),\mu_s)\Big)\cdot\nabla u^t(s,Y^\eps_s(x))\dif s.
\end{align}
Noting that $u^t(0,x)=\mE \varphi(Y_t(x))$, by \eqref{DStar00}, we have
\begin{align}\label{NN03}
\begin{split}
|\mE \varphi(X^\eps_t)-\mE \varphi(X_t)|&=\Big|\int_{\mR^d}\(\mE \varphi(Y^\eps_t(x))-\mE \varphi(Y_t(x))\)\mP\circ\xi^{-1}(\dif x)\Big|\\
&\le\sup_x\Big|\mE\int_0^t\Big(b_\eps(s,Y^\eps_s(x),\mu^\eps_s)-b_0(Y^\eps_s(x),\mu_s)\Big)\cdot\nabla u^t(s,Y^\eps_s(x))\dif s\Big|.
\end{split}
\end{align}

Here is the main result of this section:
\bt\label{S4:Main}
Under the conditions {\bf (H$_\sigma$)} and {\bf (H$_b^1$)}--{\bf (H$_b^2$)}, for any $T>0$,
there is a constant $C=C(\kappa_0,\kappa_1,d,T,p_0,\beta)>0$
such that for all $\eps>0$ and $t\in[0,T]$,
\begin{align}\label{LL00}
\|\mu^\eps_t-\mu_t\|_{var}\le C \inf_{h>0}\Big(h^{\frac12-\frac{d}{2p_0}}+\omega\left(\frac{h}{\eps}\right)\Big).
\end{align}
\et
\begin{proof}
For simplicity, in the whole proof, we assume $\|\varphi\|_\infty=1$ and drop the superscript $t$ from $u^t$.
First, let
\begin{align*}
\cB^\eps:=\Big|\mE\int_0^t\Big(b_\eps(s,Y^\eps_s(x),\mu^\eps_s)-b_0(Y^\eps_s(x),\mu_s)\Big)\cdot\nabla u(s,Y^\eps_s(x))\dif s\Big|
\end{align*}
and
\begin{align*}
\cE^\eps_h:=\Big|\mE\int_0^t\Big(b_\eps(s,Y^\eps_{\pi_h(s)}(x),\mu^\eps_{\pi_h(s)})-b_0(Y^\eps_{\pi_h(s)}(x),\mu^\eps_{\pi_h(s)})\Big)\cdot\nabla u({\pi_h(s)},Y^\eps_{\pi_h(s)}(x))\dif s\Big|,
\end{align*}
where $h\in(0,1)$.
For any map $f:\mR_+\times\mR^d\times\cP(\mR^d)\to\mR^d$ and $h\in(0,1)$, define
\begin{align*}
&U_{1,h}^\eps(f):=\Big|\mE\int_0^t\((f\cdot\nabla u)(s,Y^\eps_s(x),\mu^\eps_s)-(f\cdot\nabla u)(s,Y^\eps_{\pi_h(s)}(x),\mu^\eps_s)\)\dif s\Big|\\
&U_{2,h}^\eps(f):=\Big|\mE\int_0^t\Big(
\[f(s,Y^\eps_{\pi_h(s)}(x),\mu^\eps_s)
-f(s,Y^\eps_{\pi_h(s)}(x),\mu^\eps_{\pi_h(s)})\]\cdot\nabla u(s,Y^\eps_{\pi_h(s)}(x))\Big)\dif s\Big|\\
&U_{3,h}^\eps(f):=\Big|\mE\int_0^t\Big(
f(s,Y^\eps_{\pi_h(s)}(x),\mu^\eps_{\pi_h(s)})\cdot\[\nabla u(s,Y^\eps_{\pi_h(s)}(x))-\nabla u(\pi_h(s),Y^\eps_{\pi_h(s)}(x))\]\Big)\dif s\Big|.
\end{align*}
Then, we have
\begin{align*}
\cB^\eps\le&\cE^\eps_h+\sum_{i=1}^3 \[U_{i,h}^\eps(b_\eps)+U_{i,h}^\eps(b_0)\]\\
&
+\Big|\mE\int_0^t\Big(b_0(Y^\eps_s(x),\mu^\eps_s)
-b_0(Y^\eps_s(x),\mu_s)\Big)\cdot\nabla u(s,Y^\eps_s(x))\dif s\Big|.
\end{align*}
It follows from \eqref{NN04}, \eqref{BB06} and {\bf (H$_b^1$)} that
\begin{align*}
&\quad\Big|\mE\int_0^t\Big(b_0(Y^\eps_s(x),\mu^\eps_s)
-b_0(Y^\eps_s(x),\mu_s)\Big)\cdot\nabla u(s,Y^\eps_s(x))\dif s\Big|\\
&\lesssim \int_0^ts^{-\frac{d}{2p_0}}(t-s)^{-\frac12}\|\mu^\eps_s-\mu_s\|_{var}\dif s,
\end{align*}
which implies that
\begin{align}\label{NN05}
\cB^\eps\lesssim  \cE^\eps_h+\sum_{i=1}^3 \[U_{i,h}^\eps(b_\eps)+U_{i,h}^\eps(b_0)\]+\int_0^ts^{-\frac{d}{2p_0}}(t-s)^{-\frac12}\|\mu^\eps_s-\mu_s\|_{var}\dif s.
\end{align}

Now, we divide the rest of the proof into two steps. In Step 1, we estimate $U_{i,h}^\eps(b_\eps)+U_{i,h}^\eps(b_0)$, $i=1,2,3$, one by one; In Step 2, we calculate $\cE^\eps_h$ under the assumption {\bf (H$_b^2$)}.

{\bf(Step 1)} We only estimate $U_{i,h}^\eps(b_\eps)$, for $U_{i,h}^\eps(b_0)$ we can proceed in the same way.
First, we estimate $U_{1,h}^\eps(b_\eps)$. By \eqref{NN000} and \eqref{NN04}, we have
\begin{align*}
U_{1,h}^\eps(b_\eps)
&
=\Big|\int_0^t\(P^{Y^\eps}_s(b_\eps\cdot\nabla u)
(s,\cdot,\mu^\eps_s)(x)-P^{Y^\eps}_{\pi_h(s)}
(b_\eps\cdot\nabla u)(s,\cdot,\mu^\eps_s)(x)\)\dif s\Big|\\
&
\lesssim \int_0^t \[(h^{\alpha}(\pi_h(s))^{-\alpha})\wedge1\]
(\pi_h(s))^{-\frac{d}{2p_0}}\nor (b_\eps\cdot\nabla u)(s,\cdot,\mu^\eps_s)\nor_{p_0}\dif s\\
&
\lesssim h^\alpha\int_0^t (\pi_h(s))^{-\alpha}
(\pi_h(s))^{-\frac{d}{2p_0}}(t-s)^{-\frac12}\dif s,
\end{align*}
where $\alpha=1/2-d/(2p_0)$.
Noting that $\pi_h(s)=s$ for $s\le h$, $\pi_h(s)\le s$ for all $s\in[0,T]$ and $1-\alpha-d/(2p_0)-1/2=0$, one sees that
\begin{align*}
U_{1,h}^\eps(b_\eps)\lesssim h^{\frac12-\frac{d}{2p_0}}.
\end{align*}
For $U_{2,h}^\eps(b_\eps)$, by \eqref{BB06}, \eqref{NN04} and {\bf (H$^1_b$)}, we have
\begin{align*}
U_{2,h}^\eps(b_\eps)\lesssim \int_0^t (\pi_h(s))^{-\frac{d}{2p_0}}\|\mu^\eps_s-\mu^\eps_{\pi_h(s)}\|_{var}(t-s)^{-\frac12}\dif s.
\end{align*}
%Noting that $1/2-d/(2p_0)\le1/2$,
It follows from \eqref{LL02} with $\delta=\frac{d}{2p_0}$ that
\begin{align*}
U_{2,h}^\eps(b_\eps)\lesssim h^{\frac12-\frac{d}{2p_0}}\int_0^t (\pi_h(s))^{-\frac{d}{2p_0}}(t-\pi_h(s))^{-\frac12}\dif s\lesssim h^{\frac12-\frac{d}{2p_0}},
\end{align*}
since $p_0>d$.
Finally, in view of \eqref{BB06} and \eqref{FStar00}, because $p_0<\infty$, we have
\begin{align*}
U_{3,h}^\eps(b_\eps)\lesssim h^{\frac12-\frac{d}{2p_0}}\int_0^t (\pi_h(s))^{-\frac{d}{2p_0}}(t-\pi_h(s))^{-1+\frac{d}{2p_0}}\dif s
\lesssim h^{\frac12-\frac{d}{2p_0}}
\end{align*}
and obtain that
\begin{align}\label{NN06}
\sum_{i=1}^3\(U_{i,h}^\eps(b_\eps)+U_{i,h}^\eps(\bar{b})\)\lesssim h^{\frac12-\frac{d}{2p_0}}.
\end{align}

{\bf(Step 2)} Let $M:=[t/h]$. Without loss of generality we may assume that $M=t/h\in\mN$ and note that
\begin{align*}
\cE^\eps_h\le&\Big|\mE\int_{0}^{h}\Big(b_\eps(s,Y^\eps_{s}(x),\mu^\eps_{s})-b_0(Y^\eps_{s}(x),\mu^\eps_{s})\Big)\cdot\nabla u(s,Y^\eps_{s}(x))\dif s\Big|\\
&
+ \Big|\sum_{k=1}^{M-1}\mE\int_{kh}^{(k+1)h}
\Big(b_\eps(s,Y^\eps_{kh}(x),\mu^\eps_{kh})
-b_0(Y^\eps_{kh}(x),\mu^\eps_{kh})\Big)\cdot\nabla u(kh,Y^\eps_{kh}(x))\dif s\Big|\\
:=&\cE_1+\cE_2.
\end{align*}
From \eqref{BB06} and \eqref{NN04},
\begin{align*}
\cE_1\lesssim \int_{0}^{h} s^{-\frac{d}{2p_0}}(t-s)^{-\frac12}\dif s\lesssim h^{\frac12-\frac{d}{2p_0}}.
\end{align*}
By \eqref{NN04} and a change of variables, one sees that
\begin{align*}
\cE_2&\lesssim \sum_{k=1}^{M-1}(t-kh)^{-\frac12}\Big|\mE\int_{kh}^{(k+1)h}\Big(b_\eps(s,Y^\eps_{kh}(x),\mu^\eps_{kh})-b_0(Y^\eps_{kh}(x),\mu^\eps_{kh})\Big)\dif s\Big|\\
&\lesssim \sum_{k=1}^{M-1}(t-kh)^{-\frac12}\Big|\eps\mE\int_{kh/\eps}^{(k+1)h/\eps}\Big(b(s,Y^\eps_{kh}(x),\mu^\eps_{kh})-\bar{b}(Y^\eps_{kh}(x),\mu^\eps_{kh})\Big)\dif s\Big|.
\end{align*}
Based on the assumptions \eqref{in:00} and \eqref{BB06}, we have
\begin{align*}
\cE_2
&
\lesssim h\sum_{k=1}^{M-1}(t-kh)^{-\frac12}\omega\left(\frac{h}{\eps}\right)\mE H(Y^\eps_{kh}(x),\mu^\eps_{kh})\\
&
\lesssim h\sum_{k=1}^{M-1}(t-kh)^{-\frac12}\omega\left(\frac{h}{\eps}\right) (kh)^{-\frac{d}{2p_0}} \sup_\mu\nor H(\cdot,\mu)\nor_{p_0}\\
&
\lesssim \omega\left(\frac{h}{\eps}\right) \int_h^t (t-\pi_h(s))^{-\frac12}(\pi_h(s))^{-\frac{d}{2p_0}}\dif s
\lesssim \omega\left(\frac{h}{\eps}\right)
\end{align*}
and obtain that
\begin{align*}
|\mE \varphi(X^\eps_t)-\mE \varphi(X_t)|\lesssim\(h^{\frac12-\frac{d}{2p_0}}+\omega\left(\frac{h}{\eps}\right)\)+\int_0^ts^{-\frac{d}{2p_0}}(t-s)^{-\frac12}\|\mu^\eps_s-\mu_s\|_{var}\dif s
\end{align*}
because of \eqref{NN03}, \eqref{NN05} and \eqref{NN06}.
Finally, taking the supremum over all $\varphi\in C^\infty_b$ with $\|\varphi\|_\infty=1$ and
by Gronwall's inequality of Volterra type (see \cite[Example 2.4]{Zhang2010}), we complete the proof.

\end{proof}

\section{Proof of Theoerm \ref{in:Main} and \ref{in:Main2}}\label{Section5}
In this section, we consider the process $(X^\eps,W,X)$ on the probability space $(\Omega,\sF,(\sF_t)_{t\ge0}, \mP)$
which satisfies the following system in $\mR^d$:
\begin{align*}
X^\eps_t=\xi+\int_0^t b(\frac{s}{\eps},X^\eps_s,\mu^\eps_s)\dif s+\int_0^t\sigma(X^\eps_s)\dif W_s
\end{align*}
and
\begin{align*}
X_t=\xi+\int_0^t \bar{b}(X_s,\mu_s)\dif s+\int_0^t\sigma(X_s)\dif W_s,
\end{align*}
where $W$ is a standard $d$-dimensitional Brownian Motion, $\mu^\eps_t$ and $\mu_t$ are the distributions of $X^\eps_t$ and $X_t$ respectively and $(b,\bar{b},\sigma)$ satisfies the conditions {\bf (H$_b^1$)}-{\bf (H$_b^2$)} and {\bf (H$_\sigma$) }.
We set
\begin{align*}
X^0:=X,\ \ b_\eps(t):=b(t/\eps),\ \ \text{and} \ \ b_0:=\bar{b}.
\end{align*}

\begin{proof}[Proof of Theoerm \ref{in:Main}]
Set
\begin{align*}
B(t,x):=b_0(x,\mu_t)
\end{align*}
and consider the following backward parabolic PDE
$$
\p_tu+a_{ij}\p_i\p_j u-\lambda u+B\cdot\nabla u+B=0, ~~t\in[0,T],\ \ u(T)=0.
$$
Since $\nor B\nor_{\widetilde{\mL}^{p_0}(T)}
\le\sup\limits_\mu\nor \bar{b}(\cdot,\mu)\nor_{p_0}<\infty $, by Lemma \ref{lem210}, for $\lambda$ large enough there is a unique solution $u$ in the sense of Definition \ref{Pre4:Def} satisfying
\begin{align*}%\label{nablau1}
\|\nabla u\|_{\mL^\infty_T}\le \frac{1}{2}
\end{align*}
and for any $2/q+d/p_0<1$,
\begin{align*}
\nor\nabla^2 u\nor_{\widetilde{\mL}^{p_0}_q(T)}\le C,
\end{align*}
  which implies that for any $\lambda>0$
  \begin{align}\label{Khau1}
  \sup_{\eps\ge0}\mE \exp\left(\lambda \int_0^T|\nabla^2 u(t,X^{\eps}_t)|^2\dif t\right)<\infty,
\end{align}
where $X^0:=X$, because of \eqref{Kh02}. Moreover, by \eqref{HH02}, for all $s,t\in[0,T]$,
  \begin{align}\label{GG051}
  \|\nabla u(t)-\nabla u(s)\|_\infty \lesssim |t-s|^{1/2-d/(2p_0)}.
\end{align}
Define
$$
\Phi_t(x):=x+u(t,x)
$$
and
$$
Y^{\eps}_t:=\Phi_t({X}^{\eps}_t),\ \ Y_t:=\Phi_t(X_t).
$$
Then $\Phi_t$ is a $C^1$-diffeomorphism (see Remark \ref{Re002}) for any $t\in[0,T]$ with
\begin{align}\label{ZZ031}
\|\nabla\Phi\|_{\mL^\infty_T}+\|\nabla \Phi^{-1}\|_{\mL^\infty_T}\le 4.
\end{align}
By the generalized It\^o formula \eqref{GIF001}, we have
\begin{align*}
\dif Y_t=\lambda u(t,{X}_t)\dif t+(\sigma^*\nabla\Phi_t)({X}_t)\dif W_t
\end{align*}
and
\begin{align*}
\dif Y^{\eps}_t=\lambda u(t,X^\eps_t)\dif t+\(b_\eps(t,X^\eps_t,\mu^\eps_t)-b_0(X^\eps_t,\mu_t)\)\cdot\nabla \Phi_t\(X^{\eps}_t\)\dif t+(\sigma^*\nabla\Phi_t(X^\eps_t))\dif W_t,
\end{align*}
where $\sigma^*$ is the transpose of $\sigma$.
It follows from \eqref{ZZ031} that for any $t\in[0,T]$,
\begin{align*}
|X^\eps_t-X_t|^2\lesssim |Y^\eps_t-Y_t|^2
&
\lesssim \lambda\|\nabla u\|_{\mL^\infty_T}^2\int_0^t |X^\eps_s-X_s|^2\dif s\\
&
+\left[\int_0^t\((\sigma^*\nabla\Phi_s)(X^\eps_s)-(\sigma^*\nabla\Phi_s)(X_s)\)\dif W_s\right]^2\\
&
+\Big|\int_0^t\(b_\eps(s,X^\eps_s,\mu^\eps_s)-b_0(X^\eps_s,\mu_s)\)\cdot\nabla \Phi_s\(X^{\eps}_s\)\dif s\Big|^2.
\end{align*}
Set
\begin{align*}
A^{\eps}_t:=&\int_0^t\left(\cM(\nabla^2 u)(s,X_s)+\cM(\nabla^2 u)(s,X^\eps_s)+\|\nabla u\|_{\mL^\infty_T}\right)^2\dif s\\
&+\int_0^t\left(\cM(\nabla \sigma)(X_s)+\cM(\nabla \sigma)(X^\eps_s)+\|\sigma\|_\infty\right)^2\dif s
\end{align*}
and
\begin{align*}
\eta^\eps_t:=\Big|\int_0^t\(b_\eps(s,X^\eps_s,\mu^\eps_s)-b_0(X^\eps_s,\mu_s)\)\cdot\nabla \Phi_s\(X^{\eps}_s\)\dif s\Big|^2.
\end{align*}
Then, by \eqref{Pre1:Mix}, \eqref{Khau1} and {\bf (H$_\sigma$)}, we have
\begin{align}\label{Coffee}
\sup_\eps\mE \exp(A^{\eps}_T)<\infty.
\end{align}
We note that by \eqref{Pre1:Mnf}
\begin{align*}
&\quad
\left[\int_0^t\((\sigma^*\nabla\Phi)(X^\eps_s)-(\sigma^*\nabla\Phi)(X_s)\)\dif W_s\right]^2\\
&
\le \int_0^t|X^\eps_s-X_s|^2\dif A^{\eps}_s+M^\eps_t,
\end{align*}
where $M^\eps$ is a martingale.
Altogether, we have
\begin{align*}
|X^\eps_t-X_t|^2\lesssim\int_0^t|X^\eps_s-X_s|^2\dif s+\int_0^t|X^\eps_s-X_s|^2\dif A^{\eps}_s+M^\eps_t+\eta_t^\eps.
\end{align*}
Hence,
 by \eqref{Coffee} and the Stochastic Gronwall inequality (see Lemma 2.8 in \cite{Zh-Zh1} or \cite{Sc13}), one sees that for any $\ell\in(0,1)$,
\begin{align*}
\mE\(\sup_{t\in[0,T]}|X^\eps_t-X_t|^{2\ell}\)\lesssim \(\mE[\sup_{t\in[0,T]} \eta_t^\eps]\)^\ell.
\end{align*}

Combining \eqref{GG051}, \eqref{ZZ031} and \eqref{EE01}, we have
\begin{align}\label{GG07}
\mE[\sup_{t\in[0,T]}\eta_t^\eps]\lesssim
&
\mE\int_0^T\Big|b_0(X^\eps_s,\mu^\eps_s)-b_0(X^\eps_s,\mu_s)\Big|^2\dif s+ \inf_{h>0}\Big(h^{1-\delta}+h^{1-\frac{d}{p_0}}+\omega\left(\frac{h}{\eps}\right)\Big).
\end{align}
Taking $\delta<d/p_0$ in \eqref{GG07}, from \eqref{Kry02} and \eqref{LL00},
for any $2/q+d/p_0<1$, one sees that
\begin{align*}
\mE[\sup_{t\in[0,T]}\eta_t^\eps]
&
\lesssim\Big(\int_0^T\|\mu^\eps_s-\mu_s\|_{var}^q\dif s\Big)^{2/q} +\inf_{h>0}\Big(h^{1-\frac{d}{p_0}}
+\left(\omega\left(\frac{h}{\eps}\right)\right)^2\Big)\\
&
\lesssim\inf_{h>0}\Big(h^{1-\frac{d}{p_0}}
+\left(\omega\left(\frac{h}{\eps}\right)\right)^2\Big)
\end{align*}
and this completes the proof.

\end{proof}

%\section{Proof of Theoerm \ref{in:Main2}}\label{Section6}
In the rest of this section, we assume that
\begin{align*}
b_\eps(t,x,\mu)=b_\eps(t,x),\quad b_0(x,\mu)=b_0(x)
\end{align*}
and prove Theorem \ref{in:Main2}. The method is the same as the one of Theorem \ref{in:Main}, except for using the elliptic equation to construct the Zvonkin's transformation instead of the parabolic.

\begin{proof}[Proof of Theorem \ref{in:Main2}]
Consider the following elliptic PDE
\begin{align}\label{Ell00}
a_{ij}\p_i\p_j u-\lambda u+b_0\cdot\nabla u+b_0=0.
\end{align}
Noting that $\nor b_0\nor_{\widetilde{\mL}^{p_0}(T)}\le\|b\|_{L^\infty(\mR_+;\widetilde{L}^{p_0})} $ and by \eqref{CC02} for $\lambda$ large enough, we have
\begin{align*}%\label{nablau}
\|\nabla u\|_{\infty}\le \frac{1}{2}
\end{align*}
and
\begin{align*}
\nor\nabla^2 u\nor_{p_0}\le C \|b_0\|_{\widetilde{L}^{p_0}},\quad \forall p_0>d
\end{align*}
It follows by \eqref{Kh02} that for any $\lambda>0$
  \begin{align}\label{Khau}
  \sup_{\eps\ge0}\mE \exp\left(\lambda \int_0^T|\nabla^2 u(X^{\eps}_t)|^2\dif t\right)<\infty.
\end{align}
Define
$$
\Phi(x):=x+u(x)
$$
and
$$
Y^{\eps}_t:=\Phi({X}^{\eps}_t),\ \ Y_t:=\Phi(X_t).
$$
Then $\Phi$ is a $C^1$-diffeomorphism. Again by the generalized It\^o formula \eqref{GIF001}, we have
\begin{align*}
\dif Y_t=\lambda u({X}_t)\dif t+(\sigma^*\nabla\Phi)({X}_t)\dif W_t
\end{align*}
and
\begin{align*}
\dif Y^{\eps}_t=\lambda u(X^\eps_t)\dif t+\(b_\eps(t,X^\eps_t)-b_0(X^\eps_t)\)\cdot\nabla \Phi\(X^{\eps}_t\)\dif t+(\sigma^*\nabla \Phi)(X^\eps_t)\dif W_t.
\end{align*}
Then, we have
\begin{align*}
|X^\eps_t-X_t|^2\lesssim \int_0^t |X^\eps_s-X_s|^2\dif A^\eps_s+M^\eps_t+\eta_t^\eps,
\end{align*}
where $(M^\eps_t)_{t\ge0}$ is a martingale,
\begin{align*}
A^{\eps}_t=&t+\int_0^t\left(\cM(\nabla^2 u)(X_s)+\cM(\nabla^2 u)(X^\eps_s)+\|\sigma\|_\infty\right)^2\dif s\\
&+\int_0^t\left(\cM(\nabla \sigma)(X_s)+\cM(\nabla \sigma)(X^\eps_s)+\|\nabla u\|_{\infty}\right)^2\dif s
\end{align*}
and
\begin{align*}
\eta^\eps_t=\Big|\int_0^t\(b(s/\eps,X^\eps_s)-\bar {b}(X^\eps_s)\)\cdot\nabla \Phi\(X^{\eps}_s\)\dif s\Big|^2.
\end{align*}
Then, in view of \eqref{Pre1:Mix} and \eqref{Khau}, we have
\begin{align*}
\sup_\eps\mE \exp(A^{\eps}_T)<\infty,
\end{align*}
which implies that for any $\ell\in(0,1)$,
\begin{align*}
\mE\(\sup_{t\in[0,T]}|X^\eps_t-X_t|^{2\ell}\)\lesssim \(\mE\sup_{t\in[0,T]} \eta_t^\eps\)^\ell
\end{align*}
because of the Stochastic Gronwall inequality and we complete the proof by \eqref{EE01} with $\alpha=1$.
\end{proof}

\section*{}
\subsection*{Acknowledgments}
The first author would like to acknowledge the
warm hospitality of Bielefeld University. And
we would also like to thank Dr. Chengcheng Ling for many useful discussions.

\begin{appendix}
\renewcommand{\thetable}{A\arabic{table}}
\numberwithin{equation}{section}

\section{}
In this appendix, we prove the claim \eqref{NNA00} in Example \ref{Ex2}. To this end, we need the following lemma.
\bl\label{lem61}
Let
$$
h(t):=\int_{\R} B(\sin(\xi t))\nu(\dif \xi).
$$
and
$$
\bar{h}:=\Big(\frac{1}{2\pi}\int_0^{2\pi}
B(\sin(\xi \tau))\dif\tau\Big)\nu(\R\setminus\{0\})
+B(0)\nu(\{0\}),
$$
where $B: [-1,1]$ is measurable and $\nu$ is a finite measure on $\mR$.
Assume that there is a constant $C_B>0$ such that
\begin{align*}
|B(u)|\le C_B,\quad \forall u\in[-1,1].
\end{align*}
Then, for any $t,T\in\mR_+$,
\begin{align*}
\Big|\frac{1}{T}\int_{t}^{t+T}\(h(s)-\bar{h}\)\dif s\Big|\le \frac{4\pi C_B}{T}\int_{\mR\setminus\{0\}}\frac{\nu(\dif \xi)}{|\xi|}.
\end{align*}

\el

\begin{proof}[Proof of Lemma \ref{lem61}]
If $\int_{\mR\setminus\{0\}}\frac{\nu(\dif \xi)}{|\xi|}=\infty$, this is trivial. So, we assume that $\int_{\mR\setminus\{0\}}\frac{\nu(\dif \xi)}{|\xi|}<\infty$. First, one sees that
\begin{align*}
\sI
&
:=\left|\frac{1}{T}\int_{t}^{t+T}\(h(s)
-\bar{h}\)\dif s\right|\\
&
=\left|\frac{1}{T}\int_{t}^{t+T}\int_{\mR\setminus\{0\}} B(\sin(\xi s))\nu(\dif\xi)\dif s-\frac{1}{2\pi}\int_t^{t+2\pi} B(\sin(\tau))\nu(\mR\setminus\{0\})\dif \tau\right|\\
&=\left|\int_{\mR\setminus\{0\}} \Big[\frac{1}{T}\int_{t}^{t+T}B(\sin(\xi s))\dif s-\frac{1}{2\pi}\int_t^{t+2\pi} B(\sin(\tau))\dif \tau\Big]\nu(\dif\xi)\right|,
\end{align*}
by Fubini's theorem. From a change of variable, we have
\begin{align*}
\sI
=&\left|\int_{\mR\setminus\{0\}} \Big[\frac{1}{T|\xi|}\int_{t\xi}^{(t+T)\xi}B(\sin s)\dif s-\frac{1}{2\pi}\int_{t}^{t+2\pi} B(\sin \tau)\dif \tau\Big]\nu(\dif\xi)\right|,
%\le&\left|\int_0^\infty \Big[\frac{1}{T\xi}\int_{t\xi}^{(t+T)\xi}B(\sin s)\dif s-\frac{1}{2\pi}\int_{t}^{t+2\pi} F\left(\sin \tau,
%\Phi(x)\right)\dif \tau\Big]\nu(\dif\xi)\right|\\
%&+\left|\int_{-\infty}^0 \Big[\frac{1}{T|\xi|}\int_{(t+T)\xi}^{t\xi}B(\sin s)\dif s-\frac{1}{2\pi}\int_{t}^{t+2\pi} F\left(\sin \tau,
%\Phi(x)\right)\dif \tau\Big]\nu(\dif\xi)\right|
\end{align*}
where $\int_a^b:=-\int_b^a$ if $a>b$. %$\int_0^\infty:=\int_{(0,\infty)}$ and $\int_{-\infty}^0:=\int_{(-\infty,0)}$.
Set
\begin{align*}
G:=\int_{0}^{2\pi} B(\sin s)\dif s=\int_{t}^{t+2\pi} B(\sin s)\dif s,\quad\forall t\in\mR.
\end{align*}
Then,  noting that $s\to\sin s$ has a period $2\pi$, we have
\begin{align*}
\int_{t\xi}^{(t+T)\xi}B(\sin s)\dif s&=\Big[\frac{T|\xi|}{2\pi}\Big]G+\int_{\sgn(\xi)\left[\frac{T|\xi|}{2\pi}\right]2\pi+t\xi}^{T\xi+t\xi} B(\sin s)\dif s\\
&:=\Big[\frac{T|\xi|}{2\pi}\Big]G+H_t(\xi)
\end{align*}
where $\sgn(\xi):=\xi/|\xi|$, which implies that
\begin{align*}
\sI
=&\left|\int_{\mR\setminus\{0\}} \Big[\frac{1}{T|\xi|}\Big(\Big[\frac{T|\xi|}{2\pi}\Big]G+H_t(\xi)\Big)-\frac{1}{2\pi}G\Big]\nu(\dif\xi)\right|\\
\le&\int_{\mR\setminus\{0\}} \Big|\frac{1}{T|\xi|}\Big[\frac{T|\xi|}{2\pi}\Big]-\frac{1}{2\pi}\Big|\nu(\dif\xi)G+\int_{\mR\setminus\{0\}} \frac{1}{T|\xi|}H_t(\xi)\nu(\dif\xi).
\end{align*}
We note that
\begin{align*}
\Big|\frac{1}{T|\xi|}\Big[\frac{T|\xi|}{2\pi}\Big]-\frac{1}{2\pi}\Big|=\frac{1}{T|\xi|}\Big|\Big[\frac{T|\xi|}{2\pi}\Big]-\frac{T|\xi|}{2\pi}\Big|\le \frac{1}{T|\xi|}
\end{align*}
and
\begin{align*}
G\vee H_t(\xi)\le \int_0^{2\pi}|B(\sin s)|\dif s\le 2\pi C_B.
\end{align*}
Therefore, we have
\begin{align*}
\sI\le \frac{4\pi C_B}{T}\int_{\mR\setminus\{0\}}\frac{\nu(\dif \xi)}{|\xi|}
\end{align*}
and complete the proof.

\end{proof}
Now we can give the
\begin{proof}[Proof of \eqref{NNA00}]
Since {\bf (H$_b^1$)} holds for $b$ obviously,
it suffice to show that {\bf (H$_b^2$)} holds.
We note that in Example \ref{Ex2}
\begin{align*}
|F(u,x)|\le|F(u,0)|+|F(u,0)-F(t,x)|\le L_F+L_F|x|
\end{align*}
because of \eqref{in:EX2}, which implies that
\begin{align*}
|F\(u,\int_{\mR^d}\phi(x,y)\mu(\dif y)\)|\le L_F(1+\int_{\mR^d}|\phi(x,y)|\mu(\dif y)).
\end{align*}
Hence, by Lemma \ref{lem61}, we see that
\begin{align*}
\left|\frac{1}{T}\int_{t}^{t+T}\(b(s,x,\mu)-\bar{b}(x,\mu)\)\dif s\right|\le \frac{4\pi L_F}{T}\int_{\mR\setminus\{0\}}\frac{\nu(\dif \xi)}{|\xi|}H(x,\mu),
\end{align*}
where $H(x,\mu)=1+\int_{\mR^d}|\phi(x,y)|\mu(\dif y)$. It is easy to see that
\begin{align*}
\sup_\mu\nor H(\cdot,\mu)\nor_{p_0}\le \|1\|_\infty+\int_{\mR^d}\nor\phi(\cdot,y)\nor_{p_0}\mu(\dif y)\le1+\sup_{y}\nor\phi(\cdot,y)\nor_{p_0}.
\end{align*}
This completes the proof.
\end{proof}

\end{appendix}


\begin{thebibliography}{999}

\bibitem{BK2004}
Bakhtin V. and Kifer Y.,
Diffusion approximation for slow motion in fully coupled averaging.
{\em Probab. Theory Related Fields} {\bf129} (2004), 157--181.

\bibitem{BR2018}
Barbu V. and R\"ockner M.,
Probabilistic representation for solutions to nonlinear Fokker-Planck equations. 
{\em SIAM J. Math. Anal.} {\bf50} (2018), 4246--4260.

\bibitem{BR2020}
Barbu V. and R\"ockner M.,
From nonlinear Fokker-Planck equations to solutions of distribution dependent SDE. 
{\em Ann. Probab.} {\bf48} (2020), 1902--1920.

\bibitem{BR2022A}
Barbu V. and R\"ockner M.,
Nonlinear Fokker-Planck equations with fractional Laplacian and McKean-Vlasov SDEs with L\'evy-Noise. 
{\em arXiv:2210.05612} (2022).  


\bibitem{BR2022B}
Barbu V. and R\"ockner M.,
Uniqueness for nonlinear Fokker-Planck equations and for McKean-Vlasov SDEs: The degenerate case. 
{\em arXiv:2203.00122} (2022). 


\bibitem{BM1961}
Bogoliubov N. N. and Mitropolsky Y. A.,
{\em Asymptotic Methods in the Theory of Non-linear Oscillations.}
Translated from the second revised Russian edition International Monographs on Advanced Mathematics and Physics Hindustan Publishing Corp., Delhi, Gordon and Breach Science Publishers, Inc., New York 1961 x+537 pp.

\bibitem{CD13}
Carmona, R. and Delarue, F.,
Probabilistic analysis of mean-field games.
{\em SIAM J. Control Optim.} {\bf51} (2013), 2705--2734.

%\bibitem{CD18a}
%Carmona R. and Delarue F.,
%{\em Probabilistic theory of mean field games with applications. I. Mean field FBSDEs, control, and games.
%Probability Theory and Stochastic Modeling},{\bf 83}. { Springer, Cham,} 2018. xxv+713 pp.

\bibitem{CD18b}
Carmona R. and Delarue F.,
{\em Probabilistic theory of mean field games with applications. II. Mean field games with common noise and master equations.
Probability Theory and Stochastic Modeling}, {\bf 84}. { Springer, Cham,} 2018. xxiv+697 pp.

\bibitem{CGPS20}
Carrillo J. A., Gvalani R. S., Pavliotis G. A. and Schlichting A.,
Long-time behavior and phase transitions for the McKean-Vlasov equation on the tours.
{\em Arch. Ration. Mech. Anal.} {\bf 235} (2020), 635--690.

\bibitem{Cerr2009}
Cerrai S.,
A Khasminskii type averaging principle for stochastic reaction-diffusion equations.
{\em Ann. Appl. Probab.} {\bf19} (2009), 899--948.


\bibitem{CF2009}
Cerrai S. and Freidlin M.,
Averaging principle for a class of stochastic reaction-diffusion equations.
{\em Probab. Theory Related Fields} {\bf144} (2009), 137--177.


\bibitem{CL2017}
Cerrai S. and Lunardi A.,
Averaging principle for nonautonomous slow-fast systems
of stochastic reaction-diffusion equations: the almost periodic case.
{\em SIAM J. Math. Anal.} {\bf49} (2017), 2843--2884.

%\bibitem{CD21a} Chaintron L.P. and Diez A.,
%Propagation of chaos: a review of models, methods and applications. II. Applications.
%{\em arXiv:2106.14812} (2021).

\bibitem{CHXZ}
Chen Z.--Q., Hu E., Xie L. and Zhang X.,
Heat kernels for non-symmetric diffusion operators with jumps.
{\em J. Differential Equations} {\bf263} (2017), 6576--6634.



\bibitem{CL2021}
Cheng M. and Liu Z.,
The second Bogolyubov theorem and global averaging principle for SPDEs with monotone coefficients.
To appear in {\em  SIAM J. Math. Anal.} 
Preprint version available at {\em https://arxiv.org/abs/2109.00371}.


\bibitem{CD08}
Crippa G. and De Lellis D.,
Estimates and regularity results for the DiPerna-Lions flow.
{\em J. Reine Angew. Math.} {\bf616} (2008), 15--46.

\bibitem{DG20}
Dareiotis K. and Gerencs\'er M.,
On the regularisation of the noise for the Euler-Maruyama scheme with irregular drift.
{\it Electron. J. Probab.} {\bf25} (2020) 1-18. %https://doi.org/10.1214/20-EJP479

\bibitem{DW2014}
Duan J. and Wang W.,
{\em Effective Dynamics of Stochastic Partial Differential Equations.}
Elsevier Insights. Elsevier, Amsterdam, 2014, xii+270 pp.

\bibitem{FW2012}
Freidlin M. I. and Wentzell A. D.,
{\em Random Perturbations of Dynamical Systems.}
Translated from the 1979 Russian original by Joseph Sz\"ucs.
Third edition. Grundlehren der Mathematischen Wissenschaften
[Fundamental Principles of Mathematical Sciences], 260.
Springer, Heidelberg, 2012, xxviii+458 pp.

%\bibitem{Fried1975}
%Friedman A.,
%{\em Stochastic Differential Equations and Applications.}
%Vol. 1. Probability and Mathematical Statistics, Vol. 28. Academic Press
%[Harcourt Brace Jovanovich, Publishers], New York-London, 1975. xiii+231 pp.

%\bibitem{Fu}
%Funaki T.,
%A certain class of diffusion processes associated with nonlinear parabolic equations.
%{\em Prob. Theory and Relat. Fields} {\bf67} (1984), no. 3, 331-348.

\bibitem{ML2020}
Hairer M. and Li X.--M.,
Averaging dynamics driven by fractional Brownian motion.
{\em Ann. Probab.} {\bf48} (2020), 1826--1860.

\bibitem{HSS18}
Hammersley W., Siska D. and Szpruch L.,
McKean-Vlasov SDEs under measure dependent Lyapunov conditions.
{\em Ann. Inst. Henri Poincar\'e Probab. Stat.} {\bf57} (2021), 1032--1057.

\bibitem{Ha22}
Han Y.,
Solving McKean-Vlasov SDEs via relative entropy.
{\em arXiv:2204.05709} (2022).

\bibitem{HRZ2021}
Hao Z., R\"ockner M. and Zhang X.,
Euler scheme for density dependent stochastic differential equations. 
{\em J. Differential Equations} {\bf274} (2021), 996--1014.


\bibitem{HRZ22}
Hao Z., R\"ockner M. and Zhang X.,
Strong convergence of propagation of chaos for McKean-Vlasov SDEs with singular interactions.
{\em arXiv:2204.07952} (2022).

\bibitem{HLL2022}
Hong W., Li S. and Liu W.,
Strong convergence rates in averaging principle for slow-fast McKean-Vlasov SPDEs.
{\em J. Differential Equations} {\bf316} (2022), 94--135.

%\bibitem{Jour1995}
%Jourdain, B. Diffusions with a nonlinear irregular drift coefficient and probabilistic interpretation of generalized Burgers' equations.
%{\em ESAIM Probab. Statist.} {\bf1} (1995/97), 339--355.

%\bibitem{Kac}Kac M.,
% Foundations of Kinetic Theory. {\em In Proceedings of the Third Berkeley
%Symposium on Mathematical Statistics and Probability, Volume 3: Contributions to Astronomy and Physics,
%pages} 171-197,{\em Berkeley, Calif.}, 1956. {\em University of California Press.}


\bibitem{Khas1968}
Khasminshii R. Z.,
On the principle of averaging the It\^o's stochastic differential equations. (Russian)
{\it Kybernetika (Prague)} {\bf4} (1968), 260--279.

\bibitem{Kif2004}
Kifer Y.,
Some recent advances in averaging.
{\em Modern dynamical systems and applications, 385-403, Cambridge Univ. Press, Cambridge}, 2004.


\bibitem{KB1943}
Krylov N. V. and Bogolyubov N. N.,
{\em Introduction to Non-Linear Mechanics.}
Annals of Mathematics Studies, No. 11 Princeton University Press, Princeton, N. J., 1943. iii+105 pp.


\bibitem{KR05}
Krylov N. V. and R\"ockner M.,
Strong solutions of stochastic equations with singular time dependent drift.
{\it Probab. Theory Relat. Fields} {\bf 131} (2005), 154--196.

\bibitem{La18}
Lacker D., On a strong form of propagation of chaos for McKean-Vlasov equations, {\it Electron. Commun. Probab.} {\bf23} (2018), 1--11.

\bibitem{Lack2021}
Lacker D.,
Hierarchies, entropy, and quantitative propagation of chaos for mean field diffusions.
{\em arXiv:2105.02983} (2021).

\bibitem{LL21}
L\^e K. and Ling C.,
Taming singular stochastic differential equations: A numerical method.
{\em arXiv:2110.01343} (2021).


\bibitem{MSV1991}
Maslowski B., Seidler J. and Vrko$\check{\text{c}}$ I.,
An averaging principle for stochastic evolution equations. II.
{\em Math. Bohem.} {\bf116} (1991), 191--224.

\bibitem{Mc} McKean H. P.,
 A class of Markov processes associated with nonlinear parabolic equations. {\em Proc Nat. Acad Sci USA} {\bf56} (1966), 1907--1911.

\bibitem{MV2020}
Mishura Y. S. and Veretennikov A. Yu.,
Existence and uniqueness theorems for solutions of McKean--Vlasov stochastic equations. {\it Theory Probab. Math. Statist.} {\bf103} (2020), 59--101. 



\bibitem{PIX2021}
Pei B., Inahama Y. and Xu Y.,
Averaging principle for fast-slow system driven by mixed fractional Brownian rough path.
{\em J. Differential Equations} {\bf301} (2021), 202--235.

%\bibitem{LX21}
%Ling C., Xie L. Strong Solutions of Stochastic Differential Equations with Coefficients in Mixed-Norm Spaces. Potential Anal (2021).


%\bibitem{Mc} McKean H. P.: A class of Markov processes associated with nonlinear parabolic equations. {\it Proc Nat. Acad Sci USA}, 56(6):1907-1911, 1966.

%\bibitem{Mi-Ve}Mishura Y.S. and Veretennikov A.Y.: Existence and uniqueness theorems for
%solutions of McKean-Vlasov stochastic equations, arXiv:1603.02212v4.
%\bibitem{La21}Lacker D.: Hierarchies, entropy, and quantitative propagation of chaos for mean field diffusions. arXiv: 2105.02983.

\bibitem{RSX19}
R\"ockner M., Sun X. and Xie L.,
Strong and weak convergence in the averaging principle for SDEs with H\"older coefficients.
{\em arXiv:1907.09256} (2019).

\bibitem{RSX2021}
R\"ockner M., Sun X. and Xie Y.,
Strong convergence order for slow-fast McKean-Vlasov stochastic differential equations.
{\em Ann. Inst. Henri Poincar\'e Probab. Stat.} {\bf57} (2021), 547--576.

%\bibitem{RX2021}
%R\"ockner M. and Xie L.,
%Averaging principle and normal deviations for multiscale stochastic systems.
%{\em Comm. Math. Phys.} {\bf383} (2021), 1889--1937.

\bibitem{RXY2020}
R\"ockner M., Xie L. and Yang l.,
Asymptotic behavior of multiscale stochastic partial differential equations. 
{\em arXiv:2010.14897} (2020).

\bibitem{RZ21}
R\"ockner M. and Zhang X.,
Well-posedness of distribution dependent SDEs with singular drifts.
{\em Bernoulli} {\bf27} (2021), 1131--1158.

\bibitem{SV1985}
Sanders J. A. and Verhulst F.,
{\em Averaging Methods in Nonlinear Dynamical Systems.}
Applied Mathematical Sciences, 59. Springer-Verlag, New York, 1985. x+247 pp.

\bibitem{Sc13}
Scheutzow M.,
A stochastic Gronwall lemma.
{\em Infinite Dimens. Anal. Quantum Probab. Rel. Top.} {\bf16} (2013), 1350019, 4pp.

\bibitem{SXW2022}
Shen G., Xiang J. and Wu J.,
Averaging principle for distribution dependent stochastic differential equations driven by fractional Brownian motion and standard Brownian motion.
{\em J. Differential Equations} {\bf321} (2022), 381--414.

%\bibitem{ST1985}
%Shiga T. and Tanaka H.,
%Central limit theorem for a system of Markovian particles with mean field interactions.
%{\em Z. Wahrsch. Verw. Gebiete} {\bf69} (1985), 439--459.

%\bibitem{Skor1989}
%Skorokhod A. V.,
%{\em Asymptotic Methods in the Theory of Stochastic Differential Equations.}
%Translated from the Russian by H. H. McFaden. Translations of Mathematical Monographs, 78. American Mathematical Society, Providence, RI, 1989. xvi+339 pp.

%\bibitem{St70}
%Stein E. M.,
%Singular integrals and differentiability property of functions. Princeton Mathematical Series, No. 30 {\it Princeton University Press, Princeton, N.J.} 1970.
%\bibitem{Sc}Scheutzow M.: A stochastic Gronwall's lemma. {\it Infinite Dimensional Analysis, Quantum Probability
%and Related Topics}, {\bf 16}, No. 2 (2013) 1350019 (4 pages).

\bibitem{Sz}Sznitman A.--S.,  Topics in propagation of chaos.
In {\em\'Ecole d'\'Et\'e de Prob. de Saint-Flour XIX}-1989.
 {\em Lect. Notes in Math.} {\bf 1464} 165-251. Berlin: Springer-Verlag 1991.

%\bibitem{To20}
%Tomasevic M. Propagation of chaos for stochastic particle systems with singular mean-field inter- action of $L^q$-$L^p$ type. 2020. hal-03086253

%\bibitem{Vere1990}
%Veretennikov A. Yu.,
%On an averaging principle for systems of stochastic differential equations.
%{\em Mat. Sb.} {\bf181} (1990), 256--268. (in Russian)
%[English translation: {\em Math. USSR-Sb.} {\bf69} (1991), 271--284]

\bibitem{Vere1999}
Veretennikov A. Yu.,
On large deviations in the averaging principle for SDEs with a ``full dependence''.
{\em Ann. Probab.} {\bf27} (1999), 284--296.

%\bibitem{Va}
%Vlasov A. A.,
 %The vibrational properties of an electron gas.
 %{\em Sov. Phys., Usp.} {\bf10} (1968), 721--733.

\bibitem{Wang2018}
Wang F.--Y.,
Distribution dependent SDEs for Landau type equations. {\em Stochastic Process. Appl.} {\bf128} (2018), 595--621.

\bibitem{WR2012}
Wang W. and Roberts A. J.,
Average and deviation for slow-fast stochastic partial differential equations.
{\em J. Differential Equations} {\bf253} (2012), 1265--1286.



\bibitem{Xi-Xi-Zh-Zh}
Xia P., Xie L., Zhang X. and Zhao G.,
$L^q(L^p)$-theory of stochastic differential equations.
{\it Stochastic Processes Appl.} {\bf130} (2020), 5188--5211.

%\bibitem{Xi-Zh}
%Xie L. and Zhang X., Ergodicity of stochastic differential equations with jumps and singular coefficients.
%{\it  Ann. Inst. Henri Poincar\'e Probab. Stat.} {\bf56} (2020), 175--229.

%\bibitem{Zh0}Zhang X.: Stochastic homeomorphism flows of SDEs with singular drifts and Sobolev diffusion coefficients.
%{\it Electron. J. Probab. \bf 16} (2011), 1096-1116.

%\bibitem{Zh1}Zhang X.: A discretized version of Krylov's estimate and its applications. {\it Electron. J. Probab.}, 24 (2019), no. 131, 1-17.

\bibitem{Zh-Zh1}
Zhang, X. and Zhao, G.,
Singular Brownian diffusion processes.
{\em Commun. Math. Stat.} {\bf6} (2018), 533--581.


\bibitem{Zhao2020}
Zhao G.,
On distribution depend SDEs with singular drifts.
{\em  arXiv:2003.04829v3} (2020).

\bibitem{Zhang2010}
Zhang X.,
Stochastic Volterra equations in Banach spaces and stochastic partial differential equation.
{\em J. Funct. Anal.} {\bf258} (2010), 1361--1425.

\bibitem{Zh2020}
Zhang X., Weak solutions of McKean-Vlasov SDEs with supercritical drifts. To appear in {\em  Commun. Math. Stat.} Preprint version available at 
{\em https://arxiv.org/abs/2010.15330}.

%\bibitem{Zh-Zh2}Zhang X. and Zhao, G.: Stochastic Lagrangian path for Leray solutions of 3D Navier-Stokes
  %equations,   arXiv: 1904.04387.

\bibitem{Zv}
Zvonkin A. K., A transformation of the phase space of a diffusion process that removes the drift. (Russian)
{\it Mat. Sb. (N.S.)} {\bf 93(135)} (1974), 129-149.
\end{thebibliography}
\end{document}